\documentclass[leqno]{amsart}

\usepackage{amssymb,amsfonts}
\usepackage{graphicx}
\usepackage{subfigure}
\usepackage{amsthm}

\usepackage{amsmath}
\usepackage{bm}             
\usepackage[mathscr]{eucal}
\usepackage{color}
\usepackage{cancel}
\usepackage{enumerate}
\usepackage{psfrag}
\usepackage{bbm}
\usepackage{appendix}

\usepackage[colorlinks=true,linkcolor=red,citecolor=blue]{hyperref}

\newcommand{\R}{\mathbb{R}}
\newcommand{\bfR}{\mathbb{R}}
\newcommand{\N}{\mathbb{N}}
\newcommand{\Omv}{\Omega_p}

\newcommand{\Om}{\Omega}

\newcommand{\chfn}{\mathbbm{1}}
\newcommand{\mcA}{\mathcal{A}}
\newcommand{\mcM}{\mathcal{M}}

\newcommand{\mcS}{\mathcal{S}}
\newcommand{\mcD}{\mathcal{D}}
\newcommand{\mcG}{\mathcal{G}}

\newcommand{\mcH}{\mathcal{H}}

\newcommand{\mcK}{\mathcal{K}}
\newcommand{\mcN}{\mathcal{N}}
\newcommand{\mcP}{\mathcal{P}}
\newcommand{\mcPa}{\mathcal{P}^\alpha}
\newcommand{\mcR}{\mathcal{R}}
\newcommand{\mcW}{\mathcal{W}}
\newcommand{\mcX}{\mathcal{X}}
\newcommand{\mcXa}{\mathcal{X}^\alpha}
\newcommand{\mcV}{\mathcal{V}}
\newcommand{\mcY}{\mathcal{Y}}

\newcommand{\bS}{\mcS}
\newcommand{\mcB}{\mathcal{B}}

\newcommand{\alt}[1]{#1}

\newtheorem{theorem}{Theorem}[section]
\newtheorem*{theorem*}{Theorem}

\newtheorem{lemma}[theorem]{Lemma}

\theoremstyle{definition}
\theoremstyle{remark}

\newtheorem{remark}[theorem]{Remark}

\usepackage{fullpage}
\numberwithin{equation}{section}

\title{Modified Scattering of Solutions to the Relativistic Vlasov-Maxwell System Inside the Light Cone}
\author{Stephen Pankavich}
\address{Department of Applied Mathematics and Statistics, Colorado School of Mines, Golden, CO 80401.}
\email{pankavic@mines.edu}

\author{Jonathan Ben-Artzi}
\address{School of Mathematics, Cardiff University, Cardiff, UK.}
\email{Ben-ArtziJ@cardiff.ac.uk}

\date{\today}
\thanks{The first author was supported in part by US National Science Foundation grants DMS-1911145 and DMS-2107938.
The authors have no competing interests to declare that are relevant to the content of this article..
Data sharing is not applicable to this article as no datasets were generated or analyzed during the current study.}

\begin{document}
\maketitle

\begin{abstract}
We consider the relativistic Vlasov-Maxwell system in three dimensions and study the limiting asymptotic behavior as $t \to \infty$ of solutions launched by small, compactly supported initial data.
In particular, we prove
that such solutions scatter to a modification of the free-streaming asymptotic profile. 
More specifically, we show that the spatial average of the particle distribution function converges to a smooth, compactly-supported limit and establish the precise, self-similar asymptotic behavior of the electric and magnetic fields, as well as, the macroscopic densities and their derivatives in terms of this limiting function.
Upon constructing the limiting fields, a modified $L^\infty$ scattering result for the particle distribution function along the associated trajectories of free transport corrected by the limiting Lorentz force is then obtained.
When the \alt{limiting charge density does not vanish}, our estimates are sharp up to a logarithmic correction. 
However, when \alt{this quantity is identically zero in the limit, the limiting current density and fields may also vanish}, which gives rise to decay rates that are faster than those attributed to the dispersive mechanisms in the system.
\end{abstract}

\section{Introduction}

The motion of a collisionless plasma, namely a completely ionized gas that is sufficiently dilute to neglect collisional effects and sufficiently hot to consider relativistic velocities, is often modeled by the relativistic Vlasov-Maxwell system. Within this system charged particles interact with one another via self-consistent electromagnetic forces induced by the charges. In particular, the plasma consists of $N \in \mathbb{N}$ different species, e.g. ions and electrons, each with associated masses $m_\alpha>0$ and charges $e_\alpha\in\bfR$ for $\alpha = 1, ..., N$. Additionally, each particle species is described in phase space by a distribution function $f^\alpha(t,x,p)$, which represents a number density at time $t \geq 0$, position $x \in \bfR^3$, and momentum $p\in \bfR^3$.
The motion of particles is then governed by the Vlasov equation
\begin{equation}
\label{Vlasov}
\tag{RVM$_1$}
\partial_{t}f^\alpha+v_\alpha (p)\cdot\nabla_{x}f^\alpha + e_\alpha \left ( E + \frac{v_\alpha(p)}{c} \times B \right ) \cdot\nabla_{p}f^\alpha=0,
\end{equation}
where $c$ is the speed of light and the relativistic velocity of the $\alpha$th species is given by
\begin{equation}
\label{Rel_vel}
\tag{RVM$_2$}
v_\alpha(p) = \frac{p}{\sqrt{m_\alpha^2 + |p|^2/c^2}}.
\end{equation}
The distribution function gives rise to charge and current densities, defined by
\begin{equation}
\label{rhoj}
\tag{RVM$_3$}
\rho(t,x) =4\pi \sum_{\alpha=1}^N e_\alpha  \int_{\bfR^3} f^\alpha(t,x,p) \,dp \qquad \mathrm{and} \qquad j(t,x) = 4\pi \sum_{\alpha=1}^N e_\alpha \int_{\bfR^3} v_\alpha(p) f^\alpha(t,x,p)  \ dp,
\end{equation}
and these densities generate electromagnetic fields $E(t,x)$ and $B(t,x)$ according to Maxwell's equations
\begin{equation}
\label{Maxwell}
\tag{RVM$_4$}
\begin{aligned}
& \partial_t E - c\nabla \times B = - j, \hspace{.76cm} \nabla \cdot E = \rho\\
& \partial_t B + c\nabla \times E = 0, \hspace{1cm} \nabla \cdot B = 0.
\end{aligned}
\end{equation}
The equations \eqref{Vlasov}-\eqref{Maxwell}, which we will group together and denote by (\hyperref[Vlasov]{RVM}) are then supplemented by the initial conditions $f(0,x,p) = f_0(x,p)$, $E(0,x) = E_0(x)$ and $B(0,x) = B_0(x)$,
which must further satisfy the compatibility conditions
\begin{equation}
\label{compat}
\alt{\nabla \cdot E_0(x) = \rho(0,x), \qquad \nabla \cdot B_0(x) = 0, \qquad \mathrm{and} \qquad \int_{\bfR^3} \rho(0,x) \ dx = 0.}
\end{equation}
For each species, the particle number is conserved, namely
$$ \iint_{\bfR^3\times\bfR^3} f^\alpha(t,x,p) \ dpdx =  \iint_{\bfR^3\times\bfR^3} f^\alpha_0(x,p) \ dpdx =: \mcM^\alpha$$
for all $t \geq 0$, with the overall net charge of the system given by
$$\mcM_{\mathrm{net}} = \sum_{\alpha = 1}^N e_\alpha \mcM^\alpha \alt{ = 0} $$
\alt{due to the global neutrality assumption within \eqref{compat}.}
Additionally, as our methods will require precise estimates on the growth of the characteristics associated to (\hyperref[Vlasov]{RVM}), we define the functions $\mcXa(t, \tau, x, p)$ and $\mcPa(t, \tau, x, p)$ for all $t, \tau \geq 0$ and $x,p \in \bfR^3$ as solutions of the system
\begin{equation}
\label{char}
\left \{
\begin{aligned}
&\dot{\mcXa}(t, \tau, x, p)= v_\alpha \left ( \mcPa(t, \tau, x, p) \right )\\
&\dot{\mcPa}(t, \tau, x, p)=  K^\alpha \left (t,\mcXa(t, \tau, x, p), \mcPa(t, \tau, x, p) \right)
\end{aligned}
\right.
\end{equation}
with initial conditions
$\mcXa(\tau, \tau, x, p) = x$ and
$\mcPa(\tau, \tau, x, p) = p,$
where the Lorentz force on the $\alpha$th species is defined by
$$K^\alpha(t,x,p) = e_\alpha \left (E(t, x) +\frac{v_\alpha(p)}{c}\times B(t, x) \right ).$$

Though the global-in-time existence of classical solutions to (\hyperref[Vlasov]{RVM}) for arbitrary data remains a crucial unsolved problem in the field, a variety of small data solutions are known to exist, ranging from more classical results \cite{GSchNN, GS, SchRVM} to newer theorems \cite{Bigorgne, Wang} that utilize Fourier or vector field methods to remove the assumption that the particle distribution functions be compactly supported. We also mention \cite{WY}, which utilized methods similar to those of Glassey and Strauss \cite{GS} to establish a global existence result for small initial data without compact support, but without the requirement that the initial electromagnetic fields be small.
We further note that some recent progress has been made on the global existence problem for arbitrary size data, as well, and in particular the development of new continuation criteria that ensure the extension of a local solution \cite{Kunze, LS, LS2, Patel}.
For additional background, we refer the reader to \cite{Glassey} as a general reference concerning (\hyperref[Vlasov]{RVM}) and associated kinetic equations.\\

In addition to a lack of global-in-time well-posedness of large data, classical solutions to (\hyperref[Vlasov]{RVM}), the time asymptotic behavior of small data solutions is also not completely understood. Partial results concerning the asymptotic growth or decay of quantities for (\hyperref[Vlasov]{RVM}) or its classical limit, the Vlasov-Poisson system, are known in some situations, including small data \cite{BD, Bigorgne, HRV, Ionescu, Smulevici}, monocharged and spherically-symmetric data \cite{BCP1, Horst, Pankavich2020}, and lower-dimensional settings \cite{BKR, BMP, GPS, GPS2, GPS4, Sch}.
These results generally provide either time asymptotic growth estimates of characteristics or decay estimates of the electric field and charge density.
One typically expects that, inside the light cone, the fields and macroscopic densities tend to zero as $t \to \infty$ like $t^{-2}$ and $t^{-3}$, respectively, for all smooth solutions of (\hyperref[Vlasov]{RVM}) due to the dispersive properties induced within the system by the transport operator $\partial_t + v_\alpha(p) \cdot \nabla_x$. In fact, it is known that the Cauchy problem does not possess smooth steady states (cf., \cite{GPS5}).
%
That being said, it remains a longstanding open problem to demonstrate that smooth solutions of (\hyperref[Vlasov]{RVM})  satisfy these decay properties or scatter to a profile along the trajectories generated by the (possibly modified) free transport operator as $t \to \infty$. 
Hence, the goal of the current work is to establish the precise large-time behavior of solutions to (\hyperref[Vlasov]{RVM}) and construct small data solutions that display exactly this asymptotic dynamic.\\


%

Prior to submitting the current paper, we were made aware of \alt{a recent and} similar work in this direction due to Bigorgne \cite{Bigorgne}, concerning the asymptotic behavior of small data solutions to (\hyperref[Vlasov]{RVM}). 
Though our results are similar, we only assume $C^2$ data for the particle distribution, we demonstrate improved decay rates in the case of \alt{a limiting charge density that vanishes}, and we do not employ vector field methods to achieve the theorems that follow.
\alt{Instead, we use more classical \emph{a priori} estimates (in the spirit of Glassey-Strauss \cite{GS})}, which significantly shortens the proof and makes the result more accessible to the PDE community, especially for those
who are unfamiliar with some of the geometric and vector field methods utilized in \cite{Bigorgne}.\\

For brevity, we will use the small data solutions previously constructed by Glassey and Strauss as a starting point; however, our methods can be applied to other solutions with compact momentum support, such as the nearly-neutral small-data solutions constructed by Glassey and Schaeffer \cite{GSchNN}.
Additionally, it may be possible to remove the support assumption completely, as some recent studies have done just this for similar kinetic equations \cite{BigorgneRVM, Pausader, Ionescu}.
It must be noted, however, that removal of the compact momentum support assumption for the distribution function would allow for some particles to travel at  speeds arbitrarily close to the speed of light and thus, enable resonant particles to travel along the boundary of the light cone.
The additional decay of the electromagnetic fields and their derivatives inside of the light cone may be lost, and this would create a significant challenge to establishing sharp estimates of the limiting behavior of solutions.\\

As the small data solutions we will construct remain bounded on any finite time interval, we are essentially concerned only with large time estimates, and thus we use the notation
$$A(t) \lesssim B(t)$$
to represent the statement that there is $C > 0$, independent of $t$,
such that
$A(t) \leq CB(t)$
for $t$ sufficiently large.
When a specific constant is desired, for instance when there is $\gamma >0$  such that $|x| \leq \gamma t$ for $t$ sufficiently large, we will include the constant in this notation, namely
\alt{\begin{equation}
\label{xlesssim}
|x| \lesssim \gamma t.
\end{equation}}
Additionally, $C$ will denote a positive constant (independent of the solution) that may depend upon initial data and can change from line to line.
As $B$ represents a magnetic field, throughout we will denote the balls of radius $r >0$ (open and closed, respectively)  by 
$$\Gamma_r =\{ u \in \bfR^3 : |u| < r \} \qquad \mathrm{and} \qquad \overline{\Gamma}_r = \{u \in \bfR^3 : |u| \leq r\}.$$


\subsection{Main Results}
For $t \geq 0$ and $\alpha = 1, ..., N$ define the support of $f^\alpha(t)$ by
$$\bS_{f^\alpha}(t) = \overline{\left \{ (x,p) \in \R^6 : f^\alpha(t,x,p) \neq 0 \right \}}.$$
Next, for every $p \in \bfR^3$ denote the gradient of the velocity function by
$$\mathbb{A}_\alpha(p) := \nabla v_\alpha(p)$$
and its associated inverse determinant by
$$\mcD_\alpha(p) = \left | \det \mathbb{A}_\alpha(p) \right |^{-1}$$
for $\alpha = 1, ..., N$.
Furthermore, for every every $q \in \Gamma_1$ let 
$$v_\alpha^{-1}(q) = \frac{m_\alpha q}{\sqrt{1 - |q|^2/c^2}},$$
be the inverse function of $v_\alpha(p)$
and  
$$\mathbb{B}_\alpha(q) := \nabla v_\alpha^{-1}(q)$$
be its associated gradient for $\alpha = 1, ..., N$.\\

Our main results can now be stated precisely, and, for brevity, we normalize the speed of light $c=1$ throughout.
First, we establish a refined global existence theorem for small initial data that augments the previous result of \cite{GS} by obtaining $C^2$ solutions from $C^2$ initial data, rather than $C^1$ solutions from $C^1$ data.
The improved regularity of solutions will be necessary to construct suitable limits for the electric and magnetic fields in the sequel.

\begin{theorem}
\label{T0}
For any $L > 0$ there exist $\epsilon_0 > 0$ and $\beta > 0$ with the following property.
Let $f^\alpha_0 \in C^2$, $\alpha=1,\dots,N$, be non-negative functions supported on $\overline{\Gamma}_L \times \overline{\Gamma}_L$.
Let $E_0, B_0 \in C^3$ be supported on $\overline{\Gamma}_L$ satisfying the compatibility conditions \eqref{compat}.
If the initial data satisfy
$$\sum_{\alpha = 1}^N \Vert f^\alpha_0 \Vert_{C^2} + \Vert E_0 \Vert_{C^3} + \Vert B_0 \Vert_{C^3} \leq \epsilon_0,$$
then there exists a unique classical solution of (\hyperref[Vlasov]{RVM}) for all $x, p \in \bfR^3$ and $t \geq 0$ such that
$$f^\alpha(t,x,p) = 0 \qquad \mathrm{for} \qquad |p| \geq \beta$$
for all $\alpha = 1, ..., N$, $t\geq 0$, and $x \in \bfR^3$.
Furthermore, 
$|E(t,x)| + |B(t,x)| = 0$ for $|x| > t + L$
and we have the estimates
$$s - |\mcX^\alpha(s,t,x,p)| + 2L \geq \frac{s + L}{2 (1+\beta^2)}$$
for all $s, t \geq 0$,  $x,p \in \bfR^3$
and

$$|E(t,x)| + |B(t,x)| \leq \frac{C\alt{\epsilon_0 }}{(t+1)\left (t - |x | + 2L \right )},$$
$$|\nabla_x E(t,x)| + |\nabla_x B(t,x)| \leq \frac{C\alt{\epsilon_0 }\ln(t+2)}{(t+1)\left (t - |x | + 2L \right )^2},$$
$$|\nabla^2_x E(t,x)| + |\nabla^2_x B(t,x)| \leq \frac{C\alt{\epsilon_0 }\ln(t+2)}{(t+1)\left (t - |x | + 2L \right )^3},$$
$$\Vert \nabla_x f^\alpha(t) \Vert_\infty \leq C \qquad \mathrm{and} \qquad \Vert \nabla_p f^\alpha(t) \Vert_\infty \leq C(1+t),$$
and
$$|\rho(t,x)|  + | j(t,x)|  \leq C\alt{\epsilon_0 }(1+t)^{-3}$$
for some $C>0$ and all $t \geq 0$, $x \in \bfR^3$.
\end{theorem}

Next, we establish precise estimates on the large time asymptotic behavior of these small data solutions.

\begin{theorem}
\label{T1}
Under the conditions of Theorem \ref{T0}, the small data solutions satisfy
\begin{enumerate}[(a)]
\item 
For every $\alpha = 1, ..., N$, $\tau \geq 0$ and $(x,p) \in \bS_{f^\alpha}(\tau)$ the limiting function $\mcPa_\infty$ defined by
$$\mcPa_\infty(\tau, x, p) :=  \lim_{t \to \infty} \mcPa(t, \tau, x, p)$$ 
exists and is both $C^2$ and contained in $\Gamma_\beta$.
Additionally, for $\tau \geq 0$ and $(x,p) \in \bS_{f^\alpha}(\tau)$,
$$\left | \mcPa(t, \tau, x, p) - \mcPa_\infty(\tau, x, p)  \right | \lesssim t^{-1}.$$

\item 
For every $\alpha = 1, ..., N$ define
$$\Omv^\alpha = \left \{ \mcPa_\infty(0, x, p) : (x, p) \in \bS_{f^\alpha}(0) \right \}.$$
Then, there exist $F^\alpha_\infty \in C_c^2(\bfR^3)$ supported on $\Om_p^\alpha$ such that the spatial average
$$F^\alpha(t,p) = \int f^\alpha(t,x, p) \ dx$$
satisfies $F^\alpha(t,p) \to F^\alpha_\infty(p)$ uniformly as $t \to \infty$, namely
$$\| F^\alpha(t) - F^\alpha_\infty \|_\infty \lesssim t^{-1}\ln^{5}(t),$$
and for every $\alpha = 1, ..., N$,
$$\int  F^\alpha_\infty(p) \ dp = \mcM^\alpha.$$

\item Due to the compact support of each $F_\infty^\alpha$, define the limiting charge density for $q \in \Gamma_1$ by
$$\rho_\infty(q) = \sum_{\alpha=1}^N e_\alpha \mcD_\alpha\left ( v_\alpha^{-1}(q) \right )F^\alpha_\infty\left ( v_\alpha^{-1}(q) \right )$$
and smoothly extend $\rho_\infty(q) = 0$ for $q \in \bfR^3 \setminus \Gamma_1$.
Similarly, define the limiting current density by
$$ j_\infty(q) = q \rho_\infty(q)$$
for $q \in \bfR^3$.
Then, the charge and current densities have the self-similar 
asymptotic profiles 
\begin{align*}
\sup_{x \in \bfR^3}   \left | t^3 \rho(t,x) - \rho_\infty \left (\frac{x}{t} \right) \right | & \lesssim t^{-1} \ln^{6}(t),\\
\sup_{x \in \bfR^3}  \left | t^3 j(t,x) - j_\infty \left (\frac{x}{t} \right) \right | & \lesssim t^{-1}\ln^{6}(t),
\end{align*}
and $\rho_\infty(q)$ satisfies
\begin{equation}
\label{Pinfmass}
\int \rho_\infty(q) \ dq = \mcM_{\mathrm{net}} \alt{= 0}.
\end{equation}
Furthermore, for every $i = 1,2,3$ the derivatives satisfy
\begin{align*}
\sup_{x \in \bfR^3} \left | t^4 \partial_{x_i} \rho(t,x) - \partial_{q_i}\rho_\infty \left (\frac{x}{t} \right) \right | & \lesssim t^{-1} \ln^{8}(t),\\
\sup_{x \in \bfR^3}  \left | t^4 \partial_{x_i}j(t,x) -  \partial_{q_i}j_\infty\left (\frac{x}{t} \right) \right | & \lesssim t^{-1}\ln^{8}(t),\\
\sup_{x \in \bfR^3}  \left | t^4 \partial_t j^i(t,x) + \left [3j^i_\infty\left (\frac{x}{t} \right) +  \frac{x}{t} \cdot \nabla_qj^i_\infty\left ( \frac{x}{t} \right) \right ]  \right | & \lesssim t^{-1}\ln^{8}(t).
\end{align*}

\item Let $\gamma := \max \left \{ \frac{1}{2}, \frac{2\beta}{\sqrt{1 + 4\beta^2}} \right \} < 1$, and
define the linear operator 
$$\mathcal{L} u := \sum_{i,j=1}^3 \left ( q_i q_j - \delta_{ij} \right ) \partial_{q_i q_j} u + 6q \cdot \nabla_q u + 6u.$$
Further define for any $i = 1,2,3$ the functions $E_\infty^i(q)$ and $B_\infty^i(q)$ to be the unique $C^2$ solutions (cf. \cite{GT}) of the uniformly elliptic boundary-value problems
$$ \mathcal{L} E_\infty^i(q) = -\partial_{q_i}\rho_\infty(q) + 3j^i_\infty(q) +  q \cdot \nabla_qj^i_\infty(q)$$
and
$$ \mathcal{L}B_\infty^i(q) =  \left (\nabla_q \times j_\infty \right)^i (q),$$
respectively, for $|q| < \gamma$ with the boundary conditions
$$E_\infty^i(q) = B_\infty^i(q) = 0 \quad \mathrm{for} \ |q| = \gamma.$$
Then, $E$ and $B$ have the self-similar asymptotic profiles
\begin{align*}
\sup_{|x| \lesssim \gamma t} \left | t^2 E(t,x) - E_\infty \left (\frac{x}{t} \right ) \right | & \lesssim t^{-1}\ln^{8}(t),\\
\sup_{|x| \lesssim \gamma t} \left | t^2 B(t,x) - B_\infty \left (\frac{x}{t} \right ) \right | & \lesssim t^{-1}\ln^{8}(t).
\end{align*}
Additionally, the Lorentz force of the $\alpha$th species satisfies
$$\sup_{\substack{|x| \lesssim \ln(t)\\ |p| \leq \beta}}  \left | t^2 K^\alpha(t, x + v_\alpha(p)t , p) - K^\alpha_\infty \left ( p \right ) \right |  \lesssim t^{-1}\ln^{8}(t)$$
for all $\alpha = 1, ..., N$
where 
$$K^\alpha_\infty(p) = e_\alpha \bigg [E_\infty(v_\alpha(p)) + v_\alpha(p) \times B_\infty(v_\alpha(p) ) \bigg].$$


\item
For every $\alpha = 1, ..., N$
there is $f^\alpha_\infty \in C(\bfR^6)$ 
such that
$$f^\alpha \biggl (t,x +v_\alpha(p)t - \ln(t) \mathbb{A}_\alpha(p) K^\alpha_\infty(p),p \biggr) \to f^\alpha_\infty(x,p)$$
uniformly
as $t \to \infty$, namely we have the convergence estimate
$$\sup_{(x,p) \in \bfR^6} \left | f^\alpha \biggl (t,x +v_\alpha(p)t - \ln(t) \mathbb{A}_\alpha(p) K^\alpha_\infty(p), p \biggr) - f^\alpha_\infty(x,p) \right |  \lesssim t^{-1}\ln^{8}(t).$$
\end{enumerate}
\end{theorem}

%

\alt{In the case that $\rho_\infty \not\equiv 0$, these estimates are sharp up to a correction in the logarithmic powers of the error terms. However, because the plasma is neutral, i.e. $\mcM_{\mathrm{net}} = 0$, it is possible that the limiting charge density $\rho_\infty$ (and hence $j_\infty$, $E_\infty$, and $B_\infty$) is identically zero, which implies stronger decay of these quantities.}

\begin{theorem}
\label{T2}
\alt{If $\rho_\infty \equiv 0$}, then the asymptotic behavior described above is improved in the following manner:

\begin{enumerate}[(a)]
\item 
For any $\alpha = 1, ..., N$, $\tau \geq 0$ and $(x,p) \in \bS_{f^\alpha}(\tau)$, we have
$$\left | \mcP^\alpha(t, \tau, x, p) - \mcP_\infty^\alpha(\tau, x, p)  \right | \lesssim t^{-2}\ln^8(t).$$

\item We have the faster decay estimates
\begin{align*}
& \| \rho(t) \|_\infty + \| j(t) \|_\infty \lesssim t^{-4}\ln^6(t),\\
& \| \nabla_x \rho(t) \|_\infty + \| \nabla_x j(t) \|_\infty + \| \partial_t j(t) \|_\infty \lesssim t^{-5}\ln^8(t), & 
\end{align*}
as well as
$$\sup_{|x| \lesssim \gamma t}  \left ( \left |E(t, x) \right | + \left |B(t,x) \right | \right ) \lesssim t^{-3}\ln^8(t).$$

\end{enumerate}
\end{theorem}

\begin{remark}
\label{rem:1}
Though we do not prove such a refinement here, we note that the improved decay estimates on $\rho$, $j$, and their derivatives in Theorem \ref{T2}, also induce faster decay of the derivatives of the electric and magnetic fields.
Indeed, following the iteration argument of Glassey and Strauss \cite{GS}, and using these improved decay rates \alt{with \eqref{xlesssim}} one can establish the estimate
$$\sup_{|x| \lesssim \gamma t} \left ( \left |\nabla_x E(t, x) \right | + \left |\nabla_x B(t,x) \right | \right ) \lesssim t^{-4}\ln^8(t).$$
This faster rate uniformly bounds derivatives of the translated distribution function via Lemma \ref{Dvg} (as $\mcK_1(t) \lesssim t^{-4}\ln^8(t)$; see Section \ref{sec:prelim}) and induces faster convergence of the spatial averages via Lemma \ref{Funif}.
These results can then be iteratively utilized with our tools \alt{(see also \cite{Pankavich2021})} to improve the rates on the charge and current density, as well as other quantities in the system, and ultimately obtain the subsequent estimates
\begin{align*}
\left | \mcP^\alpha(t, \tau, x, p) - \mcP_\infty^\alpha(\tau, x, p)  \right | & \lesssim t^{-2},\\
\| F^\alpha(t) - F^\alpha_\infty \|_\infty & \lesssim t^{-2}, \\
\| \rho(t) \|_\infty + \| j(t) \|_\infty & \lesssim t^{-4},\\
\| \nabla_x \rho(t) \|_\infty + \| \nabla_x j(t) \|_\infty + \| \partial_t j(t) \|_\infty & \lesssim t^{-5},
\end{align*}
as well as
$$\sup_{|x| \lesssim \gamma t} \left ( \left |E(t, x) \right | + \left |B(t,x) \right | \right ) \lesssim t^{-3},$$
and
$$\sup_{|x| \lesssim \gamma t} \left (  \left |\nabla_x E(t, x) \right | + \left |\nabla_x B(t,x) \right | \right ) \lesssim t^{-4}.$$
Finally, these faster rates remove the need to modify the scattering trajectories and yield the improved convergence of the particle distributions, namely
$$\sup_{(x,p) \in \bfR^6} \left | f^\alpha(t,x +v_\alpha(p)t, p) - f^\alpha_\infty(x,p) \right |  \lesssim t^{-1}$$
for any $\alpha = 1, ..., N$.
%
As previously mentioned, we do not present a proof of these improvements here, and as such, do not include these statements in the above theorem.
\end{remark}



\begin{remark}
Analogues of Theorems \ref{T1} and \ref{T2} can likely be shown for small data solutions of the relativistic Vlasov-Poisson system using a combination of our methods to allow for relativistic velocity corrections and the simpler tools created in \cite{Pankavich2021} to estimate an electrostatic, rather than electromagnetic, field. Additionally, spherically symmetric solutions of the classical and relativistic Vlasov-Poisson system, as considered in \cite{Pankavich2020}, should also display these asymptotic dynamics.
\end{remark}

\subsection{Idea of the Proof}
To establish the scattering result, we will show that the time derivative of the distribution function decays faster than $t^{-1}$ in $L^\infty$. 
Dropping the $\alpha$ notation, as well as all masses and charges, we note that this cannot hold if the particle distribution is evaluated along $(x,p)$, as the transport term $v(p) \cdot \nabla_x f$  cannot be expected to decay in time.
However, evaluating this function along the translated curves in phase space $(x + v(p)t, p)$ by defining
$$g(t,x,p) = f(t, x+v(p)t, p)$$
transforms the Vlasov equation for $g$ into
\begin{equation}
\label{VMg}
\partial_{t}g -  t \mathbb{A}(p)K(t,x+v(p)t, p)\cdot \nabla_{x}g+  K(t,x+v(p)t, p) \cdot\nabla_{p}g=0.
\end{equation}
Because the spatial support of $g$ stays strictly inside the light cone, one generally expects the Lorentz force to decay like $t^{-2}$ on this set (see Theorem \ref{GS}).
As we will demonstrate later (Lemma \ref{Dvg}), the $p$-derivatives of $g$ satisfy $\| \nabla_p g(t) \|_\infty \lesssim \ln^2(t)$ thereby making the last term in \eqref{VMg} integrable in time.
Additionally, $\| \nabla_x g(t) \|_\infty \sim 1$, which means that the second term in the transformed Vlasov equation \eqref{VMg} will decay at best like $t^{-1}$.
Hence, in order to induce a faster decay rate, the trajectories must be augmented by a limiting Lorentz force that is generated by the limiting electric and magnetic fields.
In particular, we further define
$$h(t,x,p) = g(t, x - \ln(t) \mathbb{A}(p)K_\infty(p), p)$$
so that
$$\partial_{t}h =  t^{-1} \mathbb{A}(p) \left ( t^2K(t,\mcW(t,x,p), p) - K_\infty(p)\right )\cdot \nabla_{x}g -  K(t,\mcW(t,x,p), p) \cdot\nabla_{p}g$$
where
$$\mcW(t,x,p) = x+v(p)t- \ln(t) \mathbb{A}(p)K_\infty(p).$$
Upon establishing an estimate on the convergence rate of the Lorentz force inside the light cone so that $t^2 K(t,\mcW(t,x,p), p)$ tends to $K_\infty(p)$ at some algebraic order, say $t^{-\epsilon}$ for some $\epsilon > 0$, we will arrive at 
$$\| \partial_t h(t) \|_\infty \lesssim t^{-1-\epsilon},$$
which is integrable and guarantees the existence of a limiting particle distribution along such trajectories in phase space.
Of course, this requires an understanding of the higher-order behavior of the electromagnetic fields and their limits along these same trajectories. 
In turn, one must identify the higher-order convergence of the macroscopic densities and their derivatives in order to establish the limiting field behavior.
Moreover, a crucial result (Lemma \ref{LWave}) is needed to pass from the asymptotic behavior of the densities and their derivatives to the fields, which satisfy inhomogeneous wave equations with derivatives of these density terms as sources.
Within \cite{Pankavich2021}, the electrostatic scattering problem was considered, and this merely required one to deduce the asymptotic behavior of the electric field as a solution to Poisson's equation, given the known asymptotic behavior of the charge density.
Herein, the full electromagnetic description of the system is imposed, and this requires the significantly more challenging task of preserving the asymptotic behavior of the source terms in passing to the fields via inversion of wave equations.\\

In the next section, we establish some preliminary results that will be used throughout the paper.
Section \ref{sec:lower} then contains estimates on the electric and magnetic fields, characteristics, and derivatives of the translated particle distribution function, as well as, the convergence of the spatial average and charge and current densities.
Section \ref{sec:higher} assumes an estimate on second derivatives of the fields that will be provided by Theorem \ref{T0} and establishes the asymptotic behavior of derivatives of the densities and the spatial average. 
Estimates that provide the analogous asymptotic behavior for the electric and magnetic fields are contained within Section \ref{sect:fieldconv}, as is the modified scattering result for the distribution functions. The proofs of Theorems \ref{T0}, \ref{T1}, and \ref{T2} are provided in Section \ref{sec:thmpfs}. \alt{Finally, in Section \ref{sec:other-models} we briefly discuss how these results can be adapted to small data solutions of the non-relativistic Vlasov-Maxwell system.}\\

Though our results pertain directly to multiple species, for the majority of this article we will derive estimates for the single species system with charge $e_\alpha=1$. Hence, the number of species $N$, particle masses $m_\alpha$, particle charges $e_\alpha$, and the factors of $4\pi$ within (\hyperref[Vlasov]{RVM}) are all normalized to $1$.
We emphasize that these assumptions are not needed to achieve the results stated above, but they do significantly simplify the presentation.
In Section \ref{sec:thmpfs} we will then use the established lemmas to distinguish between the \alt{non-vanishing charge density limit ($\rho_\infty \not \equiv 0$) and vanishing charge density limit ($\rho_\infty \equiv 0$)} cases in the proofs of Theorems \ref{T1} and \ref{T2}.
With this, the relativistic velocity function defined by \eqref{Rel_vel} simplifies to
$$v(p) = \frac{p}{\sqrt{1+ |p|^2}}$$
and its inverse becomes
$$v^{-1}(q) = \frac{q}{\sqrt{1 - |q|^2}}.$$
Additionally, the Lorentz force becomes
$$K(t,x,p) = E(t,x) + v(p) \times B(t,x).$$
Upon inserting these simplifications, the system (\hyperref[Vlasov]{RVM}) reduces to 
\begin{equation}
\tag{RVM$_\text{Single}$}
\label{RVM}
\left \{ \begin{aligned}
& \partial_{t}f+v (p)\cdot\nabla_{x}f + K(t,x,p) \cdot\nabla_{p}f=0,\\
& \rho(t,x) = \int f(t,x,p) \,dp, \qquad j(t,x) =\int v(p) f(t,x,p)  \ dp,\\
& \partial_t E - \nabla \times B = - j,\hspace{1.3cm}\nabla \cdot E = \rho,\\
& \partial_t B + \nabla \times E = 0,\hspace{1.55cm} \nabla \cdot B = 0.
\end{aligned} \right.
\end{equation}
Many other previously-defined quantities then simplify and are now independent of $\alpha$, including $\bS_{f}(t), \mathbb{A}(p), \mcD(p)$, and $\mathbb{B}(q)$.
In this case the particle number \alt{is denoted by removing any superscripts}, namely as $\mcM$.
When useful, we will also use the notation $p_0 = \sqrt{1 + |p|^2}$ to denote the particle rest momentum.

\section{Preliminary Lemmas}
\label{sec:prelim}

First, we state the small data theorem of \cite{GS}, which provides the foundational structure of solutions from which we will build. In particular, we state the result for $N$ species, as in the original article.

\begin{theorem}[Glassey-Strauss \cite{GS}]
\label{GS}
For any $L > 0$ there exist $\epsilon_0 > 0$ and $\beta > 0$ with the following property.
Let $f^\alpha_0 \in C^1$ be non-negative functions supported on $\{ |x| \leq L\} \times \{|p| \leq L \}$.
Let $E_0, B_0 \in C^2$ be supported on $\{|x| \leq L\}$ satisfying the compatibility conditions \eqref{compat}.
If the initial data satisfy
$$\sum_{\alpha = 1}^N \Vert f^\alpha_0 \Vert_{C^1} + \Vert E_0 \Vert_{C^2} + \Vert B_0 \Vert_{C^2} \leq \epsilon_0,$$
then there exists a unique classical solution of (\hyperref[Vlasov]{RVM}) for all $x, p \in \bfR^3$ and $t \geq 0$ such that
$$f^\alpha(t,x,p) = 0 \qquad \mathrm{for} \qquad |p| \geq \beta$$
for all $\alpha = 1, ..., N$, $t\geq 0$, and $x \in \bfR^3$.
Furthermore, 
$|E(t,x)| + |B(t,x)| = 0$ for $|x| > t + L$
and we have the estimates
$$s - |\mcX^\alpha(s,t,x,p)| + 2L \geq \frac{s + L}{2 (1+\beta^2)}$$
for all $s, t \geq 0$, and $x,p \in \bfR^3$
with

$$|E(t,x)| + |B(t,x)| \leq \frac{C}{(t+1)\left (t - |x | + 2L \right )},$$
$$|\nabla_x E(t,x)| + |\nabla_x B(t,x)| \leq \frac{C\ln(t+2)}{(t+1)\left (t - |x | + 2L \right )^2},$$
$$\Vert \nabla_x f^\alpha(t) \Vert_\infty \leq C \qquad \mathrm{and} \qquad \Vert \nabla_p f^\alpha(t) \Vert_\infty \leq C(1+t),$$
and
$$|\rho(t,x)|  + | j(t,x)|  \leq C(1+t)^{-3}$$
for some $C>0$ and all $t \geq 0$, $x \in \bfR^3$.
\end{theorem}
As the small data solutions that we construct satisfy the hypotheses of the Glassey-Strauss small data theorem, we invoke their result to establish the existence of $C^1$ solutions and derive new estimates that will be critical in the construction of $C^2$ solutions.
Hence, throughout this section, we work with the solutions guaranteed by Theorem \ref{GS} and derive the large time asymptotic behavior of quantities in the system.

We first establish some basic properties of the function $v(p)$ transforming momenta into relativistic velocities and its inverse $v^{-1}(q)$ which takes velocities to momenta (see Figure \ref{fig:p-q}).

	\begin{figure}
	\includegraphics[width=8cm]{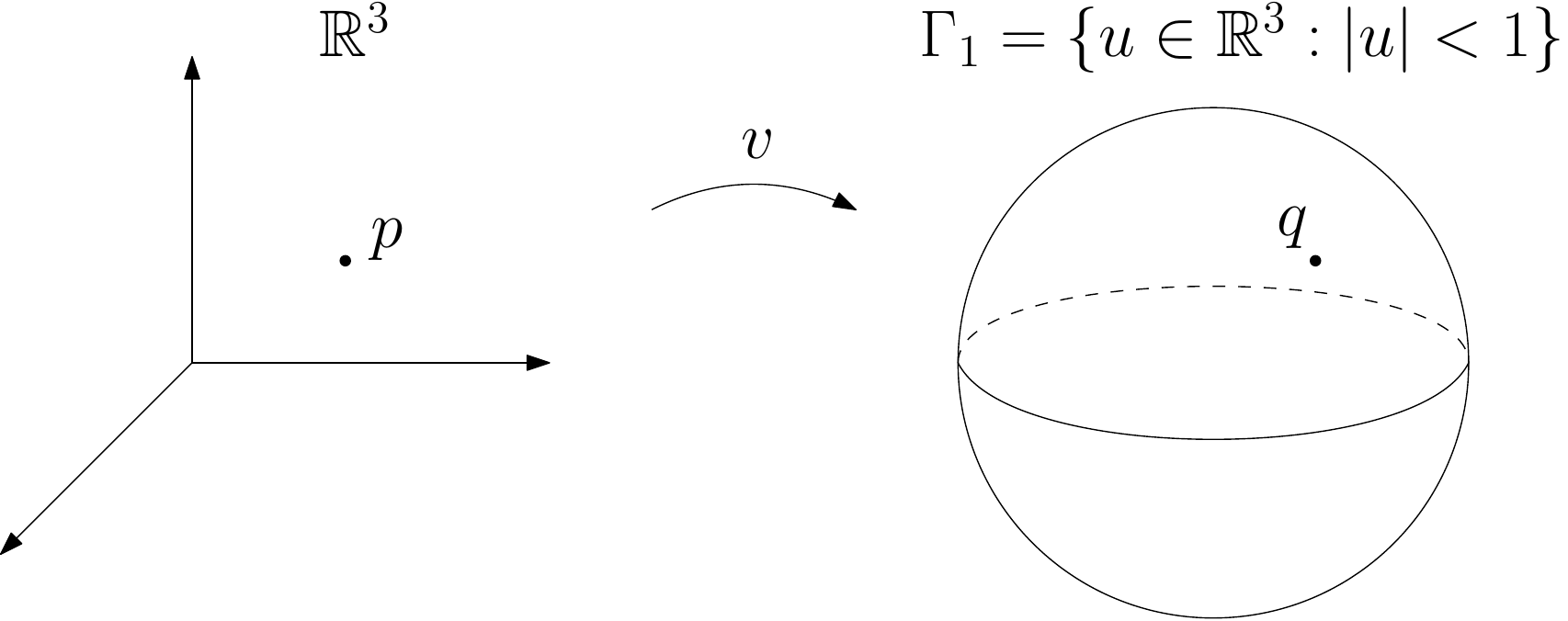}
	\caption{Throughout  this article, $p\in\R^3$ shall denote the momentum variable and $q=v(p)=p/p_0\in\Gamma_1$ shall denote the associated relativistic velocity.}\label{fig:p-q}
	\end{figure}

\begin{lemma}
\label{lem:v}
%
The symmetric, matrix-valued function $\mathbb{A}: \bfR^3 \to \bfR^{3 \times 3}$ defined by $\mathbb{A}(p) = \nabla v(p)$, so that its entrywise representation is
$$\mathbb{A}_{ij}(p) = p_0^{-3} \left [ (1 + |p|^2) \delta_{ij} - p_i p_j \right ],$$
satisfies
$$ |\mathbb{A}_{ij}(p)| \leq p_0^{-1} \leq 1$$
for all $p \in \bfR^3$ and
$$ \left |\det  \left (\mathbb{A}(p) \right ) \right | = p_0^{-5}.$$
Furthermore, 
the symmetric, matrix-valued function $\mathbb{B}: \Gamma_1 \to \bfR^{3 \times 3}$ defined by $\mathbb{B}(q) = \nabla v^{-1}(q)$, so that its associated entrywise representation is
$$\mathbb{B}_{ij}(q) = \left ( 1 - |q|^2 \right )^{-\frac{3}{2}} \left [ (1 - |q|^2) \delta_{ij} + q_i q_j \right ],$$
is  bounded on compact subsets of $\Gamma_1$. 
Additionally, we have the identity
\begin{equation}
\label{AB}
\sum_{j=1}^3 \mathbb{B}_{ij}(v(p)) \mathbb{A}_{jk}(p) = \delta_{ik}
\end{equation}
for all $i,k = 1, 2, 3$ and every $p\in\bfR^3$.
Because it will appear frequently in later estimates, we further define the function
\begin{equation}
\label{Ddef}
\mcD(p) := p_0^5, 
\end{equation}
which represents the inverse determinant of $\mathbb{A}(p)$.
Finally, the functions $v(p)$, $v^{-1}(q)$, $\mcD(p)$ and all of their derivatives are $C^\infty$ and bounded on any compact subset of their respective domains.

\end{lemma}
\begin{proof}
The proof consists of straightforward calculations and is therefore omitted.
\end{proof}

\begin{remark}
Because of the smoothness and boundedness of the velocity map $v(p)$ and its inverse, we will often neglect their contribution to future estimates, as well as those of their derivatives.
\end{remark}
In order to establish decay of the electric and magnetic fields and their derivatives both along characteristics and along translated curves in phase space, we define 
\begin{align*}
\mcK_\ell(t) & := \sup_{\substack{\tau \in [0,t]\\ (x,p) \in \bS_f(\tau)} } \biggl  ( \left |\nabla^\ell_x E(t, \mcX(t, \tau, x, p) ) \right | + \left |\nabla^\ell_xB(t, \mcX(t, \tau, x, p)  )\right | \biggr )\\
& \quad + \sup_{\substack{|x| \lesssim \ln(t)\\ |p| \leq \beta} } \biggl (\left |\nabla^\ell_xE(t, x +v(p)t) \right | + \left |\nabla^\ell_xB(t,x +v(p)t) \right | \biggr ),
\end{align*}
for $\ell = 0, 1, 2$.
%
The next result guarantees that the fields and their derivatives decay quickly on the spatial support of the particle distribution function, which stays firmly inside of the light cone.
The growth of the spatial support of $f$ will also be crucial to proving both the small data theorem and the large time behavior.
As $f$ has compact support, we will estimate only the growth of those characteristics along which $f \neq 0$.
\begin{lemma}
\label{LPrelim}
Define $\zeta = \frac{\beta}{\sqrt{1 + \beta^2}}  $ and  $\gamma = \max\left \{\frac{1}{2}, \frac{2\beta}{\sqrt{1 + 4\beta^2}} \right \} $
where $\beta > 0$ is given by Theorem \ref{GS}, so that $0<\zeta<\gamma<1$.
Then, we have
$$ |\mcX(t,\tau,x,p)| \lesssim \gamma t$$
for any $\tau\in [0,t]$, $(x, p) \in \mcS_f(\tau)$,
and
$$ |\mcX(t,\tau,x,p)| \leq L + \zeta t$$
for all $t \geq 0, \tau\in [0,t]$, $(x, p) \in \mcS_f(\tau)$.
Furthermore, for $|x| \lesssim \ln(t)$ and $|p| \leq \beta$
we have 
$$t - |x + v(p)t| + 2L \gtrsim t.$$
Finally, the fields and their derivatives satisfy
$$\mcK_0(t) \lesssim t^{-2}$$
and
$$\mcK_1(t) \lesssim t^{-3} \ln(t).$$
\end{lemma}
\begin{proof}
To begin, the first results are obtained using the compact momentum support of $f(t)$ and the characteristic equations, namely 
$$|\mcX(t,\tau,x,p)|  \leq |x| + \int_0^t v \left ( \mcP(s,\tau,x,p) \right ) ds \leq L + \frac{\beta}{\sqrt{1 + \beta^2}} t = L + \zeta t \lesssim \frac{2\beta}{\sqrt{1 + 4\beta^2}}t \leq \gamma t.$$
Next, we establish the spatial inequality on translated curves.
The growth assumptions of $x$ and $p$ imply
$$|x + v(p)t| \lesssim \ln(t) + \frac{\beta}{\sqrt{1+ \beta^2}}t \lesssim \gamma t,$$
and thus
$$ t - |x + v(p) t| + 2L \gtrsim \left (1 -\gamma \right ) t + 2L \gtrsim t.$$ 

The field estimates then follow from Theorem \ref{GS}.
In particular, we use the estimates on the fields and spatial characteristics to find
$$\left |E(t, \mcX(t, \tau, x, p)) \right | + \left |B(t, \mcX(t, \tau, x, p)) \right | \leq \frac{C}{(t+1)\left (t - |\mcX(t)| + 2L \right )} \leq \frac{C}{(t+L)^2} \lesssim t^{-2}$$
for all $\tau \in [0,t]$, $(x,p) \in \mcS_f(\tau)$.
Similarly, we have
$$\left |E(t, x +v(p)t)) \right | + \left |B(t,x +v(p)t)) \right | \leq \frac{C}{(t+1)\left (t - |x + v(p)t | + 2L \right )}  \lesssim t^{-2}$$
for $|x| \lesssim \ln(t)$ and $|p| \leq \beta$, which, when combined with the above field estimate, implies the decay of $\mcK_0(t).$

The estimate on derivatives of the fields similarly yields
$$|\nabla_x E(t,\mcX(t, \tau, x, p))| + |\nabla_x B(t,\mcX(t, \tau, x, p))| \leq \frac{C\ln(t+2)}{(t+1)\left (t - |\mcX(t)| + 2L \right )^2} \leq \frac{C\ln(t+2)}{(t+L)^3} \lesssim t^{-3}\ln(t)$$
for all $\tau \in [0,t]$, $(x,p) \in \mcS_f(\tau)$.
Additionally, we have
$$|\nabla_x E(t,x +v(p)t)| + |\nabla_x B(t,x +v(p)t)| \leq \frac{C\ln(t+2)}{(t+1)\left (t - |x +v(p)t| + 2L \right )^2} \lesssim t^{-3}\ln(t)$$
for $|x| \lesssim \ln(t)$ and $|p| \leq \beta$, which gives the stated decay of $\mcK_1(t)$.
\end{proof}

Finally, we will need one lemma that will be used repeatedly to establish \alt{the} convergence of a variety of quantities as $t \to \infty$.

\begin{lemma}
\label{LH}
Let $\mcH : [0,\infty) \times \bfR^3_y \times \bfR^3_p$ be a continuous, bounded function that is continuously differentiable in its third argument with bounded derivatives. 
Then, for any $x\in\bfR^3$  and  $\mcR(t)>0$ satisfying    $x+\Gamma_{\mcR(t)}\subset\Gamma_t$ for $t$ sufficiently large,
we have 
$$ \int_{|y| \lesssim \mcR(t)} \left | \mcH \left (t,y, v^{-1}\left ( \frac{x-y}{t} \right ) \right ) - \mcH \left (t,y, v^{-1}\left (\frac{x}{t} \right ) \right ) \right | \ dy \lesssim t^{-1} \mcR(t)^4 \Vert \nabla_p \mcH(t) \Vert_\infty.$$
\end{lemma}

\begin{proof}
As $v^{-1}(q)$ and its derivatives are uniformly bounded for $|q| < 1$, we omit their contribution below.
Estimating the given integral, we find
\begin{align*}
& \int_{|y| \lesssim \mcR(t)} \left | \mcH \left (t,y, v^{-1}\left ( \frac{x-y}{t} \right ) \right ) - \mcH \left (t,y,v^{-1}\left (\frac{x}{t} \right ) \right) \right | \ dy \\
& \qquad = \int_{|y| \lesssim \mcR(t)}   \left | \int_0^1 \frac{d}{d\theta} \left [\mcH \left (t, y, v^{-1} \left ( \frac{x-\theta y}{t}  \right ) \right ) \right ] d\theta \right | dy\\
& \qquad  \lesssim t^{-1}\Vert \nabla_p \mcH(t) \Vert_\infty \int_{|y| \lesssim \mcR(t)} |y|  \ dy \\
&  \qquad \lesssim t^{-1} \mcR(t)^4\Vert \nabla_p \mcH(t) \Vert_\infty,
\end{align*}
and the proof is complete.
\end{proof}

\section{Lower Regularity Estimates}
\label{sec:lower}

Within this section, we obtain estimates on first derivatives of the particle distribution and fields, as well as, characteristics and the charge and current densities.
Because Theorem \ref{GS} provides $C^1$ solutions, higher (i.e., $C^2$) regularity is not needed for the estimates within this section.

\subsection{Characteristics}
\label{sect:chars}

\begin{lemma}
\label{LDchar}
For $\tau$ sufficiently large, $t \geq \tau$, and $(x, p) \in \bS_f(\tau)$, we have
$$\left | \frac{\partial \mcX}{\partial p}(t,\tau, x, p) \right | \lesssim t \qquad \mathrm{and} \qquad \left | \frac{\partial \mcP}{\partial p}(t,\tau, x, p) \right | \lesssim 1.$$
\end{lemma}

\begin{proof}
Taking a $p$ derivative in \eqref{char} yields
\begin{equation*}
\left \{
\begin{aligned}
& \frac{\partial \dot{\mathcal{X}}}{\partial p}(t)= \mathbb{A}(\mcP(t)) \frac{\partial \mathcal{P}}{\partial p}(t),\\
& \frac{\partial \dot{\mathcal{P}}}{\partial p}(t)= \nabla_x K(t, \mcX(t), \mcP(t)) \frac{\partial \mcX}{\partial p}(t) + \nabla_p K(t, \mcX(t), \mcP(t)) \frac{\partial \mcP}{\partial p}(t),\\
& \frac{\partial \mcX}{\partial p}(\tau) = 0 ,\qquad \frac{\partial \mcP}{\partial p}(\tau) = \mathbb{I}.
\end{aligned}
\right.
\end{equation*}
Upon integrating, we can rewrite the latter ODE as
$$\frac{\partial \mcP}{\partial p}(t) = \mathbb{I} +\int_\tau^t \left ( \nabla_x K(s, \mcX(s), \mcP(s)) \frac{\partial \mcX}{\partial p}(s) + \nabla_p K(s, \mcX(s), \mcP(s)) \frac{\partial \mcP}{\partial p}(s) \right ) ds$$
so that by Lemmas \ref{lem:v} and \ref{LPrelim} we have 
$$\left |\frac{\partial \mcP}{\partial p}(t) \right | \leq 1 + \int_\tau^t \left ( \mcK_1(s) \left | \frac{\partial \mcX}{\partial p}(s) \right | +  \mcK_0(s) \left | \frac{\partial \mcP}{\partial p}(s) \right | \right ) ds \leq 1 + C\int_\tau^t \left (s^{-a} \left | \frac{\partial \mcX}{\partial p}(s) \right | + s^{-2} \left | \frac{\partial \mcP}{\partial p}(s) \right | \right ) ds$$
for any $a \in (2,3)$.
Now, we fix some $\delta \in [4,6]$ and define
$$T_0(\tau) = \sup \left \{ T \geq \tau : \left | \frac{\partial \mcX}{\partial p}(s) \right | \leq \delta (s - \tau) \ \mathrm{and} \   \left | \frac{\partial \mcP}{\partial p}(s) \right | \leq \delta \ \mathrm{for \ all} \ s \in [\tau, T] \right \}.$$
Note that $T_0 > \tau$ due to the initial conditions. 
Then, estimating for $t \in [\tau, T_0)$, we have
$$\left |\frac{\partial \mcP}{\partial p}(t) \right | \leq  1 + C\delta \int_\tau^t \left [ (s-\tau) s^{-a} + s^{-2} \right ] ds \leq 1 + C\delta \left ( \tau^{2-a} + \tau^{-1} \right ) \leq 2 < \frac{1}{2}\delta$$
for $\tau$ sufficiently large. 
Similarly, we integrate the first ODE to find
$$\left |\frac{\partial \mcX}{\partial p}(t) \right | \leq \int_\tau^t  \left | \frac{\partial \mcP}{\partial p}(s) \right |  \ ds < \frac{1}{2}\delta (t - \tau)$$
for $\tau$ sufficiently large. 
Hence, we find $T_0 =\infty$ and the estimate on $p$-derivatives follows.
\end{proof}

Because the fields decay rapidly in time inside the light cone, we can immediately establish the limiting behavior of the momentum characteristics. 

\begin{lemma}
\label{L6}
For any $\tau\geq 0$ and $(x, p) \in \mcS_f(\tau)$, 
the limiting momenta $\mcP_\infty$ defined by
$$\mcP_\infty(\tau, x, p) :=  \lim_{t \to \infty} \mcP(t, \tau, x, p) = p + \int_\tau^\infty K(s, \mcX(s, \tau, x, p), \mcP(s, \tau, x, p)) ds$$ 
exist, and are $C^2$, bounded, and invariant under the characteristic flow, namely $\mcP_\infty$ satisfies
$$\mcP_\infty(t, \mcX(t, \tau, x,p), \mcP(t, \tau, x, p)) = \mcP_\infty(\tau, x,p)$$
for any $t \geq 0$.
Finally, we have the convergence estimate
$$|\mcP(t, \tau, x, p) - \mcP_\infty(\tau, x, p) | \lesssim \int_t^\infty \mcK_0(s) \ ds,$$
which further yields
\begin{equation}
\label{Winfest}
|\mcP(t, \tau, x, p) - \mcP_\infty(\tau, x, p) | \lesssim t^{-1}.
\end{equation}
\end{lemma}
\begin{proof}
For any $(x,p) \in \mcS_f(\tau)$, we have from the characteristic equations
$$\mcP(t,\tau, x, p) = p + \int_\tau^t K(s, \mcX(s, \tau, x, p), \mcP(s, \tau, x, p)) ds.$$
Thus, define
$$\mcP_\infty(\tau, x, p) = p+ \int_\tau^\infty K(s, \mcX(s, \tau, x, p), \mcP(s, \tau, x, p)) ds$$
for every $\tau \geq 0$ and $(x,p) \in \mcS_f(\tau)$. 
Then, the convergence estimate
$$| \mcP(t) -\mcP_\infty| \lesssim \int_t^\infty \mcK_0(s) \ ds \lesssim t^{-1}.$$ 
follows from Lemma \ref{LPrelim}.
Thus, the limit exists and is uniformly bounded for every $\tau \geq 0$ and $(x,p) \in \mcS_f(\tau)$.
Finally, $\mcP_\infty$ is $C^2$ due to the regularity of the electromagnetic fields and spatial characteristics and invariant under the flow due to the time-reversibility of \eqref{char}. 
 \end{proof}

Next, we control the derivatives of these limiting characteristics and show that they are invertible functions of $p$ for sufficiently large $\tau$. This property will be useful later to perform a change of variables and establish limits of spatial averages.

\begin{lemma}
\label{Vinfinvert}
For $(x,p) \in \mcS_f(\tau)$ we have
$$\left |\frac{\partial \mcP_\infty}{\partial p}(\tau, x, p) - \mathbb{I} \right | \lesssim \tau^{-1}\ln(\tau).$$
Thus, there is $T_1 > 0$ such that for all $\tau \geq T_1$ and $(x, p) \in \mcS_f(\tau)$, we have
$$\left |\det \left (\frac{\partial \mcP_\infty}{\partial p}(\tau, x, p) \right ) \right | \geq \frac{1}{2}.$$
Consequently, for $\tau \geq T_1$ and $(x, p) \in \mcS_f(\tau)$, the $C^2$ mapping $p \mapsto \mcP_\infty(\tau, x, p)$ is injective and invertible.
\end{lemma}

\begin{proof}
First, note that the limiting momenta given by Lemma \ref{L6} satisfy 
$$ \frac{\partial \mcP_\infty}{\partial p}(\tau, x, p) = \mathbb{I} + \int_\tau^\infty \left ( \nabla_x K(s, \mcX(s), \mcP(s))  \frac{\partial \mcX}{\partial p}(s) + \nabla_p K(s, \mcX(s), \mcP(s))  \frac{\partial \mcP}{\partial p}(s) \right ) \ ds. $$
Hence, using Lemmas \ref{LPrelim} and \ref{LDchar}, we have
$$\left | \frac{\partial \mcP_\infty}{\partial p}(\tau) - \mathbb{I} \right | \lesssim \int_\tau^\infty s^{-2} (1 + \ln(s)) ds \lesssim \tau^{-1}\ln(\tau).$$
Therefore, by the continuity of the mapping $A \mapsto \det(A)$, there is $T_1 > 0$ such that for all $\tau \geq T_1$ and $(x, p) \in \bS_f(\tau)$, we have
$$ \left | \det \left (\frac{\partial \mcP_\infty}{\partial p}(\tau, x, p)  \right )\right | \geq \frac{1}{2}.$$
\end{proof}

Next, we define the collection of all limiting momenta on $\bS_f(t)$, which will serve as the support of the limiting spatial average.
Notice that due to its invariance under the flow defined by \eqref{char}, we have
\begin{equation}\label{inv-p}\left \{ \mcP_\infty(\tau, x, p) : (x, p) \in \bS_f(\tau) \right \} = \left \{ \mcP_\infty(0, x, p) : (x, p) \in \bS_f(0) \right \}\end{equation}
for all $\tau \geq 0$. Hence, define
$$\Omv := \left \{ \mcP_\infty(0, x, p) : (x, p) \in \bS_f(0) \right \}.$$
As $\mcP_\infty (0,x, p)$ is continuous due to Lemma \ref{L6}, its range $\Omv$ on the compact set $\bS_f(0)$ is also compact.

\subsection{Properties of the Translated Distribution Function}

With the basic properties of characteristics determined, we introduce some notation relating to the translated distribution functions.
As mentioned in the introduction, we let
$$g(t,x,p) = f(t,x+v(p)t, p)$$
so that 
$g(t,x-v(p)t,p) = f(t,x, p)$,
and because the translation alters the spatial characteristics of the system, we further define
\begin{equation}
\label{gcharalt}
\mcY(t,\tau, x,p) = \mcX(t,\tau, x, p) - v \bigg (\mcP(t, \tau, x, p) \bigg ) t
\end{equation}
so that
\begin{equation}
\label{gchar}
 \dot{\mcY}(t) = -t \mathbb{A}(\mcP(t)) K \biggl (t, \mcY(t) +  v\left (\mcP(t) \right ) t, \mcP(t) \biggr)
\end{equation}
with $\mcY(\tau) = x - v(p)\tau$. 
In addition, note that $\Vert g(t) \Vert_\infty \leq \Vert f_0 \Vert_\infty$ for all $t \geq 0$ and
$$g(t,x,p) = 0 \qquad \mathrm{for} \qquad |p| \geq \beta.$$
For $t \geq 0$ define the support of $g(t)$ by
$$\bS_g(t) = \overline{\left \{ (x,p) \in \R^6 : g(t,x,p) \neq 0 \right \}}$$
and note that $\bS_g(0) = \bS_f(0)$.

Because our approach relies upon the growth of the spatial support and momentum derivatives of $g$,
we further define 
%
$$\mcR(t) = \sup \{|\mcY(t,0,x,p)| : (x,p) \in \bS_g(0) \}$$
and let 
$$\mcG_1(t) = 1 + \| \nabla_p g (t) \|_\infty$$
and
$$\mcG_2(t) = \mcG_1(t) +  \| \nabla_p^2 g (t) \|_\infty.$$

With suitable decay of the field and its derivatives established in Section \ref{sec:prelim}, we now study the behavior of the translated characteristics of \eqref{RVM}.

\begin{lemma}
\label{Xsupp}
For every $\tau \geq 0$, and $(x,p) \in \mcS_g(\tau)$ the characteristics satisfy
$$\left | \mcY(t, \tau, x, p)  - (x -v(p)\tau) \right | \lesssim \int_\tau^t s \mcK_0(s) ds$$
and 
$$\mcR(t) \lesssim 1 + \int_0^t s \mcK_0(s) ds.$$
In particular, Lemma \ref{LPrelim} further implies
$$\left | \mcY(t, \tau, x, p) \right | \lesssim \ln(t) \qquad \mathrm{and} \qquad 
\mcR(t) \lesssim \ln(t).$$
\end{lemma}

\begin{proof}
Using \eqref{gchar} we immediately find
$$\left | \dot{\mcY}(t) \right |  \leq t |\mathbb{A}(v(\mcP(t)))| |K(t,\mcX(t), \mcP(t))| \lesssim t \mcK_0(t),$$
and thus
$$\left | \mcY(t, \tau, x, p)  - (x - v(p) \tau) \right | \leq \int_\tau^t \left |  \dot{\mcY}(s, \tau, x, p)  \right | \ ds \lesssim \int_\tau^t s \mcK_0(s) \ ds$$
for fixed $\tau \geq 0$ and $(x,p) \in \bS_g(\tau)$.
Furthermore, this implies
$$\left | \mcY(t, 0, x, p) \right | \lesssim  |x| + \int_0^t s |K(s, \mcX(s), \mcP(s))| \ ds \lesssim 1 + \int_0^t s \mcK_0(s) \ ds$$
for $(x,p) \in \bS_g(0)$.
The estimate on the spatial radius then follows as
$$\mcR(t) = \sup_{(x,p) \in \bS_g(0)} \left | \mcY(t, 0, x, p) \right |  \lesssim 1 + \int_0^t s \mcK_0(s) \ ds,$$
which by Lemma \ref{LPrelim} implies the stated estimate.
\end{proof}

This growth estimate of the spatial support of $g$ will prove useful later on when determining the asymptotic behavior of the Lorentz force, as we need only consider spatial values within the support of $g$, which satisfy $|x| \lesssim \ln(t)$, to derive such estimates.
Next, we show that momentum derivatives of $g$ grow more slowly than those of $f$, which are established by Theorem \ref{GS}. 
\begin{lemma}
\label{Dvg}
We have
$$\mcG_1(t) \lesssim 1 + \int_1^t \left (s \mcK_0(s) + s^2 \mcK_1(s)  \right ) \ ds,$$
and thus by Lemma \ref{LPrelim}
$$\mcG_1(t) \lesssim \ln^2(t).$$
\end{lemma}

\begin{proof}
To establish the result we estimate
$$\partial_{p_k} g(t,x,p) = \left ( t  \mathbb{A}_{jk} \partial_{x_j} f + \partial_{p_k} f \right )(t,x+v(p)t, p).$$
Applying the Vlasov operator 
$$\mcV := \partial_t + v(p) \cdot \nabla_x + K(t,x,p) \cdot \nabla_p$$
to the untranslated version of this quantity yields
$$ \mcV \biggl ( t  \mathbb{A}_{jk}(p) \partial_{x_j} f(t, x,p) +  \partial_{p_k} f(t, x, p)  \biggr )
= - t \mathbb{A}_{jk} \partial_{x_j} K \cdot \nabla_p f - \partial_{p_k} K \cdot \nabla_p f + t\partial_{x_j} f K \cdot \nabla_p \mathbb{A}_{jk},$$
and inverting gives
\begin{align*}
\left (t  \mathbb{A}_{jk}  \partial_{x_j} f + \partial_{p_k} f \right )(t,x,p) & = \partial_{p_k} f_0(\mcX(0),\mcP(0))- \int_0^t s \mathbb{A}_{jk}\left (\mcP(s) \right ) \partial_{x_j} K(s,\mcX(s), \mcP(s)) \cdot \nabla_p f(s, \mcX(s), \mcP(s)) ds\\
& \qquad - \int_0^t  \partial_{p_k} K(s,\mcX(s), \mcP(s)) \cdot \nabla_p f(s, \mcX(s), \mcP(s)) ds\\
& \qquad  + \int_0^t s\partial_{x_j} f(s, \mcX(s), \mcP(s))   K(s,\mcX(s), \mcP(s)) \cdot \nabla_p  \mathbb{A}_{jk}\left (\mcP(s) \right ) ds\\
\end{align*}
for $j, k = 1, 2, 3$.
From the estimates on the derivatives of $f$ within Theorem \ref{GS}, as well as, those of the fields and their derivatives within Lemma \ref{LPrelim}, this implies
\begin{eqnarray*}
\left | t  \mathbb{A}_{jk}(p) \partial_{x_k} f(t,x,p) + \partial_{p_k} f(t,x,p) \right | & \leq & \| \partial_{p_k} f_0\|_\infty + \int_0^t s \mcK_1(s) \|\nabla_p f(s)\|_\infty \ ds \\
& \ & + \int_0^t |B(s,\mcX(s)| \|\nabla_p f(s) \|_\infty ds + \int_0^t s \mcK_0(s) \|\nabla_x f(s) \|_\infty ds \\
& \lesssim & 1 + \int_0^t s^2 \mcK_1(s) ds + \int_0^t s \mcK_0(s) ds\\
& \lesssim & 1 + \int_1^t s^{-1} \left ( 1+ \ln(s) \right ) \ ds\\
& \lesssim & \ln^2(t)
\end{eqnarray*}
for $j, k = 1, 2, 3$. Thus, we conclude
$$ \| \nabla_p g(t) \|_\infty  \lesssim \ln^2(t),$$
which completes the lemma.
\end{proof}

With the behavior of momentum derivatives along the translated characteristics well-understood as $t \to \infty$, we can establish the limiting behavior of the spatial average.

\begin{lemma}
\label{Funif}
There exists $F_\infty \in C_c^2(\bfR^3)$ with $\mathrm{supp}(F_\infty) = \Omv$ such that
$$F(t,p) = \int f(t,x, p) \ dx = \int g(t,x, p) \ dx$$
satisfies $F(t,p) \to F_\infty(p)$ uniformly as $t \to \infty$
with
$$\| F(t) - F_\infty \|_\infty \lesssim \int_1^t \left (\mcK_0(s)\mcR(s)^3\mcG_1(s) + s \mcK_1(s)  \right ) \ ds \lesssim t^{-1}\ln^5(t).$$
In particular, the limit preserves particle number, i.e.
$$\int F_\infty(p) \ dp= \mcM.$$
\end{lemma}

\begin{proof}
Upon integrating the Vlasov equation of \eqref{VMg} in $x$ and integrating by parts, we find
\begin{eqnarray*}
\left | \partial_t \int g(t,x,p) \ dx \right | & = & \left | \int K(t, x+v(p)t, p) \cdot (t \mathbb{A}(p)\nabla_x - \nabla_p) g(t,x,p) \ dx \right |\\
& = & \left | t  \int \mathrm{tr} \left [ \mathbb{A}(p) \nabla_x K(t, x+v(p)t, p) \right ] g(t,x,p) \ dx + \int K(t, x+v(p)t,p) \cdot \nabla_p g(t,x,p) \ dx \right |\\
& \lesssim & t \mcK_1(t) F(t,p) + \mcK_0(t) \mcR(t)^3 \mcG_1(t).
\end{eqnarray*}
Thus, we use Lemmas \ref{LPrelim}, \ref{Xsupp}, and \ref{Dvg} to find
$$\left | \partial_t F(t,p) \right | \lesssim t^{-2} \ln(t) F(t,p) + t^{-2}\ln^5(t).$$
As $F(0) \in L^\infty(\bfR^3)$ and the latter term above is integrable in time, we find
$$F (t,p) \leq F (0,p) + \int_0^t \left | \partial_t F (s,p) \right | \ ds \lesssim 1 +  \int_1^t s^{-2} \ln(s) F(s, p) \ ds,$$
and after taking the supremum and invoking Gr\"onwall's inequality, this yields
\begin{equation}
\label{FLinfty}
\| F(t) \|_\infty \leq \exp \left (\int_1^t s^{-2} \ln(s) \ ds \right ) \lesssim 1.
\end{equation}
Returning to the estimate of $\partial_t F$, we use the uniform bound on $\|F(t)\|_\infty$ to find
$$\left | \partial_t F(t,p) \right |  \lesssim t^{-2}\ln^5(t),$$
which implies that $ \| \partial_t F(t)\|_\infty$ is integrable.
This bound then establishes the estimate for $s \geq t$
$$\Vert F(t) - F(s) \Vert_\infty = \left \Vert \int_s^t \partial_t F(\tau) \ d\tau \right \Vert_\infty
\leq \int_t^s \Vert \partial_t F(\tau) \Vert_\infty \ d\tau
 \lesssim t^{-1}\ln^5(t),$$
and taking $s \to \infty$ establishes the limit. More precisely, as $F(t,p)$ is continuous and the limit is uniform in $p$, there is $F_\infty \in C(\bfR^3)$ such that
$$ \| F(t) - F_\infty \|_\infty  \lesssim t^{-1}\ln^5(t).$$
Next, we verify the properties of the limiting spatial average.
Due to the uniform convergence, we further conclude weak convergence of $F(t,p)$ as a measure, namely
\begin{equation}
\label{Weaklimit}
\lim_{t \to \infty} \int \psi(p) F(t,p) dp= \int \psi(p) F_\infty(p) dp
\end{equation}
for every $\psi \in C_b \left (\bfR^3 \right )$.
In this direction, let $\psi  \in C_b(\bfR^3)$ be given and fix any $T \geq T_1$ from Lemma \ref{Vinfinvert}.
Then, we apply the measure-preserving change of variables $(\tilde{x}, \tilde{p}) = ( \mcX(T, t, x, p), \mcP(T, t, x, p))$, so that 
\begin{eqnarray*}
\lim_{t \to \infty} \int \psi(p) F(t,p) \ dp
& = & \lim_{t \to \infty} \iint \psi(p) \ f(t,x,p) \ dp dx\\
& = & \lim_{t \to \infty} \iint\limits_{\bS_f(t)} \psi(p) \ f(T, \mcX(T,t,x,p), \mcP(T,t,x,p)) \ dp dx\\
& = & \lim_{t \to \infty}\iint\limits_{\bS_f(T)} \psi(\mcP(t,T,\tilde{x}, \tilde{p})) \ f(T, \tilde{x}, \tilde{p}) \ d\tilde{p} d\tilde{x}\\
& = & \iint\limits_{\bS_f(T)} \psi(\mcP_\infty(T,\tilde{x}, \tilde{p})) f(T, \tilde{x}, \tilde{p}) \ d\tilde{p} d\tilde{x}
\end{eqnarray*}
by Lebesgue's Dominated Convergence Theorem. Now, by Lemma \ref{Vinfinvert} for any $(\tilde{x}, \tilde{p}) \in \bS_f(T)$, the mapping $\tilde{p} \mapsto \mcP_\infty(T, \tilde{x}, \tilde{p})$ is $C^2$ with $$\left |\det \left (\frac{\partial \mcP_\infty}{\partial p}(T, \tilde{x}, \tilde{p}) \right )\right | \geq \frac{1}{2},$$ and thus bijective onto $\Om_p$. 
Hence, letting $u = \mcP_\infty(T, \tilde{x}, \tilde{p})$, we perform a change of variables and drop the tilde notation to find
$$\lim_{t \to \infty} \int \psi(p) F(t,p) dp
= \iint \psi(u) f(T, x, W(u)) \frac{\chfn_{\Om_p}(u)}{\left | \det \left (\frac{\partial \mcP_\infty}{\partial p}(T, x, W(u)) \right ) \right |} \ du dx$$
where for fixed $x$, the $C^2$ function $W$ defined on $\Omv$ is given by
$$W(u) = \left (\mcP_\infty \right)^{-1}(T,x,u).$$
Therefore, by uniqueness of the weak limit we find from \eqref{Weaklimit}
\begin{equation}\label{Finfty}F_\infty(u) = \int f(T, x, W(u)) \chfn_{\Om_p}(u) \left | \det \left (\frac{\partial \mcP_\infty}{\partial p}(T, x, W(u)) \right) \right |^{-1} \ dx\end{equation}
for any $u \in \bfR^3$, and thus $F_\infty \in C_c^2(\bfR^3)$ with supp$(F_\infty) = \Omv$.
Due to the invariance \eqref{inv-p} of $\mcP_\infty(\tau,x,p)$ with respect to the time variable, $F_\infty$ is also independent of the choice of $T$ and thus, well defined.
Finally, the conservation of particle number is obtained by merely choosing $\psi(p) = 1$ within \eqref{Weaklimit} and using the time-independence of this quantity.
%
\end{proof}

Now that we have shown the convergence of $F(t,p)$, we establish the precise asymptotic profiles of the charge and current densities.
First, recall that 
$\mcD(p) = \left | \det(\mathbb{A}(p)) \right |^{-1} = \left (1+ |p|^2 \right )^{5/2}$ due to \eqref{Ddef}.
From the limiting density $F_\infty(p)$, we define the induced charge density by
$$\rho_\infty(q) = F_\infty \left (v^{-1}(q) \right)  \mcD \left (v^{-1}(q) \right)$$
for $q \in \Gamma_1$ and smoothly extend it to $0$ for all $q \in \bfR^3 \setminus \Gamma_1$.
The induced current density is then defined by
$$j_\infty(q) = q\rho_\infty \left (v^{-1}(q) \right )$$
for every $q \in \bfR^3$.
Of course, in the case of multiple species the quantities $F(t,p)$ and $F_\infty(p)$ are replaced by sums over $\alpha$ of $F^\alpha(t,p)$ and $F^\alpha_\infty(p)$, respectively, within the remaining lemmas.
Then, we obtain the limiting behavior of the charge and current densities as follows.

\begin{lemma}
\label{LDensity}
The scaled charge density converges to a self-similar limiting function.
In particular,  we have 
$$\sup_{x \in \bfR^3} \left | t^3 \rho(t,x) - \rho_\infty\left (\frac{x}{t} \right ) \right | \lesssim \| F(t) - F_\infty\|_\infty + t^{-1} \mcR(t)^4 \mcG_1(t).$$
Thus, in view of Lemmas \ref{Xsupp}, \ref{Dvg}, and \ref{Funif}, we find
$$\sup_{x \in \bfR^3} \left | t^3 \rho(t,x) - \rho_\infty\left (\frac{x}{t} \right ) \right | \lesssim t^{-1} \ln^6(t).$$
\end{lemma}
\begin{proof}
Due to the growth of the spatial support of $f(t,x,p)$ established in Lemma \ref{LPrelim}, it suffices to consider only $|x| \lesssim \gamma t$. To prove the result, we first conduct a change of variables
$$y =  x - v(p)t$$
with
$$dy =  t^3 \left | \det(\mathbb{A}(p)) \right | dp$$
in the integral of $\rho(t,x)$.
Hence, this transforms
$$\rho(t,x) =  \int f (t, x, p) \ dp = \int g (t, x - v(p)t, p) \ dp$$
into
\begin{equation}
\label{rhop}
\rho(t,x) = \frac{1}{t^3}  \int g \left (t, y, v^{-1} \left (\frac{x-y}{t}  \right ) \right ) \mcD\left (v^{-1} \left (\frac{x-y}{t}  \right ) \right ) dy.
\end{equation}
Letting $$\mcH(t,y,p) =\mcD(p)  g(t, y, p),$$
this can be expressed as
$$\rho(t,x) = t^{-3} \int \mcH \left (t,y,  v^{-1} \left (\frac{x-y}{t}  \right ) \right ) \ dy.$$
Therefore, as $\rho_\infty(q) = F_\infty\left (v^{-1}(q) \right ) \mcD\left (v^{-1}(q) \right )$ and $F(t,p) = \int g(t,y,p) \ dy$, the difference of the charge densities can be split into
$$ \left | t^3 \rho(t,x) - \rho_\infty \left ( \frac{x}{t}  \right) \right |  \leq I + II$$
where
$$I = \int \left | \mcH \left (t,y,  v^{-1} \left ( \frac{x-y}{t}  \right ) \right )  - \mcH \left (t, y, v^{-1} \left (\frac{x}{t} \right )  \right )\right | dy $$
and
$$II = \left | F \left (t,  v^{-1} \left (\frac{x}{t} \right) \right ) - F_\infty \left (v^{-1} \left (\frac{x}{t} \right)  \right ) \right| \mcD\left ( v^{-1} \left (\frac{x}{t}  \right )\right ).$$

We estimate $I$ by noting that $\mcH(t,y,p)$ is supported on $|y| \lesssim \ln(t)$ and invoking Lemma \ref{LH} to find
$$I \lesssim  t^{-1} \mcR(t)^4 \Vert \nabla_p \mcH(t) \Vert_\infty.$$
As $\mcD(p)$, its derivatives, and $\Vert g(t) \Vert_\infty$ are all uniformly bounded, we have
$\Vert \nabla_p \mcH(t) \Vert_\infty \lesssim \mcG_1(t)$, 
and combining this with the estimates of Lemmas \ref{Xsupp} and \ref{Dvg} gives
$$I \lesssim  t^{-1} \mcR(t)^4 \mcG_1(t) \lesssim t^{-1} \ln^6(t).$$
The estimate for $II$ uses Lemma \ref{Funif}, the compact support of $F(t)$ and $F_\infty$, and $\left | \frac{x}{t} \right | \lesssim \gamma < 1$ so that $\mcD(p)$ is bounded on the supports of these functions and
$$II \lesssim \| F(t) - F_\infty \|_\infty \lesssim t^{-1} \ln^5(t).$$
Combining these estimates then yields the stated result.
\end{proof}

Next, we estimate the current density in a similar fashion.
\begin{lemma}
\label{LCurrent}
The scaled current density converges to a self-similar limiting function.
In particular, we have 
$$\sup_{x \in \bfR^3} \left | t^3 j(t,x) - j_\infty \left (\frac{x}{t} \right) \right | \lesssim
\| F(t) - F_\infty \|_\infty + t^{-1}\mcR(t)^4 \mcG_1(t).$$
Thus, in view of Lemmas \ref{Xsupp}, \ref{Dvg}, and \ref{Funif}, we find
$$\sup_{x \in \bfR^3} \left | t^3 j(t,x) - j_\infty \left ( \frac{x}{t} \right ) \right | \lesssim t^{-1} \ln^6(t).$$
\end{lemma}
\begin{proof}
Similar to the previous lemma, it suffices to consider only $|x| \lesssim \gamma t$. 
We again prove the result by first conducting a change of variables, namely
$$y =  x - v(p)t$$
in order to transform
$$j(t,x) =  \int v(p) g (t, x- v(p)t, p) \ dp$$
into
$$j(t,x) = \frac{1}{t^3}  \int \frac{x-y}{t}  g \left (t, y, v^{-1} \left (\frac{x-y}{t}  \right ) \right ) \mcD \left ( v^{-1} \left (\frac{x-y}{t}  \right ) \right ) dy$$
or
\begin{equation}
\label{jp}
j(t,x) = \frac{1}{t^3}  \int v(r) \mcD(r) g\left (t, y, r \right ) \biggr |_{r = v^{-1}\left (\frac{x-y}{t} \right )} dy.
\end{equation}
Letting
$$\mcH(t,y,p) = v(p) \mcD(p) g(t, y, p),$$
this can be expressed more succinctly as
$$j(t,x) = t^{-3}  \int \mcH \left (t, y, v^{-1}\left (\frac{x-y}{t} \right ) \right ) dy.$$ 
Recall that
$$j_\infty (q) = q\mcD \left (v^{-1}\left (q \right ) \right ) F_\infty  \left (v^{-1}\left (q \right ) \right ).$$
With this, the difference of the current densities can be split into two terms as
\begin{eqnarray*}
\left | t^3 j(t,x) - j_\infty \left ( \frac{x}{t} \right ) \right | &\leq & \int \left | \mcH \left (t,y,  v^{-1} \left (\frac{x-y}{t}  \right ) \right )  - \mcH \left (t, y,  v^{-1} \left ( \frac{x}{t} \right) \right )\right | dy  \\
& \ & + \left | \frac{x}{t} \right | \mcD \left (v^{-1}\left (\frac{x}{t} \right ) \right ) \left | F \left (t,  v^{-1}\left (\frac{x}{t} \right )\right) - F_\infty \left ( v^{-1}\left (\frac{x}{t} \right ) \right ) \right| \\
& =: & I + II.
\end{eqnarray*}

As in the previous lemma, we estimate the first portion by invoking
Lemma \ref{LH}.
As the product $v(p)\mcD(p)$ and all of its derivatives are uniformly bounded, we again deduce
$\Vert \nabla_p \mcH(t) \Vert_\infty \lesssim \mcG_1(t)$.
Combining this with the estimates of Lemmas \ref{Xsupp} and \ref{Dvg} then gives
$$I \lesssim  t^{-1} \mcR(t)^4 \mcG_1(t) \lesssim t^{-1} \ln^6(t).$$
The latter term uses $\left | \frac{x}{t} \right | \lesssim \gamma < 1$, Lemma \ref{Funif}, and the compact support of $F(t)$ and $F_\infty$ so that $\mcD(p)$ is again bounded on the supports of these functions to find
$$II \lesssim \| F(t) - F_\infty \|_\infty \lesssim t^{-1} \ln^5(t).$$
Combining the estimates of $I$ and $II$ then yields the stated result.
\end{proof}

\section{Higher Regularity Estimates}
\label{sec:higher}

With the estimates on lower order derivatives of $g, E$, and $B$ established, we turn our attention to second-order derivatives, which have not been previously obtained for \eqref{RVM} and are crucial to constructing the time asymptotic limits and scattering profile that we wish to study.
The lemmas of this section will also be needed to prove the small data theorem.
Throughout this section we will assume that the second derivatives of the fields satisfy the estimate
\begin{equation}
\label{K2}
\mcK_2(t) \lesssim t^{-4} \ln(t).
\end{equation}

First, we estimate second-order derivatives of the translated distribution function; in particular, establishing a nearly sharp estimate on their growth using the decay of the second derivative of the fields.

\begin{lemma}
\label{D2g}
We have
$$\mcG_2(t) \lesssim 1 + \int_1^t \left (s^3 \mcK_2(s) + s^2 \ln^2(s)\mcK_1(s) + s \ln^2(s)\mcK_0(s) \right )\,ds,$$
and thus by Lemma \ref{LPrelim} and \eqref{K2} 
$$\mcG_2(t) \lesssim \ln^4(t).$$
Furthermore, 
$$\Vert \nabla^2_x g(t) \Vert_\infty \lesssim 1 \qquad \mathrm{and} \qquad \Vert \nabla_p \nabla_x g(t) \Vert_\infty \lesssim \ln^2(t).$$
\end{lemma}

\begin{remark}
\label{rmk:D2g}
If, in addition to the stated assumptions, we have $\Vert f_0 \Vert_{C^2} \leq \epsilon_0$, then the estimate on spatial derivatives becomes
$$\Vert \nabla^2_x g(t) \Vert_\infty \lesssim \epsilon_0.$$
\end{remark}

\begin{proof}
On bounded time intervals, e.g. $t \in [0,1]$, uniform-in-time estimates can be obtained via straightforward methods (see \cite{Glassey}). Hence, we take $t > 1$ throughout and focus on large time estimates only.
Additionally, we suppress the argument of $\mathbb{A}(v(p))$ and merely denote this term by $\mathbb{A}$.
To begin, we note that spatial derivatives of $g$ are merely spatial derivatives of $f$ evaluated along translated curves in phase space, and thus from Theorem \ref{GS} we find
\begin{equation}
\label{Dxg}
\Vert \nabla_x g(t) \Vert_\infty = \Vert \nabla_x f(t) \Vert_\infty \lesssim 1.
\end{equation}

With this, we let $\mcV_g$ denote the translated Vlasov operator appearing in \eqref{VMg}. Then by a straightforward computation we find
	\begin{equation}\label{eq:pp-commut}
	\begin{split}
	\mcV_g[\partial_{m}\partial_{n}g](t,x,p)
	=\ &
	t\partial_{m}\left\{\partial_{n}\left[\mathbb{A}K(t,x+v(p)t, p)\right]\cdot \nabla_{x}g\right\}(t,x,p)\\
	&+
	t\partial_{n}\left\{\partial_{m}\left[\mathbb{A}K(t,x+v(p)t, p)\right]\cdot \nabla_{x}g\right\}(t,x,p)\\
	&-
	t\partial_{m}\partial_{n}\left[\mathbb{A}K(t,x+v(p)t, p)\right]\cdot \nabla_{x}g(t,x,p)\\
	&-
	\partial_{m}\left\{\partial_{n}\left[K(t,x+v(p)t, p)\right]\cdot \nabla_{p}g\right\}(t,x,p)\\
	&-
	\partial_{n}\left\{\partial_{m}\left[K(t,x+v(p)t, p)\right]\cdot \nabla_{p}g\right\}(t,x,p)\\
	&+
	\partial_{m}\partial_{n}\left[K(t,x+v(p)t, p)\right]\cdot \nabla_{p}g(t,x,p)
	\end{split}
	\end{equation}
where $\partial_i\in\{\partial_{p_i},\partial_{x_i}\}$ for $1\leq i\leq3$.  Recall that all derivatives of $v(p)$ and $\mathbb{A}$ are bounded and that derivatives of $K$ satisfy
	\begin{align*}
	\partial_{p_m}\left[K(t,x+v(p)t, p)\right]
	=\ &
	t\partial_{x_i}K(t,x+v(p)t, p)\mathbb{A}_{mi}+\partial_{p_m}(v(p))\times B(t,x+v(p)t),\\
	\partial_{x_n}\partial_{p_m}\left[K(t,x+v(p)t, p)\right]
	=\ &
	t\partial_{x_n}\partial_{x_i}K(t,x+v(p)t, p)\mathbb{A}_{mi}+\partial_{p_m}(v(p))\times \partial_{x_n}B(t,x+v(p)t),\\
	\partial_{p_m}\partial_{p_n}\left[K(t,x+v(p)t, p)\right]
	=\ &
	{t\partial_{p_m}\left[\partial_{x_i}K(t,x+v(p)t, p)\mathbb{A}_{ni}\right]}
	+
	{\partial_{p_m}\left[\partial_{p_n}(v(p))\times B(t,x+v(p)t)\right]}\\
	=\ &
	{t^2\partial_{x_j}\partial_{x_i}K(t,x+v(p)t, p)\mathbb{A}_{ni}\mathbb{A}_{mj}}
	+
	{t\partial_{x_i}K(t,x+v(p)t, p)\partial_{p_m}\mathbb{A}_{ni}}\\
	&+
	{t\partial_{p_m}(v(p))\times\partial_{x_i}B(t,x+v(p)t)\mathbb{A}_{ni}}
	+
	{\partial_{p_m}\partial_{p_n}(v(p))\times B(t,x+v(p)t)}\\
	&+
	{t\partial_{p_n}(v(p))\times\partial_{x_i}B(t,x+v(p)t)\mathbb{A}_{mi}}.
	\end{align*}
Due to the estimates of Lemma \ref{LPrelim} and \eqref{K2}, this implies 
	\begin{align*}
	\left |\partial_{p_m}\left[K(t,x+v(p)t, p)\right] \right |
	&\lesssim
	t \mcK_1(t) + \mcK_0(t)
	\lesssim t^{-2}\ln(t),\\
	\left |\partial_{x_n}\partial_{p_m}\left[K(t,x+v(p)t, p)\right] \right |
	&\lesssim
	t \mcK_2(t) + \mcK_1(t)
	\lesssim t^{-3}\ln(t),\\
	\left |\partial_{p_m}\partial_{p_n}\left[K(t,x+v(p)t, p)\right] \right |
	&\lesssim
	t^2 \mcK_2(t) + t \mcK_1(t) + \mcK_0(t)
	\lesssim t^{-2}\ln(t).
	\end{align*}
Next, for $t > 1$ we define
	\begin{align*}
	\mcD_2(t) :=
	\|\nabla_x^2g(t)\|_\infty+t^{-1}\|\nabla_p\nabla_xg(t)\|_\infty+t^{-2}\|\nabla_p^2g(t)\|_\infty
	\end{align*}
which we shall bound uniformly in time.
Integrating \eqref{eq:pp-commut} with  $\partial_{m}\partial_{n}=\partial_{p_m}\partial_{p_n}$  along the characteristics of $\mcV_g$  gives
\begin{equation}
	\label{eq:gpp1}
	\begin{split}
	\|\nabla_p^2g(t)\|_{\infty}
\lesssim\ &
	1+\int_1^t \biggl [s\left (s^2 \mcK_2(s) + s\mcK_1(s) + \mcK_0(s) \right ) \|\nabla_xg(s)\|_\infty \\
	& + s\left ( s\mcK_1(s) + \mcK_0(s) \right )\|\nabla_p\nabla_xg(s)\|_\infty+  \left (s^2 \mcK_2(s) + s\mcK_1(s) + \mcK_0(s) \right )\|\nabla_pg(s)\|_\infty\\
	& + \left ( s\mcK_1(s) + \mcK_0(s) \right )\|\nabla_p^2g(s)\|_\infty\biggr]\,ds.
	\end{split}
	\end{equation}
Then, multiplying by $t^{-2}$ and using \eqref{Dxg} and Lemma \ref{Dvg} yields	
	\begin{equation}
	\label{eq:gpp2}
	\begin{split}
	t^{-2}\|\nabla_p^2g(t)\|_{\infty}
	\lesssim\ &
	t^{-2}+t^{-2}\int_1^ts\left(s^{-2}\ln(s)\|\nabla_xg(s)\|_\infty+s^{-2}\ln(s)\|\nabla_p\nabla_xg(s)\|_\infty\right)\,ds\\
	&+
	t^{-2}\int_1^t\left(s^{-2}\ln(s)\|\nabla_pg(s)\|_\infty+s^{-2}\ln(s)\|\nabla_p^2g(s)\|_\infty\right)\,ds\\
	\lesssim\ &
	1+ t^{-2} \int_1^t s^{-1}\ln(s) \,ds + \int_1^t s^{-3}\ln(s)\|\nabla_p\nabla_xg(s)\|_\infty \,ds\\
	&+
	t^{-2} \int_1^t s^{-2}\ln^3(s) \,ds + \int_1^ts^{-4}\ln(s)\|\nabla_p^2g(s)\|_\infty\,ds\\
	\lesssim\ &
	1+\int_1^ts^{-2}\ln(s)\left(s^{-1}\|\nabla_p\nabla_xg(s)\|_\infty +s^{-2}\|\nabla_p^2g(s)\|_\infty\right)\,ds\\
	\lesssim\ &
	1+\int_1^ts^{-2}\ln(s)\mcD_2(s)\,ds.
	\end{split}
	\end{equation}
Integrating \eqref{eq:pp-commut} with  $\partial_{m}\partial_{n}=\partial_{p_m}\partial_{x_n}$  along the characteristics of $\mcV_g$  and using the field estimates gives
	\begin{equation}
	\label{eq:gpx}
	\begin{split}
	t^{-1}\|&\nabla_p\nabla_xg(t)\|_{\infty}\\
	\lesssim\ &
	t^{-1}+t^{-1}\int_1^ts\left(s^{-3}\ln(s)\|\nabla_xg(s)\|_\infty+s^{-3}\ln(s)\|\nabla_p\nabla_xg(s)\|_\infty+s^{-2}\ln(s)\|\nabla_x^2g(s)\|_\infty\right)\,ds\\
	&+
	t^{-1}\int_1^t\left(s^{-3}\ln(s)\|\nabla_pg(s)\|_\infty+s^{-2}\ln(s)\|\nabla_p\nabla_xg(s)\|_\infty+s^{-3}\ln(s)\|\nabla_p^2g(s)\|_\infty\right)\,ds\\
	\lesssim\ &
	1+ t^{-1} \int_1^t s^{-2}\ln(s) \,ds + \int_1^t \left (s^{-3}\ln(s)\|\nabla_p\nabla_xg(s)\|_\infty+s^{-2}\ln(s)\|\nabla_x^2g(s)\|_\infty\right)\,ds\\
	&+
	t^{-1} \int_1^ts^{-3}\ln^3(s) \,ds +\int_1^t \left (s^{-3}\ln(s)\|\nabla_p\nabla_xg(s)\|_\infty+s^{-4}\ln(s)\|\nabla_p^2g(s)\|_\infty\right)\,ds\\
	\lesssim\ &
	1+
	\int_1^ts^{-2}\ln(s)\left(\|\nabla_x^2g(s)\|_\infty+s^{-1}\|\nabla_p\nabla_xg(s)\|_\infty+s^{-2}\|\nabla_p^2g(s)\|_\infty\right)\,ds\\
	\lesssim\ &
	1+ \int_1^ts^{-2}\ln(s)\mcD_2(s)\,ds.
	\end{split}
	\end{equation}
Integrating \eqref{eq:pp-commut} with  $\partial_{m}\partial_{n}=\partial_{x_m}\partial_{x_n}$  along the characteristics of $\mcV_g$  and estimating, we find
	\begin{equation}
	\label{eq:gxx}
	\begin{split}
	\|\nabla_x^2g(t)\|_{\infty}
	\lesssim\ &
	1+\int_1^ts\left(s^{-4}\ln(s)\|\nabla_xg(s)\|_\infty+s^{-3}\ln(s)\|\nabla_x^2g(s)\|_\infty\right)\,ds\\
	&+
	\int_1^t\left(s^{-4}\ln(s)\|\nabla_pg(s)\|_\infty+s^{-3}\ln(s)\|\nabla_p\nabla_xg(s)\|_\infty\right)\,ds\\
	\lesssim\ &
	1+\int_1^ts\left(s^{-4}\ln(s)+s^{-3}\ln(s)\|\nabla_x^2g(s)\|_\infty\right)\,ds\\
	&+
	\int_1^t\left(s^{-4}\ln^3(s)+s^{-3}\ln(s)\|\nabla_p\nabla_xg(s)\|_\infty\right)\,ds\\
	\lesssim\ &
	1+\int_1^ts^{-2}\ln(s)\left(\|\nabla_x^2g(s)\|_\infty+s^{-1}\|\nabla_p\nabla_xg(s)\|_\infty\right)\,ds\\
	\lesssim\ &
	1+\int_1^ts^{-2}\ln(s)\mcD_2(s)\,ds.
	\end{split}
	\end{equation}
Adding the estimates \eqref{eq:gpp2}, \eqref{eq:gpx}, and \eqref{eq:gxx} yields
$$ \mcD_2(t) \lesssim
	1+\int_1^ts^{-2}\ln(s)\mcD_2(s)\,ds,$$
and upon using Gr\"onwall's inequality, we conclude
$\mcD_2(t) \lesssim 1.$
This further implies 
$$ \|\nabla_x^2g(t)\|_\infty \lesssim 1, \qquad \|\nabla_p\nabla_xg(t)\|_\infty \lesssim t, \qquad \mathrm{and} \qquad  \|\nabla_p^2g(t)\|_\infty \lesssim t^2.$$
We use these bounds to improve the estimate \eqref{eq:gpx}. Indeed, returning to this inequality and inserting the above estimates gives
	\begin{align*}
	\|\nabla_p\nabla_xg(t)\|_{\infty}
	\lesssim\ &
	1+\int_1^ts\left(s^{-3}\ln(s)\|\nabla_xg(s)\|_\infty+s^{-3}\ln(s)\|\nabla_p\nabla_xg(s)\|_\infty+s^{-2}\ln(s)\|\nabla_x^2g(s)\|_\infty\right)\,ds\\
	&+
	\int_1^t\left(s^{-3}\ln(s)\|\nabla_pg(s)\|_\infty+s^{-2}\ln(s)\|\nabla_p\nabla_xg(s)\|_\infty+s^{-3}\ln(s)\|\nabla_p^2g(s)\|_\infty\right)\,ds\\
	\lesssim\ &
	1+\int_1^t\left(s^{-2}\ln(s)+s^{-2}\ln(s)\|\nabla_p\nabla_xg(s)\|_\infty+s^{-1}\ln(s)\right)\,ds\\
	&+
	\int_1^t\left(s^{-3}\ln^3(s)+s^{-2}\ln(s)\|\nabla_p\nabla_xg(s)\|_\infty+s^{-1}\ln(s)\right)\,ds\\
	\lesssim\ &
	\ln^2(t)+\int_1^ts^{-2}\ln(s)\|\nabla_p\nabla_xg(s)\|_\infty\,ds.\\
	\end{align*}
Using Gr\"onwall's inequality again yields
	\begin{align*}
	\|\nabla_p\nabla_xg(t)\|_{\infty}
	\lesssim
	\ln^2(t)\ \exp \left (\int_1^ts^{-2}\ln(s)\,ds \right )
	\lesssim
	\ln^2(t).
	\end{align*}
Using this estimate of the mixed partial derivatives within \eqref{eq:gpp1}, we have the improved estimate
	\begin{align*}
	\|\nabla_p^2g(t)\|_{\infty}
	\lesssim\ &
	1+\int_1^t \left [\left (s^3 \mcK_2(s) + s^2\mcK_1(s) + s\mcK_0(s) \right ) \|\nabla_xg(s)\|_\infty + \left ( s^2\mcK_1(s) + s\mcK_0(s) \right )\|\nabla_p\nabla_xg(s)\|_\infty\right]\,ds\\
	&+
	\int_1^t\left[ \left (s^2 \mcK_2(s) + s\mcK_1(s) + \mcK_0(s) \right )\|\nabla_pg(s)\|_\infty+ \left ( s\mcK_1(s) + \mcK_0(s) \right )\|\nabla_p^2g(s)\|_\infty\right]\,ds\\
	\lesssim\ &
	1+\int_1^t \left (s^3 \mcK_2(s) + s^2 \ln^2(s)\mcK_1(s) + s \ln^2(s)\mcK_0(s) \right )\,ds + \int_1^t s^{-2}\ln(s)\|\nabla_p^2g(s)\|_\infty\,ds\\
	\lesssim\ &
	\ln^4(t)+\int_1^ts^{-2}\ln(s)\|\nabla_p^2g(s)\|_\infty\,ds,
	\end{align*}
which leads to the final application of Gr\"onwall's inequality, namely
	\begin{align*}
	\|\nabla_p^2g(t)\|_{\infty}
	\lesssim
	\ln^4(t)\ \exp \left (\int_1^ts^{-2}\ln(s)\,ds \right )
	\lesssim
	\ln^4(t).
	\end{align*}
Hence, we find
$$\mcG_2(t) \lesssim \ln^4(t),$$
and the proof is complete.
\end{proof}

Next, we use the estimate of $\mcG_2(t)$ to establish the convergence rate of derivatives of the spatial average as $t \to \infty$.
\begin{lemma}
\label{DFunif}
The spatial average further satisfies
$\nabla_pF(t,p) \to \nabla_pF_\infty(p)$ uniformly as $t \to \infty$
with
$$\| \nabla_p F(t) - \nabla_pF_\infty \|_\infty \lesssim t^{-1}\ln^7(t).$$
\end{lemma}
\begin{proof}
To prove the derivative estimate, we proceed in a similar fashion to the estimate of $\partial_t F$ in Lemma \ref{Funif}.
Upon taking $\partial_{p_i}$ within \eqref{VMg}, integrating in $x$, and integrating by parts, we find 
\begin{eqnarray*}
\partial_t \int \partial_{p_i} g(t,x,p) \ dx & = & t  \int \partial_{p_i} \left (\mathrm{tr} \left [ \mathbb{A}(p) \nabla_x K(t, x+v(p)t, p) \right ] g(t,x,p) \right ) dx\\
& &  + \int  \partial_{p_i} \left (K(t, x+v(p)t,p) \cdot \nabla_p g(t,x,p) \right ) dx.
\end{eqnarray*}
Hence, using the estimates of Lemmas \ref{LPrelim}, \ref{Xsupp}, \ref{Dvg}, and \ref{D2g} yields 
\begin{eqnarray*}
\left | \partial_t \partial_{p_i} F(t,p) \right | & \lesssim & t \left (\mcK_1(t)  + t \mcK_2(t)\right )F(t,p) + \left (\mcK_0(t) + t\mcK_1(t) \right ) \mcG_1(t)\mcR(t)^3 + \mcK_0(t) \mcG_2(t)\mcR(t)^3 \\
& \lesssim & t^{-2} \ln(t) F(t,p) + t^{-2}\ln^6(t) + t^{-2}\ln^7(t)\\
& \lesssim & t^{-2}\ln^7(t).
\end{eqnarray*}
Of course, this implies that $ \| \partial_t \nabla_p F(t)\|_\infty$ is integrable,
and as before, establishes the existence of a limit $\mathfrak{F}$ with the estimate 
$$ \| \nabla_pF(t) - \mathfrak{F} \|_\infty  \lesssim t^{-1}\ln^7(t).$$
Finally, we find $\mathfrak{F}(p) = \nabla_p F_\infty(p)$ by the uniqueness of the uniform limit.
\end{proof}

Now, estimates of the large time behavior of density derivatives are needed in order to obtain a more refined understanding of the field behavior,
as spatial derivatives of $\rho$ and both spatial and time derivatives of $j$ appear as source terms in the inhomogeneous wave equations for $E$ and $B$ induced by Maxwell's equations.
As $F_\infty \in C_c^2(\bfR^3)$, we can define the limiting derivatives of $\rho$ and $j$ in terms of derivatives of $F_\infty$.
More specifically, we have for $q \in \Gamma_1$
$$\partial_{q_k} \rho_\infty(q) = \partial_{q_k} \left [ \mcD \left (v^{-1}(q) \right ) F_\infty \left (v^{-1}(q) \right )  \right ]$$
and
$$\partial_{q_k} j^\ell_\infty(q) = \partial_{q_k} \left [ q^\ell \mcD \left (v^{-1}(q) \right )  F_\infty \left (v^{-1}(q) \right )  \right ]$$
for every $k,\ell = 1,2,3$,
respectively.
Additionally, due to the compact support of $F_\infty$, these quantities are smoothly extended by 0 to $q \in \bfR^3$.
As within the previous section, in the case of multiple species the quantities $F(t,p)$ and $F_\infty(p)$ are replaced by sums over $\alpha$ of $F^\alpha(t,p)$ and $F^\alpha_\infty(p)$, and their respective derivatives, within the following lemmas.
Then, we have the following convergence estimates on derivatives.

\begin{lemma}
\label{LDensityderivative}
The spatial derivatives of the charge density satisfy 
$$\sup_{x \in \bfR^3} \left | t^4 \partial_{x_i}\rho(t,x) -  \partial_{q_i} \rho_\infty \left (\frac{x}{t} \right ) \right | \lesssim \|F(t) - F_\infty \|_\infty + \| \nabla_pF(t) - \nabla_p F_\infty\|_\infty + t^{-1} \mcR(t)^4 \mcG_2(t)$$
for every $i = 1,2,3$.
Furthermore, in view of the previous lemmas, we have
$$\sup_{x \in \bfR^3} \left | t^4 \partial_{x_i}\rho(t,x) -  \partial_{q_i} \rho_\infty \left (\frac{x}{t} \right ) \right | \lesssim t^{-1} \ln^8(t)$$
for every $i = 1,2,3$.
\end{lemma}

\begin{proof}
Where necessary, we use the Einstein summation convention. As in previous lemmas, it suffices to consider only $|x| \lesssim \gamma t$ due to the growth of the spatial support of $f$. 
We begin by taking a derivative with respect to $x_i$ within \eqref{rhop} to find
$$\partial_{x_i}\rho(t,x) 
= t^{-4}\int \mathbb{B}_{ik}(v(r)) \partial_{r_k} \left [ \mcD  \left ( r \right )  g  \left (t, y, r\right ) \right ] \biggr |_{r = v^{-1} \left (\frac{x-y}{t} \right )} dy.$$
%
%
Furthermore, we have
$$\partial_{q_i} \rho_\infty(q) = \partial_{q_i} \left [ \mcD \left ( v^{-1}(q) \right ) F_\infty\left ( v^{-1}(q) \right ) \right ] =  \mathbb{B}_{ik}(v(r))  \partial_{r_k} \left [ \mcD(r) F_\infty(r) \right ] \biggr |_{r = v^{-1}(q)}.$$
Letting
$$\mcH(t,y,p) = \mathbb{B}_{ik}(v(p))\partial_{p_k} \left [ \mcD(p) g(t, y, p) \right ]$$
and subtracting the above expressions for the charge density and its expected limit we find
$$t^4\partial_{x_i}\rho(t,x)- \partial_{q_i}\rho_\infty\left( \frac{x}{t}\right) = \int \mcH \left (t,y, v^{-1}\left (\frac{x-y}{t} \right ) \right ) \ dy - \mathbb{B}_{ik}(v(p)) \partial_{p_k} \left [ \mcD(p)F_\infty(p) \right ] \biggr |_{p = v^{-1}\left (\frac{x}{t} \right) }.$$
We split this into two pieces so that
$$ \left | t^4\partial_{x_i}\rho(t,x) - \partial_{q_i}\rho_\infty\left( \frac{x}{t}\right) \right | \leq I + II$$
where
$$I  = \int \left | \mcH \left (t,y, v^{-1}\left (\frac{x-y}{t} \right ) \right ) -\mcH \left (t,y, v^{-1}\left (\frac{x}{t} \right ) \right )\right | \ dy$$
and
$$II = \left | \mathbb{B}_{ik} (v(p) ) \right | \  \left | \partial_{p_k} \left [ \mcD(p) \left (F(t, p) - F_\infty(p) \right ) \right ] \right | \biggr |_{p= v^{-1} \left (\frac{x}{t} \right )}$$
as $\int g(t,y,p) \ dy = F(t,p)$.
We estimate $I$ using Lemma \ref{LH}.
As the time-independent portions of $\mcH(t,y,p)$ and their derivatives are uniformly bounded, we find
$\Vert \nabla_p \mcH(t) \Vert_\infty \lesssim \mcG_2(t)$.
Combining this with the estimates of Lemmas \ref{Xsupp} and \ref{D2g} then gives
$$I \lesssim  t^{-1} \mcR(t)^4  \mcG_2(t) \lesssim t^{-1} \ln^8(t).$$
The estimate of $II$ uses $\left | \frac{x}{t} \right | \lesssim \gamma < 1$, the compact support of $F(t)$ and $F_\infty$, and Lemmas \ref{Funif} and \ref{DFunif}, to conclude
$$ II \lesssim \|F(t) - F_\infty \|_\infty + \| \nabla_pF(t) - \nabla_pF_\infty \|_\infty \lesssim t^{-1} \ln^7(t).$$
Combining these estimates then yields the result.
\end{proof}

\begin{lemma}
\label{LCurrentderivative}
The spatial derivatives of the current density satisfy  
$$\sup_{x\in \bfR^3} \left | t^4 \partial_{x_i}j(t,x) -  \partial_{q_i}j_\infty\left (\frac{x}{t} \right) \right | \lesssim \|F(t) - F_\infty \|_\infty + \| \nabla_pF(t) - \nabla_p F_\infty\|_\infty + t^{-1} \mcR(t)^4  \mcG_2(t)$$
for every $i = 1,2,3$.
Furthermore, in view of the previous lemmas, we have
$$\sup_{x\in \bfR^3} \left | t^4 \partial_{x_i}j(t,x) -  \partial_{q_i}j_\infty\left (\frac{x}{t} \right) \right |  \lesssim t^{-1} \ln^8(t)$$
for every $i = 1,2,3$.
Additionally, the time derivative of the current density satisfies
$$\sup_{x\in \bfR^3} \left | t^4 \partial_tj(t,x) +\left [3j_\infty\left (\frac{x}{t} \right) + \frac{x}{t} \cdot \nabla_qj_\infty\left (\frac{x}{t} \right )\right ] \right | \lesssim 
\|F(t) - F_\infty \|_\infty + \| \nabla_pF(t) - \nabla_p F_\infty\|_\infty + t^{-1} \mcR(t)^4 \mcG_2(t).$$
and in view of the previous lemmas, we have
$$\sup_{x\in \bfR^3} \left | t^4 \partial_tj(t,x) +\left [3j_\infty\left (\frac{x}{t} \right) + \frac{x}{t} \cdot \nabla_qj_\infty\left (\frac{x}{t} \right )\right ] \right | \lesssim t^{-1} \ln^8(t).$$
\end{lemma}
\begin{proof}
Where necessary, we use the Einstein summation convention.
As before, we need only consider $|x| \lesssim \gamma t$ due to the growth of the support of $f$. 
To prove the first result, we begin by rewriting the $\ell$th component of \eqref{jp} as
$$j^\ell(t,x) = \frac{1}{t^3}  \int v^\ell(r) \mcD(r) g\left (t, y, r \right ) \biggr |_{r = v^{-1}\left (\frac{x-y}{t} \right )} dy.$$
Then, taking an $x_i$ derivative within this expression yields
\begin{equation}
\label{nabla_j}
\partial_{x_i}j^\ell(t,x)  = t^{-4}\int \mathbb{B}_{ik}\left (v(r) \right ) \partial_{r_k} \left [ v^\ell(r) \mcD(r) g \left (t, y, r \right )  \right ] \biggr |_{r = v^{-1}\left (\frac{x-y}{t} \right )} dy.
\end{equation}
%
%
Furthermore, we have
$$\partial_{q_i} j^\ell_\infty(q) = \partial_{q_i} \left [ v^\ell\left ( v^{-1}(q) \right ) \mcD \left( v^{-1}(q) \right ) F_\infty\left ( v^{-1}(q) \right ) \right ] =  \bigg(\mathbb{B}_{ik}(v(r))  \partial_{r_k} \left [ v^\ell(r) \mcD(r) F_\infty(r) \right ] \bigg)\biggr |_{r = v^{-1}(q)}.$$

Next, we let
$$\mcH(t,y,p) = \mathbb{B}_{ik}(v(p)) \partial_{p_k} \left [ v^\ell(p) \mcD(p) g(t, y, p) \right ],$$
and subtract the derivative terms to find
$$t^4\partial_{x_i}j^\ell(t,x)  - \partial_{q_i}j^\ell_\infty\left( \frac{x}{t}\right) = \int \mcH \left (t,y, v^{-1}\left (\frac{x-y}{t} \right ) \right ) dy - \mathbb{B}_{ik}(v(p))
\partial_{p_k} \left [ v^\ell(p) \mcD(p) F_\infty(p) \right ] \biggr |_{p = v^{-1}\left (\frac{x}{t} \right) }.$$
Similar to the previous lemma, we split this into two pieces so that
$$ \left | t^4\partial_{x_i}j^\ell(t,x)  - \partial_{q_i}j^\ell_\infty\left( \frac{x}{t}\right) \right | \leq I + II$$
where
$$I  = \int \left | \mcH \left (t,y, v^{-1}\left (\frac{x-y}{t} \right ) \right )  - \mcH \left (t,y, v^{-1}\left (\frac{x}{t} \right ) \right ) \right | \ dy$$
and
$$II = \left | \mathbb{B}_{ik}\left (v(p) \right ) \right | \  \left | \partial_{p_k}\left [ v^\ell(p)\mcD(p) \left ( F(t, p) - F_\infty(p) \right ) \right ] \right | \biggr |_{p= v^{-1} \left (\frac{x}{t} \right )}$$
as $\int g(t,y,p) \ dy = F(t,p)$.
As in previous lemmas, we estimate $I$ using Lemma \ref{LH}.
Because the time-independent portions of $\mcH(t,y,p)$ and their derivatives are uniformly bounded, this
yields the estimate
$\Vert \nabla_p \mcH(t) \Vert_\infty \lesssim \mcG_2(t)$.
Combining this with the estimates of Lemmas \ref{LH}, \ref{Xsupp}, and \ref{D2g} then gives
$$I \lesssim  t^{-1} \mcR(t)^4 \mcG_2(t) \lesssim t^{-1} \ln^8(t)$$
as $v(p)$, $\mcD(p)$, and all of their derivatives are uniformly bounded.
The term $II$ uses $\left | \frac{x}{t} \right | \lesssim \gamma < 1$, the compact support of $F(t)$ and $F_\infty$, and Lemmas \ref{Funif} and \ref{DFunif}  so that 
$$ II \lesssim \|F(t) - F_\infty \|_\infty + \| \nabla_pF(t) - \nabla_pF_\infty \|_\infty \lesssim t^{-1} \ln^7(t).$$
The combination of these estimates yields the first result.

Turning to the second result, we perform a similar computation and invoke the first result where needed.
As before, we write the $\ell$th component of the current density as
$$j^\ell(t,x) = \frac{1}{t^3}  \int v^\ell(r) \mcD(r) g\left (t, y, r \right ) \biggr |_{r = v^{-1}\left ( \frac{x-y}{t} \right )} dy.$$
Taking a $t$ derivative, we find
\begin{align*}
\partial_tj^\ell(t,x)  &= -3 t^{-1} j^\ell(t,x) - t^{-4}\int  \frac{x_i-y_i}{t}\mathbb{B}_{ik}\left (v(r) \right ) \partial_{r_k} \left [ v^\ell(r)\mcD(r) g \left (t, y, r \right ) \right ] \biggr |_{r = v^{-1}\left (\frac{x-y}{t} \right )} dy\\
& \qquad + t^{-3}  \int v^\ell(r) \mcD(r) \partial_t g\left (t, y, r \right ) \biggr |_{r = v^{-1}\left (\frac{x-y}{t} \right )} dy\\
&= -3 t^{-1} j^\ell(t,x) - \frac{x}{t} \cdot \nabla_x j^\ell(t,x)\\
& \qquad + t^{-4}\int  \frac{y_i}{t}\mathbb{B}_{ik}\left (v(r) \right ) \partial_{r_k} \left [ v^\ell(r)\mcD(r) g \left (t, y, r \right ) \right ] \biggr |_{r = v^{-1}\left (\frac{x-y}{t} \right )} dy\\
& \qquad +  t^{-3}  \int v^\ell(r) \mcD(r) \partial_t g\left (t, y, r \right ) \biggr |_{r = v^{-1}\left (\frac{x-y}{t} \right )} dy
\end{align*}
due to \eqref{nabla_j}.
Multiplying by $t^4$, we expect the first term on the right side to converge to its previously established limit, while the second term will converge to a limit involving the gradient of $j_\infty$, and the last two terms are higher order.
Indeed, the difference we wish to estimate becomes
\begin{align*}
& \left | t^4  \partial_{t} j^\ell(t,x) + 3j_\infty^\ell\left (\frac{x}{t} \right ) + \frac{x}{t} \cdot \nabla_qj^\ell_\infty\left (\frac{x}{t} \right )  \right |\\
& \qquad \leq 3 \left | t^{3} j^\ell(t,x) - j_\infty^\ell\left (\frac{x}{t} \right ) \right |  + \left | \frac{x}{t} \cdot \left ( t^4 \nabla_xj^\ell(t,x) - \nabla_qj^\ell_\infty\left (\frac{x}{t} \right )  \right ) \right |  \\
& \qquad \qquad + \left |  \int \frac{y_i}{t}\mathbb{B}_{ik}\left (v(r) \right ) \partial_{r_k} \left [ v^\ell(r)\mcD(r) g \left (t, y, r \right )  \right ] \biggr |_{r = v^{-1}\left (\frac{x-y}{t} \right )} dy  \right |\\
& \qquad \qquad + t \left |\int v^\ell(r) \mcD(r) \partial_t g\left (t, y, r \right ) \biggr |_{r = v^{-1}\left (\frac{x-y}{t} \right )} dy  \right |\\
& \qquad =: I - IV.
\end{align*}

To estimate $I$, we use Lemma \ref{LCurrent} to find
\begin{eqnarray*}
I & \lesssim & \sup_{x \in \bfR^3} \left | t^3 j(t,x) - j_\infty \left ( \frac{x}{t} \right ) \right |\\
& \lesssim &  \| F(t) - F_\infty \|_\infty + t^{-1}\mcR(t)^4 \mcG_1(t)  \\
& \lesssim & t^{-1} \ln^6(t).
\end{eqnarray*}
Next, we use the first result of the current lemma for $II$, yielding
\begin{eqnarray*}
II & \lesssim & \gamma \sup_{x \in \bfR^3} \left | t^4 \nabla_xj^\ell(t,x) - \nabla_q j^\ell_\infty\left (\frac{x}{t} \right )  \right |\\
& \lesssim & \|F(t) - F_\infty \|_\infty + \| \nabla_pF(t) - \nabla_p F_\infty\|_\infty + t^{-1} \mcR(t)^4  \mcG_2(t)\\
& \lesssim & t^{-1} \ln^8(t).
\end{eqnarray*}
Within $III$, the derivatives appearing in the integrand are all dominated by $\mcG_1(t)$, and thus
$$III \lesssim t^{-1} \mcR(t)^4 \mcG_1(t) \lesssim t^{-1}\ln^6(t).$$

Finally, $IV$ is the most challenging term.
We estimate the integral and use \eqref{VMg} to write
$$\partial_{t}g(t,x,p) =  t \mathbb{A}(p)K(t,x+v(p)t, p)\cdot \nabla_{x}g(t,x,p) -K(t,x+v(p)t, p) \cdot\nabla_{p}g(t,x,p).$$
This expression is inserted into the integral but evaluated at the point $(t,y,r)$ where $r = v^{-1}\left (\frac{x-y}{t} \right )$ so that
\begin{eqnarray*}
\partial_{t}g(t,y,r) & = & t \mathbb{A}(r)K(t,y+v(r)t, r)\cdot \nabla_{x}g(t,y,r) -K(t,y+v(r)t, r) \cdot\nabla_{p}g(t,y,r)\\
& = & t \mathbb{A}(r)K(t,x, r)\cdot \nabla_{x}g(t,y,r) -K(t,x, r) \cdot\nabla_{p}g(t,y,r).
\end{eqnarray*}
We will integrate by parts in the $y$ variable, thus we compute
$$\partial_{y_i} \left [ g(t,y,r) \right ] = \partial_{x_i}g(t,y,r) - t^{-1} \mathbb{B}_{ik}(v(r)) \partial_{p_k} g(t,y,r)$$
so that
$$\partial_{x_i}g(t,y,r) = \partial_{y_i} \left [ g(t,y,r) \right ] + t^{-1} \mathbb{B}_{ik}(v(r))\partial_{p_k} g(t,y,r).$$
Replacing $\partial_t g(t,y,r)$, we see that $IV$ can be split into
\begin{align*}
IV & \leq t^2\left |\int v^\ell(r) \mcD(r)\mathbb{A}(r)K(t,x, r)\cdot \nabla_{x}g(t,y,r) \biggr |_{r = v^{-1}\left (\frac{x-y}{t} \right )} dy  \right |\\
& \qquad + t \left |\int v^\ell(r) \mcD(r) K(t,x, r) \cdot\nabla_{p}g(t,y,r) \biggr |_{r = v^{-1}\left (\frac{x-y}{t} \right )} dy  \right |\\
&   \leq t^2\left |\int v^\ell(r) \mcD(r) \mathbb{A}_{ij}(r) K^j(t,x, r)\partial_{y_i} \left [ g(t,y,r) \right ] \biggr |_{r = v^{-1}\left (\frac{x-y}{t} \right )} dy  \right |\\
& \qquad  + t\left |\int v^\ell(r) \mcD(r) \mathbb{A}_{ij}(r) K^j(t,x, r)\mathbb{B}_{ik}(v(r)) \partial_{p_k} g(t,y,r) \biggr |_{r = v^{-1}\left (\frac{x-y}{t} \right )} dy  \right |\\
& \qquad  + t \left |\int v^\ell(r) \mcD(r) K(t,x, r) \cdot\nabla_{p}g(t,y,r) \biggr |_{r = v^{-1}\left (\frac{x-y}{t} \right )} dy  \right |\\
& =: IV_A + IV_B + IV_C.
\end{align*}
Within $IV_A$, we integrate by parts in $y$, and as every $y_i$ derivative of $r = v^{-1}\left (\frac{x-y}{t} \right )$ provides an additional factor of $t^{-1}$, Lemma \ref{LPrelim} produces the estimate
\begin{align*}
IV_A & = t^2\left |\int  \partial_{y_i} \left [ v^\ell(r) \mcD(r)\mathbb{A}_{ij}(r)K^j(t,x, r)  \biggr |_{r = v^{-1}\left (\frac{x-y}{t} \right )}\right ] g\left (t,y,v^{-1}\left (\frac{x-y}{t} \right ) \right) dy  \right |\\
& \lesssim t \int_{|y| \lesssim \mcR(t)} \left | K(t,x, r) \right | g(t,y,r) \biggr |_{r = v^{-1}\left (\frac{x-y}{t} \right )} dy\\
& \lesssim t^{-1} \ln^3(t).
\end{align*}
The term $IV_B$ is more straightforward as the boundedness of the functions of $r$ within the integrand and Lemmas \ref{LPrelim} and \ref{Dvg} give
$$IV_B \lesssim t \mcG_1(t) \int_{|y| \lesssim \mcR(t)} \left | K \left (t,x, v^{-1}\left (\frac{x-y}{t} \right ) \right ) \right | dy \lesssim t^{-1} \ln^5(t).$$
Finally, the estimate for $IV_C$ is nearly identical to that of $IV_B$, and we similarly find
$$IV_C \lesssim t \mcG_1(t)\mcR(t)^3  \mcK_0(t) \lesssim t^{-1} \ln^5(t).$$
Therefore, we have 
$$IV \lesssim t^{-1} \ln^5(t),$$
and further collecting the estimates for $I$, $II$, and $III$ yields the second result.
\end{proof}

Finally, we present a pivotal lemma that will be used to estimate kernel terms appearing in second order field derivatives within the proof of the small data theorem.

\begin{lemma}
\label{Lm}
Let $m : \partial\Gamma_1 \times \Gamma_1 \to \mathbb{R}$ be bounded on the support of $f$ and continuously differentiable in its second argument with bounded derivatives. Define
$$ M(t,\omega, y) = \int m(\omega, v(p)) \partial_ {x_k} f(t,y,p) \ dp$$
for any of $k = 1, 2, 3$.
Then, for all $|\omega | = 1$ and $|y| \lesssim \gamma t$  we have
$$|M(t,\omega, y) | \lesssim t^{-4}.$$
\end{lemma}
\begin{remark}
\label{Lm_rmk}
For the proof of the small data theorem, we will apply this result to the iterates.
Hence, if in addition to the stated assumptions we have $\Vert f_0 \Vert_{C^1} \leq \epsilon_0$, then the estimate becomes
$$|M(t,\omega, y) | \leq C\epsilon_0 (1+t)^{-4} \leq C\epsilon_0 (t + |y| + 2L)^{-4}$$
for all $t \geq 0$, $|\omega | = 1$, and $y \in \bfR^3$ due to the bound on the support of $f$ and $\partial_{x_k} f$.
\end{remark}
\begin{proof}
Again, we use the Einstein summation convention where necessary.
We first rewrite the integral as
$$ M(t,\omega, y) = \int m(\omega, v(p)) \partial_ {x_k} g(t,y - v(p)t,p) \ dp$$
and conduct a change of variables
$$z =  y - v(p)t$$
with
$$dz =  t^3 \left | \det(\mathbb{A}(p)) \right | dp.$$
Recall that 
$\mcD(p) = \left | \det(\mathbb{A}(p)) \right |^{-1} = \left (1+ |p|^2 \right )^{5/2}$ due to \eqref{Ddef}.
Hence, we have
$$ M(t,\omega, y) = t^{-3} \int m \left (\omega, v(r) \right ) \partial_ {x_k} g \left (t,z, r \right ) \mcD \left (r \right ) \biggr |_{r =  v^{-1} \left (\frac{y-z}{t} \right )}  \ dz.$$

Next, we note that
$$ \partial_ {x_k} g \left (t,z, v^{-1} \left (\frac{y-z}{t} \right ) \right ) =  \frac{d}{dz_k} \left [ g \left (t,z, v^{-1} \left (\frac{y-z}{t} \right ) \right ) \right ] + t^{-1} \mathbb{B}_{k\ell}\left (\frac{y-z}{t} \right ) \partial_{p_\ell}  g \left (t,z, v^{-1} \left (\frac{y-z}{t} \right ) \right ).$$
Hence, replacing the $x_k$-derivative of $g$ within the above integral splits it into two terms, and we find
\begin{eqnarray*}
M(t,\omega, y) & = & t^{-3}\int  \frac{d}{dz_k} \left [ g \left (t,z, v^{-1} \left (\frac{y-z}{t} \right ) \right ) \right ] m \left (\omega, \frac{y-z}{t}  \right ) \mcD \left (v^{-1} \left ( \frac{y-z}{t}  \right ) \right )  \ dz\\
& & +  \ t^{-4} \int  \mathbb{B}_{k\ell}\left (v(r) \right ) \partial_{p_\ell}  g \left (t,z, r \right ) m \left (\omega, v(r) \right ) \mcD \left (r  \right ) \biggr |_{r =  v^{-1} \left (\frac{y-z}{t} \right )} \ dz.
\end{eqnarray*}
Integrating by parts in the first term yields
\begin{eqnarray*}
& & \int  \frac{d}{dz_k} \left [ g \left (t,z, v^{-1} \left (\frac{y-z}{t} \right ) \right ) \right ] m \left (\omega, \frac{y-z}{t}  \right ) \mcD \left (v^{-1} \left ( \frac{y-z}{t}  \right ) \right )  \ dz\\
& & = t^{-1} \int g \left (t,z, r \right ) \left [ \partial_{v_k} m \left (\omega, v(r)  \right ) + \mathbb{B}_{k\ell}  \left ( v(r)  \right ) \partial_{p_\ell}  \mcD \left (r \right )  \right ]  \biggr |_{r =  v^{-1} \left (\frac{y-z}{t} \right )}   \ dz.
\end{eqnarray*}
Noting that
$$\partial_{p_\ell} \left [ m(\omega, v(p)) \right ] = \partial_{v_k} m\left (\omega, v(p)  \right ) \mathbb{A}_{k\ell}(p),$$
using the symmetry of $\mathbb{A}$ and $\mathbb{B}$ along with the identity \eqref{AB}, 
and assembling this with the remaining term in the representation for $M(t,\omega, y)$ gives
$$M(t,\omega, y) =  t^{-4}\int \mathbb{B}_{k \ell}(v(r)) \partial_{r_\ell} \left [ g m \mcD \right ] \left (t,\omega,z,r \right )  \biggr |_{r =  v^{-1} \left (\frac{y-z}{t} \right )}  \ dz.$$
Therefore, we define
$$M_\infty(\omega, p) = \mathbb{B}_{k \ell}(v(p)) \partial_{p_\ell} \left [ F_\infty(p) m \left (\omega, v(p) \right ) \mcD \left (p \right ) \right ]$$
and note that this quantity is uniformly bounded due to the compact support of $F_\infty$.
Then, we denote
$$\mcH(t,\omega, z, p) = \mathbb{B}_{k \ell}(v(p)) \partial_{p_\ell} \left [ g \left (t,z,p \right ) m \left (\omega, v(p) \right ) \mcD \left (p \right ) \right ]$$
and estimate the difference
$$\left | t^4 M(t, \omega, y) - M_\infty \left (\omega, \frac{y}{t} \right ) \right |
=  
\biggl | \int \mcH \left (t, \omega, z,  v^{-1} \left (\frac{y-z}{t} \right ) \right )  \ dz
- \left (\mathbb{B}_{k \ell} \partial_{p_\ell} \left [  m \mcD F_\infty \right ] \right ) \left (\omega, \frac{y}{t} \right ) \biggr |.$$
This can be rewritten as 
$$\left | t^4 M(t, \omega, y) - M_\infty \left (\omega, \frac{y}{t} \right ) \right | \leq I + II$$
where
$$I = \int \left | \mcH \left (t, \omega, z,  v^{-1} \left (\frac{y-z}{t} \right ) \right ) - \mcH \left (t, \omega, z,  v^{-1} \left (\frac{y}{t} \right ) \right )\right | \ dz.$$
and
$$II = \left | \left (\mathbb{B}_{k \ell} \partial_{p_\ell} \left [ m \mcD \left ( F - F_\infty \right ) \right ] \right ) \left (t, \omega, v^{-1} \left (\frac{y}{t} \right ) \right )  \right | $$
as
$$\int \mcH \left (t, \omega, z,  p \right ) \ dz = \mathbb{B}_{k \ell}(v(p)) \partial_{p_\ell} \left [ F \left (t,p \right ) m \left (\omega, v(p) \right ) \mcD \left (p \right ) \right ].$$

We estimate $I$ merely by fixing $\omega \in S^1$ and invoking Lemma \ref{LH} to find
$$I \lesssim  t^{-1} \mcR(t)^4 \Vert \nabla_p \mcH(t) \Vert_\infty.$$
As $m(\omega, v(p))$, $\mcD(p)$, their $p$-derivatives, and $\Vert g(t) \Vert_\infty$ are all uniformly bounded, we have
$\Vert \nabla_p \mcH(t) \Vert_\infty \lesssim \mcG_2(t),$ 
and combining this with the estimates of Lemmas \ref{Xsupp} and \ref{Dvg} gives
$$I \lesssim  t^{-1} \mcR(t)^4\mcG_2(t)  \lesssim t^{-1} \ln^7(t).$$
The estimate of $II$ uses Lemmas \ref{Funif} and \ref{DFunif} and the compact support of $F(t)$ and $F_\infty$ to find
$$II \lesssim \| F(t) - F_\infty \|_\infty +  \| \nabla_p F(t) - \nabla_p  F_\infty \|_\infty \lesssim t^{-1} \ln^8(t).$$

Combining these estimates then yields 
$$\left | t^4 M(t, \omega, y) - M_\infty \left (\omega, \frac{y}{t} \right ) \right | \lesssim t^{-1} \ln^8(t).$$
Finally, using this inequality and the boundedness of $M_\infty(\omega, p)$ gives
$$\left |M(t, \omega, y) \right | \lesssim  t^{-4} \left |M_\infty \left (\omega, \frac{y}{t} \right ) \right | + \left | M(t, \omega, y) - t^{-4} M_\infty \left ( \omega, \frac{y}{t} \right ) \right | \lesssim  t^{-4} \left |M_\infty \left (\omega, \frac{y}{t} \right ) \right | + t^{-5} \ln^8(t) \lesssim t^{-4},$$
and yields the stated result.
\end{proof}

\section{Field Limits and Modified Scattering}
\label{sect:fieldconv}


\subsection{Convergence of the Electric and Magnetic Fields}

With the large time asymptotic behavior established for the charge density, current density, and their derivatives, we can further generate the induced electric and magnetic fields as $t \to \infty$ via self-similar solutions of inhomogeneous wave equations. In order to accomplish this, we will prove a lemma relating the asymptotic behavior of solutions to the inhomogeneous wave equation in terms of the associated forcing function.

First, we state a change of variables lemma from \cite{Glassey} that will allow us to estimate complicated integrals, both in the next lemma and within the next section, in which Theorem \ref{T0} is proved.
\begin{lemma}\cite[Lemma 6.5.2 (p. 176)]{Glassey}
\label{LGS_CV}
For any continuous function $\Phi(\tau, \lambda)$ of two real variables, and $\Psi(\sigma)$ of one real variable, we have
$$\int_{|y-z| \leq s} \Phi(s - |y-z|, |z|) \Psi(|y-z|) \ dz = \frac{2\pi}{|y|} \int_0^s \int_{|s - |y| - \tau|}^{s+ |y| - \tau} \Phi(\tau, \lambda) \lambda d\lambda (s-\tau) \Psi(s-\tau) d\tau,$$
where the integration on the left is over a ball in $\mathbb{R}^3$.
\end{lemma}

With this, we can establish the asymptotic behavior of solutions to the inhomogeneous wave equation.

\begin{lemma}
\label{LWave}
Let $\eta \in C^1\left( [0,\infty) \times \mathbb{R}^3 \right )$ be supported on $\{ (t,x) : |x| \leq \zeta t + L\}$ 
where $\zeta$ is defined within Lemma \ref{LPrelim}.
Assume that for $|x| \lesssim \gamma t$, the function $\eta$ satisfies the estimate
$$ \left | t^4 \eta(t,x) - \eta_\infty \left(\frac{x}{t}\right) \right | \lesssim t^{-1}\ln^8(t)$$
where $\eta_\infty \in C^1(\Gamma_1)$ has compact support.
Let $\psi(t,x)$ be the unique solution of
\begin{equation}
\label{wavepsi}
\Box_{t,x} \psi = \eta(t,x)
\end{equation}
with initial conditions
$$ \psi(0,x) = \psi_0(x)$$
$$ \partial_t\psi(0,x) = \psi_1(x),$$
both of which are supported in $\{ |x| \leq L\}$.
Then, there exists $\psi_\infty \in C^{2,\delta}(\Gamma_1)$ for any $\delta \in (0,1)$ such that
$$ \left | t^2\psi(t,x) - \psi_\infty \left ( \frac{x}{t} \right ) \right | \lesssim t^{-1}\ln^8(t)$$
for all $|x| \lesssim \gamma t$.
Additionally, if $\eta_\infty \equiv 0$, then $\psi_\infty \equiv 0$.
\end{lemma}
\begin{proof}
We begin by constructing the limiting solution of the wave equation, $\psi_\infty$.
In particular, note that by assumption, the limiting profile of $\eta$ (denoted $\eta_\infty$) possesses a self-similar structure.
Therefore, we will find a solution of the wave equation with this same property.
In this direction, we fix $T > 0$ such that for $t \geq T$, we have
$\zeta t + L \leq \gamma t$ and thus
$\eta(t,x)$ is supported on $\{ (t,x) : |x| \leq \gamma t\}$ for $t \geq T$ (see Figure \ref{fig:lightcone}).
	\begin{figure}
	\includegraphics[width=10cm,height=5cm]{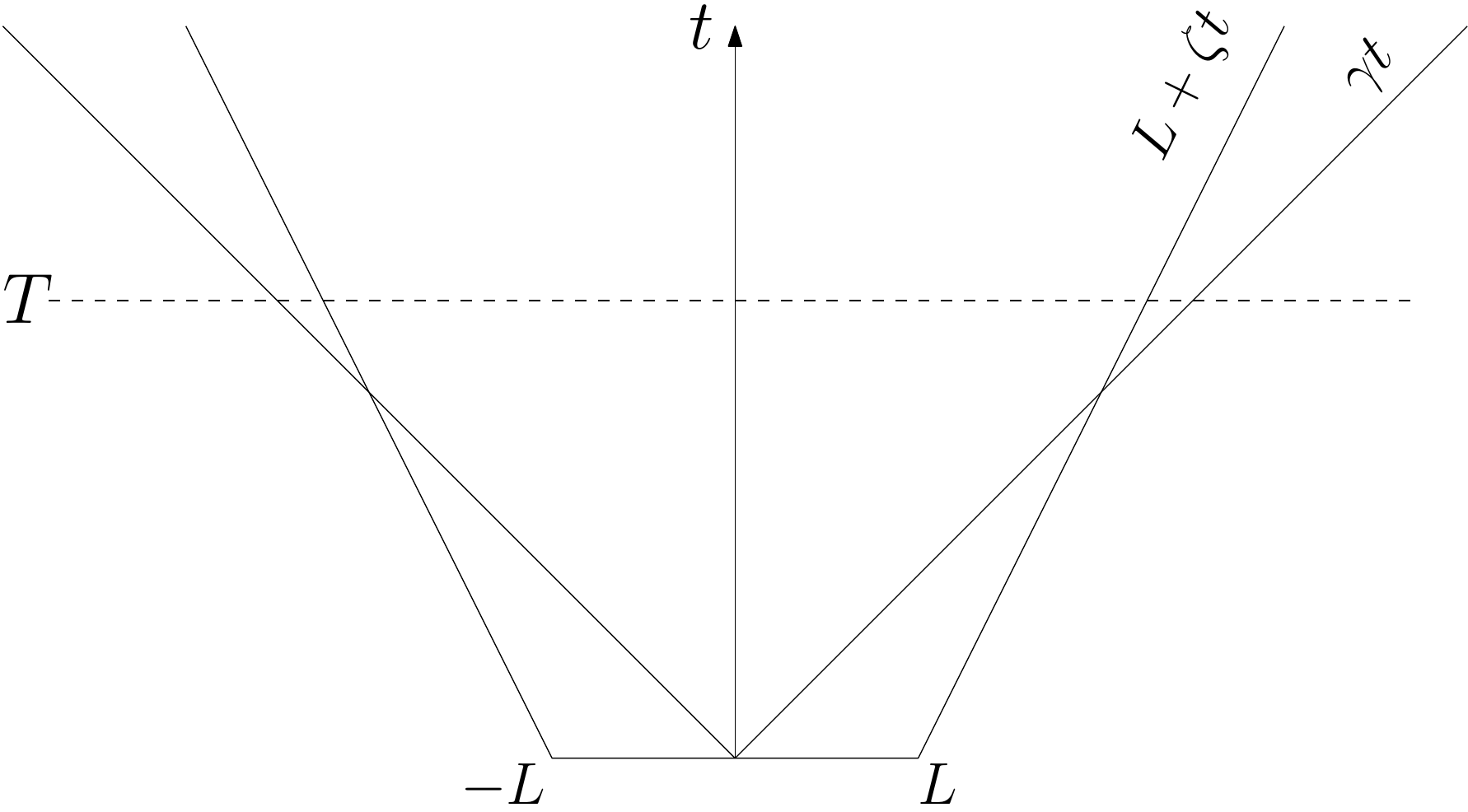}
	\caption{Comparison of the sets enclosed by $L+\zeta t$ and $\gamma t$ where $0<\zeta<\gamma<1$.}\label{fig:lightcone}
	\end{figure}
Define the linear operator 
$$\mathcal{L} u := \sum_{i,j=1}^3 \left (y_i y_j - \delta_{ij} \right ) \partial_{y_i y_j} u + 6y \cdot \nabla_y u + 6u$$
and note that for $|y| < \gamma < 1$
$$ - \sum_{i,j=1}^3 \left ( y_i y_j - \delta_{ij} \right )\xi_i \xi_j = |\xi|^2 - |y \cdot \xi|^2 > (1 - \gamma^2) |\xi|^2 = C |\xi|^2$$
where $C > 0$.
Thus, the operator is uniformly elliptic for sufficiently smooth functions defined on $\{ y \in \mathbb{R}^3 : |y| < \gamma \}$.
Further, as $\eta_\infty \in C^1$, we define $\psi_\infty \in C^{2, \delta}$ for any $\delta \in (0,1)$ to be the unique solution of the uniformly elliptic boundary-value problem \cite{GT}
$$ \mathcal{L} \psi_\infty(y) = \eta_\infty(y)$$
for $|y| < \gamma$ with the boundary condition
$$\psi_\infty(y) = 0 \quad \mathrm{for} \ |y| = \gamma.$$
Due to the boundary condition, \alt{we merely} extend $\psi_\infty(y) = 0$ for $|y| > \gamma$ to preserve continuity \alt{and define the function for all $y \in \bfR^3$. Further,} 
$\psi_\infty$ will remain $C^{2, \delta}$ on the interior of the sphere of radius $|y| = \gamma$.
Note that $\eta_\infty \equiv 0$ implies $\psi_\infty \equiv 0$ due to uniqueness.
Then, a brief calculation shows
\begin{equation}
\label{psi_wave}
\Box_{t,x} \left [ t^{-2} \psi_\infty \left ( \frac{x}{t} \right )  \right ] = t^{-4} \left (\mathcal{L} \psi_\infty \right )\left ( \frac{x}{t} \right ) = t^{-4} \eta_\infty\left ( \frac{x}{t}  \right )
\end{equation}
for $|x| < \gamma t$.
%

%
With this, we let $\psi(t,x)$ be the unique classical solution of \eqref{wavepsi} with the given initial conditions $\psi_0$ and $\psi_1$ at $t=0$. As we are interested in the values of this solution within the region $|x| \leq \gamma t$ for $t \geq T$, we recast this as a forward-shifted initial-value problem for $t \geq T$ with ``initial'' data at $t = T$ and estimate within this region of space-time.
Using superposition, we may decompose the solution as
$$\psi(t,x) = \psi_H(t,x)  + \psi_I(t,x)  + \psi_{II}(t,x)$$
for $|x| < \gamma t$
where (i) $\psi_H$ is a homogeneous solution satisfying
\begin{eqnarray*}
& \Box_{t,x} \psi_H = 0 & \mathrm{for} \  t > T,  \alt{x \in \bfR^3}\\
& \psi_H(T,x) = \psi(T,x) - T^{-2} \psi_\infty \left (\frac{x}{T} \right ) =: \widetilde{\psi}_0(x) & \mathrm{for} \ \alt{x \in \bfR^3}\\
& \partial_t \psi_H(T,x) = \partial_t\psi(T,x) - \partial_t \left (t^{-2} \psi_\infty \left (\frac{x}{t} \right ) \right ) \biggr |_{t = T} =: \widetilde{\psi}_1(x) \quad &  \mathrm{for} \  \alt{x \in \bfR^3},
\end{eqnarray*}
(ii) the function $\psi_I$ satisfies
\begin{eqnarray*}
& \Box_{t,x} \psi_I = t^{-4} \eta_\infty \left (\frac{x}{t} \right ) & \mathrm{for} \  t > T,  \alt{x \in \bfR^3}\\
& \psi_I(T,x) = T^{-2} \psi_\infty \left (\frac{x}{T} \right ) & \mathrm{for} \  \alt{x \in \bfR^3}\\
& \partial_t \psi_I(T,x)  = \partial_t \left (t^{-2} \psi_\infty \left ( \frac{x}{t} \right ) \right ) \biggr |_{t = T} \quad & \mathrm{for} \  \alt{x \in \bfR^3},
\end{eqnarray*}
and (iii) $\psi_{II}$ satisfies
\begin{eqnarray*}
& \Box_{t,x} \psi_{II} =  \eta(t,x) - t^{-4} \eta_\infty \left ( \frac{x}{t} \right ) \quad & \mathrm{for} \  t > T,  \alt{x \in \bfR^3}\\
& \psi_{II}(T,x) = 0 & \mathrm{for} \  \alt{x \in \bfR^3}\\
&\partial_t \psi_{II}(T,x) = 0 & \mathrm{for} \  \alt{x \in \bfR^3}.
\end{eqnarray*}

\alt{With these functions well-defined, we consider only $t \geq T$ and $|x| \leq \gamma t$ and estimate $\psi_H$, $\psi_I$, and $\psi_{II}$ separately.
We will further take $t$ larger, as needed below, and in this case use the notation  $|x| \lesssim \gamma t$.}
First, we turn our attention to $\psi_I$. Note that by \eqref{psi_wave} and uniqueness, we must have
\begin{equation}
\label{psi_I}
\psi_I(t,x) = t^{-2} \psi_\infty \left (\frac{x}{t} \right )
\end{equation}
for all $t \geq T$ and $|x| \leq \gamma t$, as these functions satisfy the same inhomogeneous wave equation with identical ``initial'' conditions at $t = T$.
As we will show, the leading contribution to $\psi(t, x)$ will be  the $\psi_I(t, x)$ term, which is merely $t^{-2}\psi_\infty\left ( \frac{x}{t} \right )$ due to \eqref{psi_I},

Next, we show that the homogeneous term $\psi_H$ vanishes for sufficiently large $t$.
In particular, using the Kirchoff formula \cite{Evans} we express the homogeneous solution $\psi_H(t,x)$ for $t \geq T$ and $|x| \leq \gamma t$ exactly as
\begin{equation}
\label{psiH_rep}
\psi_H(t,x) = \frac{1}{4\pi (t- T)^2} \left [ \int_{|x-z| = t-T} \left ( \widetilde{\psi}_0(z) + \nabla \widetilde{\psi}_0(z) \cdot (z-x) \right ) dS_z + (t-T) \int_{|x-z| = t-T} \widetilde{\psi}_1(z) dS_z \right ].
\end{equation}
As $\psi_0(x)$ and $\psi_1(x)$ are supported on the set $\{|x| \leq L\}$, the solution $\psi(T, x)$ of \eqref{wavepsi} is supported on $\{|x| \leq T+L\}$ due to the finite speed of propagation of solutions to the wave equation and the support assumption on the source term $\eta$.
In particular, we note that the portion of $\psi(T, x)$ that stems from the source term is exactly
$$\psi_\text{inhom}(T,x) = \int_{|x-z| \leq T} \eta(T - |x-z|, z) \frac{dz}{|x - z|},$$
and a brief calculation shows that $|x| > T + L$ implies 
$$|z| > \left (T - |x-z| \right ) + L > \zeta \left (T - |x-z| \right ) + L$$ 
for any $z \in \bfR^3$ satisfying $|x - z| \leq T$.
As $\eta(t,x)$ is supported on the set $\{(t,x) : |x| \leq \zeta t + L\}$, such values lie outside of the support of this function when evaluated at the point $(T - |x - z|, z)$. 
Hence, $\eta(T - |x - z|, z) = 0$ for $|x| > T + L$ and  any $z \in \bfR^3$ satisfying $|x - z| \leq T$,
and we conclude that $\psi(T, x)$ is supported on $\{|x| \leq T+L\}$.

Furthermore, as $\psi_\infty \left (\frac{x}{t} \right ) = 0$ for $|x| \geq \gamma t$, we see that
$\psi_\infty \left (\frac{x}{t} \right )$ is supported on $\{ |x| \leq \gamma t\}$, and this is a subset of $\{|x| \leq t + L\}$.
Thus, the new ``data'' functions $\widetilde{\psi}_0(x)$ and $\widetilde{\psi}_1(x)$ are also supported on $\{|x| \leq T+L\}$. 
Recall that for $t \geq T$ we have
$|x| \leq \gamma t.$
Therefore, we further take $t > \frac{2T + L}{1-\gamma}$ so that for $z$ in the support of $\widetilde{\psi}_0$ and $\widetilde{\psi}_1$, we find
\begin{equation}
\label{out_support}
|x - z| \leq \gamma t + |z| \leq \gamma t + T + L  < t-T
\end{equation}
where $|x| \leq \gamma t$.
However, the integrals in \eqref{psiH_rep} are taken over $|x - z| = t - T$, and as shown by \eqref{out_support}, this set lies outside of the support of the integrands. Therefore, we have
\begin{equation}
\label{psiH}
\psi_H(t,x) = 0
\end{equation}
for $|x| \lesssim \gamma t$.

Finally, estimating $\psi_{II}$ requires more care than the previous terms. We first denote 
$$\mcN(t,x) = \left [\eta(t,x) - t^{-4} \eta_\infty \left (\frac{x}{t} \right )  \right ]  \mathbbm{1}_{|x| \leq \gamma t}.$$
With this, we can express the solution of the associated inhomogeneous wave equation with zero ``data'' at $t = T$ as
$$\psi_{II}(t,x) = \int_{|x-z| \leq t - T} \mcN(t - T - |x-z|, z) \frac{dz}{|x - z|}$$
for $t \geq T$.
For ease of notation, we let $s = t - T$ and rewrite this as
$$\psi_{II}(s+T,x) = \int_{|x-z| \leq s} \mcN(s - |x-z|, z) \frac{dz}{|x - z|}$$
for $s \geq 0$.
Upon multiplying by $t^{-4}$, the assumed estimate on the integrand gives
$$ \left | \eta(t,x) - t^{-4}\eta_\infty \left ( \frac{x}{t} \right ) \right | \lesssim t^{-5}\ln^8(t).$$
In addition, as $\{|x| \leq \gamma t\} \subset \{|x| \leq t + L\}$, we can use the support of $\mcN$ within this estimate to find
$$|\mcN(t,x)| \lesssim \left (t + |x| + 2L \right )^{-5}\ln^8\left (t + |x| + 2L \right ) \mathbbm{1}_{|x| \leq t+L}.$$
This implies
$$|\psi_{II}(s+T,x)| \lesssim \int_{|x-z| \leq s} (s - |x-z| +|z| + 2L)^{-5}\ln^8(s - |x-z| + |z| + 2L) |x-z|^{-1} \mathbbm{1}_{|z| \leq s - |x-z| + L} dz.$$
To estimate the remaining integral, we rely upon Lemma \ref{LGS_CV}.
More specifically, we define a cutoff function $\phi_L(\lambda)$ to be zero for $\lambda >  \tau + L$ and then choose
$$\Phi(\tau, \lambda) = (\tau+\lambda+2L)^{-5}\ln^8(\tau+\lambda+2L) \phi_L(\lambda) \qquad \mathrm{and} \qquad \Psi(\sigma) = \sigma^{-1}$$ 
in invoking Lemma \ref{LGS_CV}. 
Thus, the integral of interest becomes
$$|\psi_{II}(s+T,x)| \lesssim r^{-1} \int_0^{s} \int_{a}^{b} (\tau + \lambda + 2L)^{-5} \ln^8(\tau + \lambda + 2L) \lambda \phi_L(\lambda) \  d\lambda  d\tau$$
where 
$$r = |x|, \qquad a = |s - r - \tau|, \qquad \mathrm{and} \qquad b = s + r - \tau.$$
Using the inequalities
$$\tau + \lambda + 2L \leq s + r + 2L,  \qquad \tau + \lambda + 2L\geq s - r + 2L, \qquad \mathrm{and} \qquad \lambda \leq \tau+ \lambda + 2L$$ 
simplifies the estimate to
$$|\psi_{II}(s+T,x)| \lesssim r^{-1} \ln^8(s+r+2L) (s - r + 2L)^{-1} \int_0^{s} \int_{a}^{b} (\tau + \lambda + 2L)^{-3} \  d\lambda  d\tau.$$
Then, upon integrating in $\lambda$ we find
\begin{equation}
\label{psiII_est}
|\psi_{II}(s+T,x)| \lesssim r^{-1} \ln^8(s+r+2L) (s - r + 2L)^{-1} \int_0^{s} \frac{(b-a)(b+a + 2\tau + 4L)}{(\tau + a + 2L)^2(\tau + b + 2L)^2} \ d\tau.
\end{equation}
To estimate the remaining integral, we first define the function
$$\mu(x)=\begin{cases}x, &x\geq0,\\ 0, &x<0\end{cases}$$ 
and split this into two pieces, namely
$$\int_0^{s} \frac{(b-a)(b+a + 2\tau + 4L)}{(\tau + a + 2L)^2(\tau + b + 2L)^2}d\tau = I + II$$
where
$$I := \int_0^{\mu(s-r)} \frac{(b-a)(b+a + 2\tau + 4L)}{(\tau + a + 2L)^2(\tau + b + 2L)^2} \ d\tau $$
and
$$II := \int_{\mu(s-r)}^s \frac{(b-a)(b+a + 2\tau + 4L)}{(\tau + a + 2L)^2(\tau + b + 2L)^2} \ d\tau. $$
To estimate $I$, we note that $\tau \leq s-r$ on this interval so that
$$b - a = 2\min\{r, s-\tau\} = 2r \qquad \mathrm{and} \qquad b + a = 2\max\{r, s-\tau\} = 2(s-\tau).$$
Furthermore, this implies $\tau + a = s- r$ and $ \tau + b = s+r$, and we find
\begin{eqnarray*}
I & = & 4r \int_0^{\mu(s-r)} \frac{s + 2L}{(s - r + 2L)^2 (s + r + 2L)^2} \ d\tau\\
  & \lesssim & r(s - r + 2L)^{-2} (s + r + 2L)^{-2} (s-r)(s+2L)\\
  & \lesssim & r(s + r + 2L)^{-1} (s - r + 2L)^{-1}
\end{eqnarray*}
Contrastingly, we estimate $II$, by first noting that $\tau \geq s-r$ on this interval so that
$$b - a = 2(s-\tau) \leq 2r   \qquad \mathrm{and} \qquad b + a = 2r.$$
Furthermore, because $\tau + a = 2\tau - s + r$ and $ \tau + b = s+r$, we find
\begin{eqnarray*}
II & = & 4r \int_{\mu(s-r)}^s \frac{\tau + r + 2L}{(2\tau - s + r + 2L)^2 (s + r + 2L)^2} \ d\tau\\
  & \lesssim & r(s + r + 2L)^{-1} \int_{\mu(s-r)}^\infty (\tau + 2L)^{-2} \ d\tau\\
  & \lesssim & r(s + r + 2L)^{-1} (s - r + 2L)^{-1}
\end{eqnarray*}
as $\mu(s - r) \geq s-r$.
Now, as $I$ and $II$ satisfy the same estimate, we insert these within \eqref{psiII_est} to find
$$|\psi_{II}(s+T,x)| \lesssim (s + r + 2L)^{-1} (s - r + 2L)^{-2} \ln^8(s+r+2L).$$
Changing back to the $t$ and $T$ variables and removing $r = |x| > 0$ where possible yields
$$|\psi_{II}(t,x)| \lesssim (t - T  + 2L)^{-1} (t - T - |x| + 2L)^{-2} \ln^8(t - T+ |x| + 2L).$$
Therefore, for $t$ sufficiently large and $|x| \leq \gamma t$, we find
\begin{eqnarray*}
|\psi_{II}(t,x)| & \lesssim & (t - T + 2L)^{-1} ( (1-\gamma)t - T + 2L)^{-2} \ln^8( (1+\gamma) t - T + 2L)\\
& \lesssim & t^{-3}\ln^8(t).
\end{eqnarray*}

Finally, assembling the estimate of $\psi_{II}$ along with \eqref{psi_I} and \eqref{psiH} gives
\begin{eqnarray*}
\left | \psi(t,x) - t^{-2} \psi_\infty \left ( \frac{x}{t} \right ) \right | & \leq & \left |\psi_I(t,x) - t^{-2} \psi_\infty \left ( \frac{x}{t} \right ) \right | + | \psi_H(t,x)| + |\psi_{II}(t,x)|\\
& \lesssim & t^{-3}\ln^8(t),
\end{eqnarray*}
for $|x| \lesssim \gamma t$.
Upon multiplying by $t^2$, the stated estimate follows, and the proof is complete.
\end{proof}

With the lemma established, we now use it along with previous estimates of derivatives of $\rho$ and $j$ in order to construct the limiting asymptotic behavior of the electric and magnetic fields.

\begin{lemma}
\label{LField}
There exist $E_\infty, B_\infty \in C^{2,\delta}(\Gamma_\gamma)$ 
such that
$${\sup_{|x| \lesssim \gamma t}} \left | t^{2} E(t,x) - E_\infty\left (\frac{x}{t} \right )\right |\lesssim t^{-1}\ln^{8}(t),$$
and
$${\sup_{|x| \lesssim \gamma t}} \left | t^{2} B(t,x) - B_\infty\left (\frac{x}{t} \right )\right |\lesssim t^{-1}\ln^{8}(t).$$
%
Furthermore, there is $K_\infty \in C^{2,\delta}(\Gamma_\gamma)$ such that
$$\sup_{\substack{|x| \lesssim \ln(t)\\ |p| \leq \beta}}   \left | t^2 K(t, x + v(p)t , p) - K_\infty \left ( p \right ) \right |  \lesssim t^{-1}\ln^{8}(t).$$
\end{lemma}
\begin{proof}
Due to Lemma \ref{LPrelim}, the supports of $\rho(t,y)$, $j(t,y)$, and their derivatives all lie within $\{(t,y) : |y| \leq \zeta t + L\}$.
Additionally, Lemmas \ref{LDensityderivative} and \ref{LCurrentderivative} provide the estimates
$$\sup_{x \in \bfR^3}   \left | t^4\partial_{x_i}\rho(t,x) - \partial_{q_i}\rho_\infty \left ( \frac{x}{t} \right) \right | \lesssim t^{-1} \ln^{8}(t)$$
and
$$\sup_{x\in \bfR^3} \left | t^4 \partial_tj(t,x) +\left [3j_\infty\left (\frac{x}{t} \right) + \frac{x}{t} \cdot \nabla_qj_\infty\left (\frac{x}{t} \right )\right ] \right | \lesssim t^{-1} \ln^8(t)$$
for every $i=1,2,3$.
From Maxwell's equations, a straightforward calculation shows that the $i$th component of the electric field satisfies the inhomogeneous wave equation
$$\Box_{t,x} E^i = - \partial_{x_i} \rho(t,x) - \partial_t j^i(t,x)$$
with the given initial data expressed in terms of $E_0(x)$ and $B_0(x)$.
Thus, applying Lemma \ref{LWave} with $\psi(t,x) = E^i(t,x)$ and $\eta(t,x) = -\partial_{x_i} \rho(t,x) - \partial_t j^i(t,x)$
provides the existence of a limiting electric field $E^i_\infty \in C^{2,\delta}(\Gamma_\gamma)$ and the estimates
\begin{equation}
\label{Eest}
\left | t^2E^i(t,x) - E^i_\infty \left ( \frac{x}{t} \right ) \right | \lesssim t^{-1}\ln^8(t)
\end{equation}
for every $i=1,2,3$ with $|x| \lesssim \gamma t$.

Similarly a wave equation for each component of the magnetic field can be derived from Maxwell's equations resulting in
$$\Box_{t,x} B^i = \left (\nabla \times j \right)^i(t,x).$$
Thus, using the spatial derivative estimate on $j$ from Lemma \ref{LCurrentderivative}, namely
$$\sup_{x\in \bfR^3} \left | t^4 \partial_{x_i}j(t,x) -  \partial_{q_i}j_\infty\left (\frac{x}{t} \right) \right |  \lesssim t^{-1} \ln^8(t)$$
for every $i = 1, 2, 3$,
and applying Lemma \ref{LWave} with $\psi(t,x) = B^i(t,x)$, and $\eta(t,x) = (\nabla \times j)^i(t,x)$
provides the existence of a limiting magnetic field $B^i_\infty \in C^{2,\delta}(\Gamma_\gamma)$ and the corresponding estimate
\begin{equation}
\label{Best}
\left | t^2B^i(t,x) - B^i_\infty \left ( \frac{x}{t} \right ) \right | \lesssim t^{-1}\ln^8(t)
\end{equation}
for every $i=1,2,3$  with $|x| \lesssim \gamma t$.
Finally, assembling the convergence estimates \eqref{Eest} and \eqref{Best} and defining $E_\infty(q)$ and $B_\infty(q)$ componentwise, we use these to construct the limiting Lorentz force given by
$$K_\infty(p) := E_\infty(v(p)) + v(p) \times B_\infty(v(p))$$
for every $|p| \leq \beta$.
Further, considering $|x| \lesssim \ln(t)$ and $|p| \leq \beta$, and using the boundedness of $v(p)$ then directly yields
\begin{eqnarray*}
 \left | t^2 K(t, x + v(p)t , p) - K_\infty \left ( p \right ) \right |  & \leq &  \left | t^2 E(t, x + v(p)t) - E_\infty \left ( v(p) \right ) \right | +  \left | t^2 B(t, x + v(p)t) - B_\infty \left ( v(p) \right ) \right |\\
 & \leq &  \left | t^2 E(t, x + v(p)t) - E_\infty \left ( \frac{x+v(p)t}{t} \right ) \right |\\
 & & \quad + \left | E_\infty \left ( v(p) + \frac{x}{t} \right ) - E_\infty \left ( v(p) \right ) \right |\\
 & & \quad + \left | t^2 B(t, x + v(p)t) - B_\infty \left ( \frac{x+v(p)t}{t} \right ) \right |\\
 & & \quad  + \left | B_\infty \left ( v(p) + \frac{x}{t} \right ) - B_\infty \left ( v(p) \right ) \right | \\
 & \lesssim &  t^{-1}\ln^{8}(t) + \left ( \Vert \nabla_q E \Vert_\infty + \Vert \nabla_q B \Vert_\infty \right )\frac{|x|}{t}\\
 & \lesssim &  t^{-1}\ln^{8}(t).
\end{eqnarray*}
\end{proof}

In the next section, we will use this limiting Lorentz force and the associated convergence estimate to obtain a modified scattering result for solutions.

\subsection{Modified Scattering}
\label{sect:modscatt}

With the field estimates solidified, we prove that the distribution function scatters to a limiting value as $t \to \infty$ along a specific trajectory in phase space that may differ from its linear profile, and this is known as ``modified scattering''.
Many of the ideas in this direction arise from \cite{Ionescu, Pankavich2021, Pausader}. 
First, we remind the reader that using the field estimates of Lemma \ref{LPrelim} in Lemma \ref{Xsupp} produces
\begin{equation}
\label{Y}
|\mcY(t,0,x,p)| \lesssim \ln(t)
\end{equation} 
where the implicit constant in this inequality may depend upon fixed $(x,p) \in \bS_g(0)$.
Thus, the spatial characteristics of $g$ grow unbounded in time and require an additional logarithmic correction to construct trajectories along which $f$ converges as $t \to \infty$. 
\begin{lemma}
\label{LModScattering}
There exists $f_\infty \in C(\bfR^6)$ such that
$$h(t,x,p) = f \left (t,x + v(p)t - \ln(t) \mathbb{A}(p)K_\infty(p), p \right)$$
satisfies 
$$h(t,x, p) \to f_\infty(x,p)$$ 
uniformly as $t \to \infty$.
\alt{Moreover}, we have the convergence estimate
$$\sup_{(x,p) \in \bfR^6} \left | f \biggl (t,x + v(p)t - \ln(t)\mathbb{A}(p) K_\infty(p), p \biggr ) - f_\infty(x,p) \right |  \lesssim t^{-1}\ln^8(t).$$
\end{lemma}

\begin{proof}
Let
$$\mcW(t,x,p) = x -  \ln(t)\mathbb{A}(p)K_\infty(p)$$
for every $t \geq 1$ and $x,p \in \bfR^3$ so that
$$h(t,x,p) := f \biggl (t,x + v(p)t -  \ln(t)\mathbb{A}(p) K_\infty(p), p \biggr ) = g(t, \mcW(t,x,p), p).$$
We will further take $t$ to be sufficiently large as necessary.
Because of the compact momentum support of $g$, we note that $g(t,\mcW(t,x,p),p) = 0$ for $|p| > \beta$.
Additionally, there is $C_0>0$ 
such that $|x| \geq C_0\ln(t)$ implies $|\mcW(t,x,p)| \geq C_1 \ln(t)$ for some $C_1 > 0$, and thus $g(t,\mcW(t,x,p),p) = 0$ due to \eqref{Y}.
Hence, $g(t,\mcW(t,x,p),p) = 0$ whenever $|x| \geq C\ln(t)$ or $|p| > \beta$, and it suffices to take $|x| \lesssim \ln(t)$ and $|p| \leq \beta$ throughout the proof. This further implies $|\mcW(t,x,p)| \lesssim \ln(t)$.
Because $g$ satisfies \eqref{VMg}, we deduce that $h$ satisfies
$$\partial_t h =  t^{-1}  \left ( t^2 K(t, \mcW+v(p)t,p) - K_\infty(p) \right ) \cdot \mathbb{A}(p) \nabla_x g(t,\mcW,p) - K(t, \mcW+v(p)t)\cdot \nabla_p g(t,\mcW,p).$$
Similar to the proof of Lemma \ref{Funif}, we wish to show that 
$\displaystyle \| \partial_t h(t) \|_\infty$ is integrable in order to establish the existence of a limiting function in this norm.

To this end, we decompose $\partial_t h$ so that
$$\left |\partial_t h(t,x,p) \right | \alt{\lesssim} I+ II$$
where
$$I \leq t^{-1}\left | t^2 K(t, \mcW+v(p)t, p) - K_\infty(p) \right | \  \left |\nabla_x g(t,\mcW,p)  \right |$$
and
$$II = \left | K(t, \mcW+v(p)t, p)\cdot \nabla_p g(t,\mcW,p) \right |.$$
The second term is well-behaved. Indeed, using Lemmas \ref{LPrelim} and \ref{Dvg} we find
\begin{equation}
\label{IIh}
II \leq \sup_{\substack{|x| \lesssim \ln(t)\\ |p| \leq \beta}} |K(t, x+v(p)t, p) | \mcG_1(t)  \lesssim \mcK_0(t) \mcG_1(t)  \lesssim t^{-2}\ln^2(t).
\end{equation}
Using Lemma \ref{LPrelim} yields
$$ \| \nabla_x g(t) \|_\infty  = \| \nabla_x f(t) \|_\infty \lesssim 1.$$
Thus, the latter term in $I$ is uniformly bounded in time, which implies
$$I \lesssim t^{-1} \left | t^2 K(t, \mcW+v(p)t, p) - K_\infty(p) \right |.$$
Because $\left |\mcW(t,x,p) \right |  \lesssim \ln(t)$ and $|p| \leq \beta$, we invoke Lemma \ref{LField} to conclude
\begin{equation}
\label{Ih}
I \lesssim t^{-2}\ln^8(t).
\end{equation}
Combining \eqref{IIh} and \eqref{Ih}, we have
$$ \|\partial_t h(t) \|_\infty \lesssim t^{-2}\ln^8(t).$$
As this bound is integrable in time, there is $f_\infty \in C(\bfR^6)$ such that
$$\| h(t) -f_\infty \|_\infty \lesssim t^{-1}\ln^8(t),$$
which completes the proof.
\end{proof}

\section{Proofs of Theorems}
\label{sec:thmpfs}

We first prove the refined small data theorem, which comprises the majority of this section.
\begin{proof}[Proof of Theorem \ref{T0}]
To begin, we note that the smallness assumption of Theorem \ref{T0} can be guaranteed to satisfy the hypotheses of Theorem \ref{GS} by taking $\epsilon_0$ sufficiently small.
Hence, the estimates from the latter result, for instance on first-order derivatives of the fields and particle distributions, follow immediately. 
To establish the new result, which requires estimates on second derivatives, we follow a similar argument to that of \cite{GS} and let
$$ \Vert K \Vert_0 = \sup_{\substack{t\geq 0\\ x \in \mathbb{R}^3}} \biggl ( ( t + |x| + 2L)(t - |x| + 2L) \left ( |E(t,x)| + |B(t,x)| \right) \biggr ),$$
$$ \Vert K \Vert_1 = \sup_{\substack{t\geq 0\\ x \in \mathbb{R}^3}} \left (\frac{( t + |x| + 2L)(t - |x| + 2L)^2}{\ln(t + |x| + 2L)} \left ( |\nabla_xE(t,x)| + |\nabla_xB(t,x)| \right) \right ),$$
$$ \Vert K \Vert_2 = \sup_{\substack{t\geq 0\\ x \in \mathbb{R}^3}} \left ( \frac{(t + |x| + 2L)(t - |x| + 2L)^3}{\ln(t + |x| + 2L)} \left ( |\nabla_x^2E(t,x)| + |\nabla_x^2B(t,x)| \right) \right ),$$
and
$$\Vert K \Vert = \Vert K \Vert_0 + \Vert K \Vert_1 + \Vert K \Vert_2.$$
Let $\epsilon > 0$ and define
$$\mcX = \{ K \in C^2 : K(t,x) = 0 \ \mathrm{for} \ |x| > t + L \ \mathrm{and} \ \Vert K \Vert \leq \epsilon \}.$$
The iteration scheme is, by now, standard. 
Given $K \in \mcX$, we define the associated characteristic system \eqref{char} and $f^\alpha(t,x,p)$ as the evaluation of $f^\alpha_0$ along the characteristics, so that $f^\alpha$ satisfies the Vlasov equation.
With this, we can define $\rho(t,x)$ and $j(t,x)$ as in \eqref{RVM} and finally, the new electric and magnetic fields $E^*(t,x)$ and $B^*(t,x)$ as solutions of Maxwell's equations with the associated source terms given by derivatives of $\rho$ and $j$.
Finally, we label $K^* = (E^*, B^*)$.
Under this iteration scheme, Glassey and Strauss \cite[p. 177]{Glassey} established for a given $L > 0$
\begin{equation}
\label{K*1}
\Vert K^*\Vert_0 + \Vert K^*\Vert_1 \leq C\epsilon_0 \left ( 1 + \Vert K\Vert_0 + \Vert K\Vert_1^2 \right ).
\end{equation}
We will further derive the estimate
\begin{equation}
\label{K*2}
\Vert K^*\Vert_2 \leq C_0\epsilon_0 \left ( 1 + \Vert K \Vert + \Vert K\Vert_0^2 + \Vert K\Vert_1^2 \right )
\end{equation}
for some $C_0 > 0$,
which, when combined with \eqref{K*1}, will then imply
$$\Vert K^*\Vert \leq C\epsilon_0 \left ( 1 + \Vert K\Vert + \Vert K\Vert^2 \right )$$
where $C$ is independent of $\epsilon_0$.
Taking $K \in \mcX$ so that $\Vert K \Vert \leq \epsilon$ and then taking $\epsilon$ and $\epsilon_0$ sufficiently small will show $\Vert K^* \Vert \leq \epsilon$, and thus $K^* \in \mcX$. 
Finally, this can be used to show (cf. \cite[Section 5.8]{Glassey} and \cite{GS}) that the iterates and their first and second derivatives converge pointwise and satisfy the stated decay estimates.
What remains is to establish \eqref{K*2}, and the remainder of the proof will focus on doing exactly this.

Let $K \in \mcX$ be given. 
Then, we have
$$|\nabla_x^2 E(t,\mcX(t, \tau, x, p))| + |\nabla_x^2 B(t,\mcX(t, \tau, x, p))| \leq \frac{C\ln(t+2)}{(t+1)\left (t - |\mcX(t)| + 2L \right )^3} \leq \frac{C\ln(t+2)}{(t+L)^4} \lesssim t^{-4}\ln(t)$$
for all $\tau \geq 0$, $(x,p) \in \mcS_f(\tau)$
and 
$$|\nabla_x^2 E(t,x +v(p)t)| + |\nabla_x^2 B(t,x +v(p)t)| \leq \frac{C\ln(t+2)}{(t+1)\left (t - |x +v(p)t| + 2L \right )^3} \lesssim t^{-4}\ln(t)$$
for $|x| \lesssim \ln(t)$ and $|p| \leq \beta$.
Putting these together, we have established
$$\mcK_2(t) \lesssim t^{-4} \ln(t),$$
and the results of Section $4$, which rely on this estimate, are all valid.
We estimate $\Vert K^* \Vert_2$  by utilizing the representation theorem of Glassey and Strauss applied to second derivatives of the fields.
In their proof of \eqref{K*1}, they rely on a representation of the fields $E^*$ and $B^*$, as well as their first derivatives, in coordinates that respect the symmetries of the wave and Vlasov equations. This was introduced in \cite{GS2} and later presented in \cite[Theorems 5.3.1 \& 5.4.1]{Glassey}. For completeness, let us recall the representation of the first derivatives of the fields (stated for $E^*$, with an analogous representation for $B^*$).

\begin{theorem}[\cite{Glassey}, Theorem 5.4.1]
Assume the hypotheses of Theorem \ref{GS}. Then for $i,k=1,2,3$ the following representation holds:
	\begin{align*}
	\partial_{x_k} E^{*i}
	&=
	\text{data}
	+
	\iint_{|\omega|=1} d(\omega,v(p))f\ d\omega dp
	+
	\iint_{|x-y|\leq t} a(\omega,v(p))f|x-y|^{-3}\ dp dy\\
	&+
	\iint_{|x-y|\leq t} b(\omega,v(p))(Sf)|x-y|^{-2}\ dp dy
	+
	\iint_{|x-y|\leq t} c(\omega,v(p))(S^2f)|x-y|^{-1}\ dp dy
	\end{align*}
where $S=\partial_t+v(p)\cdot\nabla_x$, $f,Sf,S^2f$ are evaluated at $(t-|x-y|,y,p)$,   $\omega=\frac{x-y}{|x-y|}$  and where the data term includes all terms that depend upon the initial data, and is given by
	\begin{align*}
	\text{data}
	&=
	(\partial_{x_k} E^{*i})_0
	-
	\frac{1}{t^2}\iint_{|x-y|=t}d(\omega,v(p))f_0(y,p)\ dp dS_y
	+
	\frac{1}{t}\iint_{|x-y|=t}e(\omega,v(p))(Sf)(0,y,p)\ dp dS_y.
	\end{align*}
The functions $a,b,c,d,e$ are $C^\infty$ except at $1+v(p)\cdot\omega=0$ and have algebraic singularities at such points. Moreover, $\int_{|\omega|=1}a(\omega,v(p))\ d\omega=0$. 
\end{theorem}

We are now ready to represent and estimate the second derivatives of $E^*$ in an effort to prove \eqref{K*2}. The second derivatives of $B^*$ follow an analogous argument, which is omitted for brevity. Since the wave operator commutes with spatial derivatives, differentiating the wave equation satisfied by the $i$th component of $E^*$ and applying the Glassey-Strauss representation formula \cite{Glassey, GS} yields the same representation as above, replacing $f$ with $\partial_{x_\ell} f$, as follows (we  suppress the $i$ index to simplify notation)
	\[
	\partial_{x_k x_\ell} E^*
	=
	A_z + A_w + A_{TT} + A_{TS} + A_{ST} + A_{SS}
	\]
where
	\begin{align*}
	 A_{\omega}
	 &=
	\iint_{|\omega|=1} d(\omega,v(p))\partial_{x_\ell} f\ d\omega dp,\\
	 A_{TT}
	 &=
	\iint_{|x-y|\leq t} a(\omega, v(p)) \partial_{x_\ell} f |x-y|^{-3} \ dp dy,\\
	A_{ST} + A_{TS} 
	&=
	\iint_{|x-y|\leq t} b(\omega,v(p))S(\partial_{x_\ell} f)|x-y|^{-2}\ dp dy\\
	&=
	\iint_{|x-y|\leq t} b(\omega,v(p))\partial_{x_\ell} (Sf)|x-y|^{-2}\ dp dy\\
	&=
	\iint_{|x-y|\leq t} \nabla_p b(\omega, v(p)) \cdot \partial_{x_\ell}\left(K  f \right)|x-y|^{-2} \ dp dy\\
	&=
	\iint_{|x-y|\leq t} \nabla_p b(\omega, v(p)) \cdot \left(f\partial_{x_\ell} K  +K\partial_{x_\ell} f \right)|x-y|^{-2} \ dp dy,\\
	A_{SS}
	&=
	\iint_{|x-y|\leq t} c(\omega, v(p)) S^2 (\partial_{x_\ell} f) |x-y|^{-1} \ dp dy,
	\end{align*}
and 
	\begin{align*}
	A_z
	&=
	(\partial_{x_k x_\ell} E^{*i})_0
	-
	\frac{1}{t^2}\iint_{|x-y|=t}d(\omega,v(p))\partial_{x_\ell}f_0(y,p)\ dp dS_y
	+
	\frac{1}{t}\iint_{|x-y|=t}e(\omega,v(p))(S\partial_{x_\ell} f)(0,y,p)\ dp dS_y.
	\end{align*}

Note that the operator $S$ commutes with any spatial derivative, so we are free to replace the order of $S$ and $\partial_{x_\ell}$  within these expressions. Now we wish to estimate these terms one by one to show that $K^*\in\mcX$.

\begin{remark}
Though we suppress the arguments, all quantities involving $f$ and $K$ within the integrals above are evaluated at the point $(t - |x-y|, y, p)$.
\end{remark}

\begin{remark}
In the estimates that follow we will routinely use the following bounds, which hold for all $t$ sufficiently large and $r\leq t+L$, namely
	\begin{equation}\label{eq:useful-t-bounds}
	t^{-1}\leq C_0(1+t)^{-1}\leq C_1(t+r+2L)^{-1}\leq C_2(t-r+2L)^{-1}
	\end{equation}
for some $C_0,C_1,C_2>0$ which do not depend on $t$.
\end{remark}

Throughout, we will also make use of an estimate on $|\rho(t,x)|$ implied by Theorem \ref{GS}, namely
\begin{equation}
\label{rhoest}
|\rho(t,x) | \leq C(t + |x| + 2L)^{-3}
\end{equation}
for all $t \geq 0, x \in \bfR^3$, which follows because $f(t,x,p) = 0$ for $|x| > t+ L$ due to Lemma \ref{LPrelim}.
Finally, we mention that the estimates that follow are not merely an extension of the tools in \cite{GS}, but instead require a precise estimate on the growth of the support of $f$, which is not used in the first derivative estimates of \cite{GS}.

\textbf{1. Estimate of $A_{\omega}$.}

We first use Lemma \ref{Lm} (more specifically, Remark \ref{Lm_rmk}) and the boundedness of $d(\omega,v(p))$ to obtain
	\begin{equation}
	\label{A_omega}
	\left|A_{\omega}\right|
	\leq
	\int_{|\omega|=1}\left|\int d(\omega,v(p))\partial_{x_\ell} f\  dp\right|d\omega
	\lesssim
	\epsilon_0t^{-4}.
	\end{equation}

\textbf{2. Estimate of  $A_{TT}$.} 

Following the strategy of Glassey-Strauss (see, e.g. \cite[pp. 178-179]{Glassey}), we split this term into two parts in order to deal with the singularity as $|x-y| \to 0$.
In particular, we take $t$ sufficiently large and write
	\begin{align*}
	A_{TT}
	 =
	\iint_{|x-y|\leq \frac1t} a(\omega, v(p)) \partial_{x_\ell} f |x-y|^{-3} \ dp dy+\iint_{ \frac1t\leq |x-y|\leq t} a(\omega, v(p)) \partial_{x_\ell} f |x-y|^{-3} \ dp dy
	=:
	I+II.
	\end{align*}
For $I$ we use the fact that $\int_{|\omega|=1} a(\omega,v(p))\ d\omega=0$ to introduce a term that makes no contribution but allows us to write
	\begin{align*}
	I
	&=
	\iint_{|x-y|\leq  \frac1t} a(\omega, v(p)) \partial_{x_\ell} f(t-|x-y|,y,p) |x-y|^{-3} \ dp dy\\
	&=
	\iint_{|x-y|\leq  \frac1t} a(\omega, v(p)) \partial_{x_\ell} \bigg(f(t-|x-y|,y,p)-f(t-|x-y|,x,p) \bigg) |x-y|^{-3}\ dp dy.
	\end{align*}
By the Mean-Value Theorem and Lemma \ref{D2g} (more specifically, Remark \ref{rmk:D2g}), we bound the difference by
$$\frac{\partial_{x_\ell}(f(t-|x-y|,y,p)-f(t-|x-y|, x,p))}{|y-x|} \leq \left \|\nabla^2_xf (t-|x-y|) \right \|_\infty= \left \|\nabla^2_xg (t-|x-y|) \right \|_\infty \lesssim \epsilon_0.$$  
Moreover, the integration in the $p$ variable is only over the set
	\[
	\left\{
	p \in \bfR^3\ : \ |\mcX(0,t-|x-y|,z,p)|\leq L\text{ for all }z\in[x,y]
	\right\}
	\]
where $[x,y]$ is the line segment connecting $x$ and $y$. From Lemma \ref{LDchar} (see also \cite[Lemma 6.3.3]{Glassey}) we know that the diameter of this set is of order $(1+t-|x-y|)^{-1}$. Hence, we find
	\begin{align*}
	|I|
	&\lesssim
	\|\nabla^2_xg\|_\infty\|a\|_\infty\int_{|x-y|\leq  \frac1t}   (1+t-|x-y|)^{-3}|x-y|^{-2}\ dy\\
	&\lesssim
	\epsilon_0t^{-3}\int_{|x-y|\leq  \frac1t}   \left(\frac1t+1-\frac{|x-y|}{t}\right)^{-3}|x-y|^{-2}\ dy.
	\end{align*}
For $|x-y|\leq \frac1t$, we have $ \left(\frac1t+1-\frac{|x-y|}{t}\right)^{-3}\lesssim 1$ and thus conclude the estimate
\begin{equation}
\label{ITT}
	|I| \lesssim
	\epsilon_0t^{-3}\int_{|x-y|\leq \frac1t}   |x-y|^{-2}\ dy\lesssim \epsilon_0 t^{-4}.
\end{equation}

Turning to the estimate of $II$, we use  Lemma \ref{Lm}, \eqref{eq:useful-t-bounds}, and Lemma \ref{LGS_CV} to find
	\begin{align*}
	|II|
	&\leq
	\int_{ \frac1t\leq |x-y|\leq t}\left|\int a(\omega, v(p)) \partial_{x_\ell} f  (t-|x-y|,y,p)  \ dp \right| |x-y|^{-3}\ dy\\
	&\lesssim
	\epsilon_0\int_{ \frac1t\leq |x-y|\leq t} \left(t-|x-y|+|y|+2L\right)^{-4} |x-y|^{-3}\ dy\\
	&\lesssim
	\frac{\epsilon_0}{r}\int_0^{t- \frac1t} \int^{t+r-\tau}_{|t-r-\tau|}\left(\tau+\lambda+2L\right)^{-4} (t-\tau)^{-2}\ \lambda \ d \lambda d\tau
	\end{align*}
where $r = |x|$.
Next, we take $t$ sufficiently large so that $1/t<t/2$  and split the $\tau$-integral into
$$II = \frac{\epsilon_0}{r} \left ( II_A + II_B \right )$$
where
$$II_A = \int_0^{t/2} \int^{t+r-\tau}_{|t-r-\tau|}\left(\tau+\lambda+2L\right)^{-4} (t-\tau)^{-2}\ \lambda \ d \lambda d\tau$$
and
$$II_B = \int_{t/2}^{t-\frac1t} \int^{t+r-\tau}_{|t-r-\tau|}\left(\tau+\lambda+2L\right)^{-4} (t-\tau)^{-2}\ \lambda \ d \lambda d\tau.$$
For brevity, we  rename the lower and upper bounds of the $\lambda$ integral $a$ and $b$, respectively, so that
$$a = |t-r-\tau| \qquad \mathrm{and} \qquad b = t+r-\tau.$$
As $\lambda,\tau,L>0$ we have  $\lambda\left(\tau+\lambda+2L\right)^{-4} \leq \left(\tau+\lambda+2L\right)^{-3}$. Then, using the relationships
\begin{equation}
\label{ab_lambda}
t - r \leq \tau + a \leq \tau + \lambda \leq \tau + b = t +r
\end{equation}
so that $\left(\tau+\lambda+2L\right)^{-1} \leq (t-r+2L)^{-1}$
and performing the integration in $\lambda$, we find
	\begin{align*}
	II_A
	&\leq
	\frac{4}{t^2(t - r + 2L)}\int_0^{t/2} \int^{b}_{a}\left(\tau+\lambda+2L\right)^{-2} \  d \lambda d\tau\\
	&\leq
	\frac{4}{t^2(t - r + 2L)}\int_0^t\frac{b-a}{\left(a + \tau + 2L\right)\left(b + \tau + 2L\right)} \  d \lambda d\tau\\
	&=
	\frac{4}{t^2(t - r + 2L)(t + r + 2L)}\int_0^t\frac{b-a}{a + \tau + 2L} \ d \lambda d\tau.
	\end{align*}
Now, we distinguish between the cases where $t-r-\tau$ changes sign. Define the function
$$\mu(x)=\begin{cases}x, &x\geq0,\\ 0, &x<0.\end{cases}$$ 
We split the remaining integral at $\tau = \mu(t-r)$ and use both $b-a = 2\min\{r, t-\tau\}$ and $\mu(t-r) \geq t-r$ to find
	\begin{align*}
	\int_0^t\frac{b-a}{a + \tau + 2L} \ d \lambda d\tau
	&=
	\int_0^{\mu(t-r)} \frac{2r}{t-r+2L} \   d\tau+\int_{\mu(t-r)}^t\frac{2(t-\tau)}{r-t+2\tau+2L} \   d\tau \\
	&\leq
	2r \frac{\mu(t-r)}{t-r+2L} + \frac{1}{t-r+2L} \int_{\mu(t-r)}^t 2(t-\tau)  d\tau \\
	&\leq
	C \left (r + \frac{r^2}{t-r+2L} \right )  \\
	&\leq
	Cr \left (1 + \frac{t+r+2L}{t-r+2L} \right )  \\
	&\leq
	Cr \left (t+r+2L \right ) \left (t-r+2L \right )^{-1}.
	\end{align*}
Using this within the estimate for $II_A$ then yields
$$II_A \leq \frac{Cr}{t^2(t - r + 2L)^2} \lesssim r(t+r+2L)^{-1} (t - r + 2L)^{-3}.$$

Next, we estimate $II_B$. Because $\tau\geq t/2$, $\lambda\geq0$, and $b^2 - a^2 = 4r(t-\tau)$, we have
	\begin{align*}
	II_B
	&=
	\int_{t/2}^{t- \frac1t} \int^{b}_{a}\left(\tau+\lambda+2L\right)^{-4} (t-\tau)^{-2}\ \lambda \ d \lambda d\tau\\
	&\leq
	\left(t/2+2L\right)^{-4} \int_{t/2}^{t- \frac1t} (t-\tau)^{-2}\int^{b}_{a} \lambda \ d \lambda d\tau\\
	&\leq
	C\left(t+2L\right)^{-4} \int_{t/2}^{t- \frac1t} (t-\tau)^{-2}(b^2-a^2) \  d\tau\\
	&\leq
	Cr\left(t+2L\right)^{-4}\int_{t/2}^{t- \frac1t} (t-\tau)^{-1}  \  d\tau\\
	&\leq
	Cr\left(t+2L\right)^{-4}\ln(t)\\
	&\lesssim
	r\ln(t+r+2L)(t+r+2L)^{-1}(t-r+2L)^{-3}.
	\end{align*}
Combining the estimates for $II_A$ and $II_B$ then provides
	\begin{align*}
	|II|
	&\lesssim
	\frac{\epsilon_0\ln(t+r+2L)}{(t-r+2L)^{3}(t+r+2L)}.
	\end{align*}

Finally, assembling the estimates for $I$ and $II$ within $A_{TT}$, we have
\begin{equation}
\label{A_TT}
	\left|A_{TT}\right|
	\lesssim
	\frac{\epsilon_0}{t^{4}}
	+
	\frac{\epsilon_0\ln(t+r+2L)}{(t-r+2L)^{3}(t+r+2L)}\\
	\lesssim
\epsilon_0\ln(t+r+2L)(t-r+2L)^{-3}(t+r+2L)^{-1}.
\end{equation}

\textbf{3. Estimate of $A_{ST}+A_{TS}$.} 

First, we introduce a cutoff function $\varphi_L=\varphi_L(|y|)$ that is $0$ outside the support of $f (t-|x-y|,y,p)$; that is, $\varphi_L(|y|)=0$ for all $y$ such that $|y|\geq \gamma \left (t-|x-y| \right)+L$ due to Lemma \ref{LPrelim}. Then, we estimate this term as
	\begin{align*}
	A_{ST}+A_{TS}
	=&\
	\iint_{|x-y|\leq t} \nabla_p b(\omega, v(p)) \cdot \left(f\partial_{x_\ell} K  +K\partial_{x_\ell} f \right)|x-y|^{-2} \ dp dy\\
	\leq&\
	\left|\iint_{|x-y|\leq t} \nabla_p b \cdot \left(f\partial_{x_\ell} K   \right)|x-y|^{-2} \ dp dy\right|+\left|\iint_{|x-y|\leq t} \nabla_p b \cdot \left(K\partial_{x_\ell} f \right)|x-y|^{-2} \ dp dy\right|\\
	\leq&\
	\int_{|x-y|\leq t}
	\left|\int
	(\nabla_pb) f (t-|x-y|,y,p)\ dp\right|\left|\nabla_xK(t-|x-y|,y)\right|
	\ |x-y|^{-2}\varphi_L(|y|)\ dy\\
	&+
	\int_{|x-y|\leq t}
	\left|\int
	(\nabla_pb)\partial_{x_\ell} f (t-|x-y|,y,p)\ dp\right|\left|K(t-|x-y|,y)\right|
	\ |x-y|^{-2}\varphi_L(|y|)\ dy\\
	=:&\
	I+II.
	\end{align*}
The two resulting terms are then estimated using \eqref{rhoest} as well as, Lemmas \ref{Lm} (more specifically, Remark \ref{Lm_rmk}) and \ref{LGS_CV} to find
	\begin{align*}
	I
	\lesssim&\
	\epsilon_0 \Vert K \Vert_1\int_{|x-y|\leq t}
	\frac{\ln(t-|x-y|+|y|+2L)}{(t-|x-y|-|y|+2L)^{2}(t-|x-y|+|y|+2L)^{4}|x-y|^{2}}
	\ \varphi_L(|y|)\ dy\\
	\lesssim&\
	\frac{\epsilon_0}{r} \Vert K \Vert_1
	\int_{0}^t\int_a^b
	\frac{\ln(\tau+\lambda+2L)}{(\tau-\lambda+2L)^{2}(\tau+\lambda+2L)^{4}(t-\tau)}\
	\varphi_L(\lambda)\lambda\ d\lambda\ d\tau
	\end{align*}
and
	\begin{align*}
	II
	\lesssim&\
	\epsilon_0 \Vert K \Vert_0 \int_{|x-y|\leq t}
	\frac{1}{(t-|x-y|-|y|+2L)(t-|x-y|+|y|+2L)^{5}|x-y|^{2}}\
	\varphi_L(|y|)\ dy\\
	\lesssim&\
	\frac{\epsilon_0}{r} \Vert K \Vert_0
	\int_0^t\int_{a}^{b}
	\frac{1}{(\tau-\lambda+2L)(\tau+\lambda+2L)^5(t-\tau)}\ 
	 \varphi_L(\lambda)\lambda\ d\lambda\ d\tau\\	
	\lesssim&\
	\frac{\epsilon_0}{r} \Vert K \Vert_0
	\int_0^t\int_{a}^{b}
	\frac{1}{(\tau-\lambda+2L)^2(\tau+\lambda+2L)^4(t-\tau)}\ 
	 \varphi_L(\lambda)\lambda\ d\lambda\ d\tau.
	\end{align*}
Upon reducing the three-dimensional integrals to a two-dimensional representation via Lemma \ref{LGS_CV}, we have $\varphi_L(\lambda)=0$ for all $\lambda \geq \gamma \tau +L$ where $\gamma \in \left ( \frac{1}{2}, 1 \right )$.
Hence, combining these terms yields
\begin{equation}
\label{ASTTS_rep}
|A_{ST}+A_{TS}| \lesssim \frac{\epsilon_0}{r} \Vert K \Vert \mcA(t,r)
\end{equation}
where 
\begin{equation}
\label{A}
\mcA(t,r) := 
	\int_{0}^t\int_a^b
	\frac{\ln(\tau+\lambda+2L)}{(\tau-\lambda+2L)^{2}(\tau+\lambda+2L)^{4}(t-\tau)}\
	\varphi_L(\lambda)\lambda\ d\lambda\ d\tau.
\end{equation}
As we'll see in the next section, this term will appear again in the estimate of $A_{SS}$. To conclude this section, though, we merely estimate $\mcA(t,r)$.
We first use the upper bound on $\tau$ to estimate the logarithmic term and then employ $\tau+\lambda+2L\geq\tau+a+2L\geq t-r+2L$ to find
\begin{equation}
\label{A_rep}
	\mcA(t,r)
	\lesssim
	\frac{\ln(t+r+2L)}{t-r+2L}
	\int_{0}^t\int_a^b
	\frac{1}{(\tau-\lambda+2L)^2(\tau+\lambda+2L)^{3}(t-\tau)}\
	\varphi_L(\lambda) \lambda\ d\lambda\ d\tau.
\end{equation}
With this, we focus on estimating the remaining integral near and far from the singularity at $\tau = t$.
In particular, we break the integral at $\tau = \mu(t-r)$ into
$$ \int_{0}^t\int_a^b \frac{1}{(\tau-\lambda+2L)^2(\tau+\lambda+2L)^{3}(t-\tau)} \varphi_L(\lambda) \lambda\ d\lambda\ d\tau = I_A + I_B$$
where
$$I_A =  \int_{0}^{\mu(t-r)}\int_a^b \frac{1}{(\tau-\lambda+2L)^2(\tau+\lambda+2L)^{3}(t-\tau)} \varphi_L(\lambda) \lambda\ d\lambda\ d\tau$$
and
$$I_B =  \int_{\mu(t-r)}^t\int_a^b \frac{1}{(\tau-\lambda+2L)^2(\tau+\lambda+2L)^{3}(t-\tau)} \varphi_L(\lambda) \lambda\ d\lambda\ d\tau.$$

To estimate $I_A$, we consider $\tau \leq \mu(t-r)$. Then, $\mu(t-r) = 0$ implies $\tau = 0$. Hence, for $\tau \neq 0$ we have $\mu(t-r) = t-r$, and $\tau \leq \mu(t-r) = t-r$ yields $r \leq t-\tau$.
Thus, we conclude
$$\lambda \leq b = r + t - \tau \leq 2(t-\tau).$$
Similarly, as $\mu(t-r) \leq t$ we find
$$ I_A \leq 2 \int_{0}^t\int_a^b \frac{1}{(\tau-\lambda+2L)^2(\tau+\lambda+2L)^{3}} \varphi_L(\lambda)\ d\lambda\ d\tau.$$
Next, we change variables via $\alpha = \tau + \lambda$ and $\beta = \tau - \lambda$, using $\tau + b = t+r$ and $\tau + a \geq |t-r|$, as well as, 
$$\tau - \lambda \geq \tau - (L + \gamma \tau) = (1-\gamma) \tau - L \geq -L, $$ 
to give the correct limits of integration so that
$$ I_A  \lesssim\int_{|t-r|}^{t+r} \frac{d\alpha}{(\alpha+2L)^3} \int_{-L}^\infty \frac{d\beta}{(\beta+2L)^{2}} \lesssim\int_{|t-r|}^{t+r} \frac{d\alpha}{(\alpha+2L)^3}.$$
Finally, we perform the integration in $\alpha$ and use $|t-r| \geq t-r$ to find
\begin{align*}
I_A &\lesssim \left (|t-r| + 2L \right )^{-2} - \left (t+r+ 2L \right )^{-2} \\
&\lesssim r \left (t-r + 2L \right )^{-2} \left (t+r+ 2L \right )^{-2} (t+2L) \\
&\lesssim r \left (t-r + 2L \right )^{-2} \left (t+r+ 2L \right )^{-1}.
\end{align*}

Next, we estimate $I_B$ by considering $\tau \geq \mu(t-r)$. Then, using $\lambda \leq L + \gamma \tau$ and $\frac{1}{1-\gamma} \geq 2$, we find
$$(\tau - \lambda + 2L)^{-2} \leq \left [ (1-\gamma) \tau + L \right ]^{-2} = (1-\gamma)^{-2} \left [\tau + \frac{1}{1-\gamma} L \right ]^{-2} \leq C(\tau + 2L)^{-2}.$$
With this, we further employ the inequalities
$$\lambda \leq \gamma \tau + L \leq \tau + 2L$$
and $(\tau + \lambda + 2L)^{-3} \leq (\tau + 2L)^{-3}$
to yield
\begin{align*}
I_B & \lesssim \int_{\mu(t-r)}^t (\tau + 2L)^{-1} (t-\tau)^{-1} \int_a^b (\tau+\lambda+2L)^{-3} \varphi_L(\lambda)\ d\lambda\ d\tau\\
& \lesssim \int_{\mu(t-r)}^t (\tau + 2L)^{-4} (t-\tau)^{-1} (b-a) \ d\tau\\
& \lesssim \int_{\mu(t-r)}^t (\tau + 2L)^{-4} \ d\tau
\end{align*}
as $b-a = 2\min\{r, t-\tau \} \leq 2(t-\tau)$.
Finally, removing two powers, evaluating the remaining integral, and using $\mu(t-r) \geq t-r$ yields
\begin{align*}
I_B & \lesssim (\mu(t-r) + 2L)^{-2} \int_{\mu(t-r)}^t (\tau + 2L)^{-2} \ d\tau\\
& \lesssim (\mu(t-r) + 2L)^{-2} \left [ (\mu(t-r) + 2L)^{-1}  -  (t + 2L)^{-1}\right ]\\
& = (\mu(t-r) + 2L)^{-3} (t + 2L)^{-1} \left [ t -\mu(t-r) \right ]\\
& \lesssim r( t-r + 2L)^{-3} (t + r + 2L)^{-1}\\
& \lesssim r( t-r+ 2L)^{-2} (t + r + 2L)^{-1}
\end{align*}
as $t -r +2L \geq L$ is implied by the field support condition $r \leq t + L$.

Combining the estimates for $I_A$ and $I_B$ ultimately gives
$$ \int_{0}^t\int_a^b \frac{1}{(\tau-\lambda+2L)^2(\tau+\lambda+2L)^{3}(t-\tau)} \varphi_L(\lambda) \lambda\ d\lambda\ d\tau \lesssim r(t-r + 2L)^{-2} (t + r + 2L)^{-1},$$
and inserting this within \eqref{A_rep} yields
\begin{equation}
\label{A_est}
\mcA(t,r) \lesssim r \ln(t + r + 2L) (t-r + 2L)^{-3} (t + r + 2L)^{-1}.
\end{equation}
Finally, using \eqref{ASTTS_rep}, we obtain
\begin{equation}
\label{A_STTS}
|A_{ST}+A_{TS}| \lesssim \epsilon_0 \Vert K \Vert \ln(t + r + 2L) (t-r + 2L)^{-3} (t + r + 2L)^{-1}.
\end{equation}


\textbf{4. Estimate of $A_{SS}$.}

We follow the strategy of \cite{Glassey} where this term is broken up into several pieces, relying on the following representations of  the expression $S(Kf)$ and the operator $S^2$  (see  \cite[pp. 150-151]{Glassey})
	\begin{align*}
	S(Kf)
	&=
	-K\nabla_p\cdot(Kf)+fSK,\\
	S^2f
	&=
	-\nabla_p\cdot[S(Kf)]+\sum_{j,k=1}^3\frac{\delta_{jk}-v^j(p)\cdot v^k(p)}{p_0}(f\partial_{x_j} K^k+K^k\partial_{x_j} f),\\
	&=
	\nabla_p\cdot\left[K\nabla_p\cdot(Kf)-fSK\right]+\sum_{j,k=1}^3\tilde{c}_{jk}(p)(f\partial_{x_j} K^k+K^k\partial_{x_j} f),
	\end{align*}
where we recall $p_0=\sqrt{1+|p|^2}$ and have denoted $\tilde{c}_{jk}(p):=\frac{\delta_{jk}-v^j(p) v^k(p)}{p_0}$. We therefore have
	\begin{align*}
	A_{SS}
	=&
	\iint_{|x-y|\leq t} c(\omega, v(p)) S^2 (\partial_{x_\ell} f) |x-y|^{-1} \ dp dy
	=
	\iint_{|x-y|\leq t} c(\omega, v(p)) \partial_{x_\ell} (S^2 f) |x-y|^{-1} \ dp dy\\
	=&
	\iint_{|x-y|\leq t} c(\omega, v(p)) \partial_{x_\ell} \left(\nabla_p\cdot[K\nabla_p\cdot(Kf)-fSK]+\sum_{j,k=1}^3\tilde{c}_{jk}(p)(f\partial_{x_j} K^k+K^k\partial_{x_j} f)\right) |x-y|^{-1} \ dp dy\\
	=&
	\iint_{|x-y|\leq t} c(\omega, v(p)) \partial_{x_\ell} \left(\nabla_p\cdot[K\nabla_p\cdot(Kf)]\right) |x-y|^{-1} \ dp dy\\
	&+
	\iint_{|x-y|\leq t} c(\omega, v(p)) \partial_{x_\ell} \left(\nabla_p\cdot[-fSK]\right) |x-y|^{-1} \ dp dy\\
	&+
	\iint_{|x-y|\leq t} c(\omega, v(p)) \partial_{x_\ell} \left(\sum_{j,k=1}^3\tilde{c}_{jk}(p)(f\partial_{x_j} K^k)\right) |x-y|^{-1} \ dp dy\\
	&+
	\iint_{|x-y|\leq t} c(\omega, v(p)) \partial_{x_\ell} \left(\sum_{j,k=1}^3\tilde{c}_{jk}(p)(K^k\partial_{x_j} f)\right) |x-y|^{-1} \ dp dy=:I+II+III+IV.
	\end{align*}

We begin with $I$ and integrate by parts to find
	\begin{align*}
	|I|
	=&\
	\left|\iint_{|x-y|\leq t} c\  \partial_{x_\ell} \left(\nabla_p\cdot[K\nabla_p\cdot(Kf)]\right) |x-y|^{-1} \ dp dy\right|\\
	=&\
	\left|\iint_{|x-y|\leq t} \Big[\nabla_p(\nabla_pc \cdot\left(\partial_{x_\ell} K\right))\cdot(Kf)+\nabla_p(\nabla_pc\cdot K)\cdot \partial_{x_\ell}(Kf) \Big]|x-y|^{-1} \ dp dy\right|\\
	\leq&\
	\left|\iint_{|x-y|\leq t} \Big[\nabla_p(\nabla_pc \cdot\left(\partial_{x_\ell} K\right))\cdot(Kf)\Big]|x-y|^{-1} \ dp dy\right|\\
	&+\left|\iint_{|x-y|\leq t} \Big[\nabla_p(\nabla_pc\cdot K)\cdot \partial_{x_\ell}(Kf) \Big]|x-y|^{-1} \ dp dy\right|=:I_A+I_B.
   	\end{align*}
Estimating $I_A$ and using \eqref{rhoest} and Lemma \ref{LGS_CV} we find
	\begin{align*}
	I_A
	&\leq
	(\|\nabla_pc\|_\infty+\|\nabla_p^2c\|_\infty)\int_{|x-y|\leq t}\left(\int f\ dp\right) |\partial_{x_\ell} K||K|\ |x-y|^{-1}\ dy\\
	&\lesssim
	\frac{\epsilon_0}{r} \left ( \Vert K \Vert_0^2 + \Vert K \Vert_1^2 \right ) \int_0^t\int_{a}^{b}
	\frac{\ln(\tau+\lambda+2L)}{(\tau+\lambda+2L)^{5}(\tau-\lambda+2L)^{3}} \varphi_L(\lambda) \lambda 
	\ d\lambda d\tau
	 \end{align*}
as $\Vert K \Vert_0 \Vert K \Vert_1 \leq \frac{1}{2} \left (\Vert K \Vert_0^2 + \Vert K \Vert_1^2 \right )$.
Similarly, we use \eqref{rhoest} and Lemmas \ref{Lm} and \ref{LGS_CV} to estimate $I_B$ as
	\begin{align*}
	I_B
	&\leq
	(\|\nabla_pc\|_\infty+\|\nabla_p^2c\|_\infty)\int_{|x-y|\leq t}|K|\left[\left(\int f\ dp\right) |\partial_{x_\ell} K|+\left|\int \partial_{x_\ell} f\ dp\right||K| \right]|x-y|^{-1}\ dy\\
	&\lesssim
	\frac{\epsilon_0}{r}  \left (\Vert K \Vert_0^2 + \Vert K \Vert_1^2 \right ) \int_0^t\int_{a}^{b}
	\frac{\lambda}{(\tau+\lambda+2L)^2}
	\left[
	\frac{\ln(\tau+\lambda+2L)}{(\tau+\lambda+2L)^{3}(\tau-\lambda+2L)^{3}} \right . \\
	& \qquad  \left . +\frac{1}{(\tau+\lambda+2L)^{4}(\tau-\lambda+2L)^{2}}
	\right]
	\varphi_L(\lambda) \ d\lambda d\tau\\
	&\lesssim
	\frac{\epsilon_0}{r}  \left (\Vert K \Vert_0^2 + \Vert K \Vert_1^2 \right ) \int_0^t\int_{a}^{b}
	\frac{\ln(\tau+\lambda+2L)}{(\tau+\lambda+2L)^{5}(\tau-\lambda+2L)^{3}}
	\varphi_L(\lambda) \lambda \ d\lambda d\tau.
	\end{align*}
Hence, $I_A$ and $I_B$ satisfy the same estimate, allowing us to conclude
	\begin{align*}
	|I|
	&\lesssim
	\frac{\epsilon_0}{r}  \left (\Vert K \Vert_0^2 + \Vert K \Vert_1^2 \right ) \int_0^t\int_{a}^{b}
	\frac{\ln(\tau+\lambda+2L)}{(\tau+\lambda+2L)^{5}(\tau-\lambda+2L)^{3}}
	\varphi_L(\lambda) \lambda \ d\lambda d\tau,
	\end{align*}
	and thus
$$| I | \lesssim
	\frac{\epsilon_0}{r}  \left (\Vert K \Vert_0^2 + \Vert K \Vert_1^2 \right ) \int_0^t\int_{a}^{b}
	\frac{\ln(\tau+\lambda+2L)}{(\tau+\lambda+2L)^{4}(\tau-\lambda+2L)^{3}}
	\varphi_L(\lambda) \lambda \ d\lambda d\tau.$$
We define the remaining integral on the right side of the above estimate to be
\begin{equation}
\label{B}
\mcB(t,r) := \int_0^t\int_{a}^{b}
	\frac{\ln(\tau+\lambda+2L)}{(\tau+\lambda+2L)^{4}(\tau-\lambda+2L)^{3}}
	\varphi_L(\lambda) \lambda \ d\lambda d\tau
\end{equation}
so that
\begin{equation}
\label{ATTI}
| I | \lesssim \frac{\epsilon_0}{r}  \left (\Vert K \Vert_0^2 + \Vert K \Vert_1^2 \right ) \mcB(t,r). 
\end{equation}

As we shall see below, all other terms $II -IV$ will satisfy the same estimate as $I$, namely they will be bounded above by the right side of \eqref{ATTI}. Once we demonstrate this, we will then bound $\mcB(t,r)$ at the end of this section to complete the estimate for $A_{SS}$. 
We therefore turn our attention to  $II$. We note that this term includes expressions involving the operator $S=\partial_t+v(p)\cdot\nabla_x$ acting on the fields. The second part of this operator, namely the spatial derivative, is straightforward. 
To estimate time derivatives we use Maxwell's equations. Indeed, this yields
$$\Vert \partial_tE(t) \Vert_\infty \lesssim \Vert \nabla_x\times B(t) \Vert_\infty \lesssim \Vert \nabla_x K(t) \Vert_\infty $$ 
and from Lemma \ref{LCurrent}
$$\Vert \partial_tB(t) \Vert_\infty \lesssim \Vert \nabla_x\times E(t) \Vert_\infty+\Vert j(t) \Vert_\infty \lesssim \Vert \nabla_x K(t) \Vert_\infty+t^{-3}$$
with the latter term being higher order.
Using these bounds with \eqref{rhoest} and Lemmas \ref{Lm} and \ref{LGS_CV}, we find
	\begin{align*}
	| II |
	&=
	\left|\iint_{|x-y|\leq t} c(\omega, v(p)) \partial_{x_\ell} \left(\nabla_p\cdot[-fSK]\right) |x-y|^{-1} \ dp dy\right|\\
	&=
	\left|\iint_{|x-y|\leq t} \nabla_pc(\omega, v(p)) \partial_{x_\ell} \left(fSK\right) |x-y|^{-1} \ dp dy\right|\\
	&=
	\left|\iint_{|x-y|\leq t} \nabla_pc(\omega, v(p)) \left((SK)\partial_{x_\ell} f+fS\partial_{x_\ell} K\right) |x-y|^{-1} \ dp dy\right|\\
	&\leq
	\|\nabla_pc\|_\infty\int_{|x-y|\leq t}\left[\left|\int \partial_{x_\ell} f \ dp\right| \Vert \nabla_x K(t) \Vert_\infty +\left(\int f \ dp\right) \Vert \nabla_x^2 K(t) \Vert_\infty \right]|x-y|^{-1} \ dp dy\\
	&\lesssim
	\frac{\epsilon_0}{r} \Vert K \Vert  \int_0^t\int_{a}^{b}
	\left[
	\frac{\ln(\tau+\lambda+2L)}{(\tau+\lambda+2L)^{5}(\tau-\lambda+2L)^{2}}+\frac{\ln(\tau+\lambda+2L)}{(\tau+\lambda+2L)^{4}(\tau-\lambda+2L)^{3}}
	\right]
	\varphi_L(\lambda)\ \lambda\ d\lambda d\tau\\
	& \lesssim \frac{\epsilon_0}{r} \Vert K \Vert \mcB(t,r).
	\end{align*}

We next consider the term $III$, and obtain the same bound. For simplicity we denote $\hat{c}_{jk}(\omega,v(p))=c(\omega,v(p))\tilde{c}_{jk}(p)$.
Using \eqref{rhoest} and Lemmas \ref{Lm} and \ref{LGS_CV} as before, we arrive at
	\begin{align*}
	| III |
	&=
	\left|\iint_{|x-y|\leq t} c(\omega, v(p)) \partial_{x_\ell} \left(\sum_{j,k=1}^3\tilde{c}_{jk}(p)(f\partial_{x_j} K^k)\right) |x-y|^{-1} \ dp dy\right|\\
	&=
	\left|\iint_{|x-y|\leq t} \hat{c}(\omega,v(p)) \sum_{j,k=1}^3\left(\partial_{x_\ell} f\partial_{x_j} K^k+ f\partial_{x_j x_\ell} K^k\right) |x-y|^{-1} \ dp dy\right|\\
	&\leq
	\|\hat{c}\|_{\infty}\int_{|x-y|\leq t}  \sum_{j,k=1}^3\left[\left|\int\partial_{x_\ell} f\ dp\right|\left|\partial_{x_j} K^k\right|+\left( \int f\ dp\right)\left|\partial_{x_j x_\ell} K^k\right|\right] |x-y|^{-1} \ dy\\
	&\lesssim
	\frac{\epsilon_0}{r} \Vert K \Vert  \int_0^t\int_{a}^{b}
	\left[
	\frac{\ln(\tau+\lambda+2L)}{(\tau+\lambda+2L)^{5}(\tau-\lambda+2L)^{2}}+\frac{\ln(\tau+\lambda+2L)}{(\tau+\lambda+2L)^{4}(\tau-\lambda+2L)^{3}}
	\right]
	\varphi_L(\lambda) \lambda\ d\lambda d\tau\\
	&\lesssim \frac{\epsilon_0}{r} \Vert K \Vert
	\mcB(t,r).
	\end{align*}
Lastly, we estimate the term $IV$. With $\hat{c}$ defined as before, we have
	\begin{align*}
	IV
	&=
	\iint_{|x-y|\leq t} c(\omega, v(p)) \partial_{x_\ell} \left(\sum_{j,k=1}^3\tilde{c}_{jk}(p)(K^k\partial_{x_j} f)\right) |x-y|^{-1} \ dp dy\\
	&=
	\iint_{|x-y|\leq t} \hat{c}(\omega,v(p)) \sum_{j,k=1}^3\left(\partial_{x_\ell}  K^k\partial_{x_j}f + K^k\partial_{x_j x_\ell} f\right) |x-y|^{-1} \ dp dy=:IV_A+IV_B.
	\end{align*}

Estimating these two terms separately, we begin with $IV_A$ and use \eqref{rhoest} with Lemmas \ref{Lm} and \ref{LGS_CV} to find
	\begin{align*}
	|IV_A|
	&=
	\left|\iint_{|x-y|\leq t} \hat{c}(\omega,v(p)) \sum_{j,k=1}^3\partial_{x_\ell}  K^k\partial_{x_j}f  |x-y|^{-1} \ dp dy\right|\\
	&\leq
	\|\hat{c}\|_{\infty}\int_{|x-y|\leq t}  \sum_{j,k=1}^3\left|\int\partial_{x_j} f\ dp\right|\left|\partial_{x_\ell} K^k\right| |x-y|^{-1} \ dy\\
	&\lesssim
	\frac{\epsilon_0}{r} \Vert K \Vert_1 \int_0^t\int_{a}^{b}
	\frac{\ln(\tau+\lambda+2L)}{(\tau+\lambda+2L)^{5}(\tau-\lambda+2L)^{2}}
	\varphi_L(\lambda) \lambda \ d\lambda d\tau\\
	&\lesssim \frac{\epsilon_0}{r} \Vert K \Vert
	\mcB(t,r)
	\end{align*}
which, again, is the expected bound. 

Turning to $IV_B$, we face the difficulty of having to deal with two derivatives of $f$. This is overcome by expressing one of those derivatives in terms of the operators $S$ and $T_j$ (see \cite[p. 141]{Glassey}), with the portion that involves the operator $S$ being integrated by parts in $p$ so that
	\begin{align*}
	\int \partial_{x_j}f(t-|x-y|,y,p)\ dp
	&=
	\int \left[\frac{\omega_j S}{1+v\cdot\omega}+\sum_{i=1}^3\left(\delta_{ji}-\frac{\omega_j v_i}{1+v\cdot\omega}\right)T_i\right]f(t-|x-y|,y,p)\ dp\\
	&=
	\int \left[-\frac{\omega_j \nabla_p\cdot(Kf)}{1+v\cdot\omega}+\sum_{i=1}^3\left(\delta_{ji}-\frac{\omega_j v_i}{1+v\cdot\omega}\right)T_if(t-|x-y|,y,p)\right] dp.
	\end{align*}
Inserting this into $IV_B$, we integrate by parts in the $T_i$ derivative as outlined in \cite[p. 181]{Glassey}, which gives
	\begin{align*}
	|IV_B|
	=&\
	\Bigg|\iint_{|x-y|\leq t} \hat{c}(\omega,v(p)) \sum_{j,k=1}^3K^k\partial_{x_\ell}  \Bigg[-\frac{\omega_j \nabla_p\cdot(Kf)}{1+v\cdot\omega}+\sum_{i=1}^3\left(\delta_{ji}-\frac{\omega_j v_i}{1+v\cdot\omega}\right)T_if\Bigg]  |x-y|^{-1} \ dp dy\Bigg|\\
	\leq&\
	\Bigg|\iint_{|x-y|\leq t} \hat{c}(\omega,v(p)) \sum_{j,k=1}^3K^k  \frac{\omega_j\partial_{x_\ell} \nabla_p\cdot(Kf)}{1+v\cdot\omega}  |x-y|^{-1} \ dp dy\Bigg|\\
	&+
	\Bigg|\iint_{|x-y|\leq t} \hat{c}(\omega,v(p)) \sum_{i,j,k=1}^3K^k \left(\delta_{ji}-\frac{\omega_j v_i}{1+v\cdot\omega}\right)T_i\partial_{x_\ell}f |x-y|^{-1} \ dp dy\Bigg|\\
	\leq&\
	C\|\hat{c}\|_\infty\int_{|x-y|\leq t}\left(|K||\partial_{x_\ell} K|\int f\ dp+|K|^2\left|\int\partial_{x_\ell} f\ dp\right|\right)  |x-y|^{-1} \  dy\\
	&+
	C\|\hat{c}\|_\infty\left(\int_{|x-y|\leq t}\left|\int\partial_{x_\ell} f\ dp\right||\nabla_xK||x-y|^{-1} \ dy+\int_{|x-y|\leq t}\left|\int\partial_{x_\ell} f\ dp\right| |K| |x-y|^{-2} \ dy\right)\\
	&+
	C\|\hat{c}\|_\infty t^{-1}\iint_{|x-y|=t}\frac{|\omega|}{|1+v\cdot\omega|}|K(0,y)||\partial_{x_\ell} f_0(y,p)|\ dpdS_y.
	\end{align*}
The last term, which we denote $A_z'$, is estimated separately, along with the $A_z$ term, which is of a similar nature, while the first two terms are bounded using \eqref{rhoest} and Lemmas \ref{Lm} and \ref{LGS_CV} by
	\begin{align*}
	&\frac{\epsilon_0}{r} \left ( \Vert K \Vert + \Vert K \Vert_0^2 + \Vert K \Vert_1^2\right )
	\int_0^t\int_{a}^{b}
	\left[
	\frac{\ln(\tau+\lambda+2L)}{(\tau+\lambda+2L)^{5}(\tau-\lambda+2L)^{3}}+\frac{1}{(\tau+\lambda+2L)^{6}(\tau-\lambda+2L)^{2}(t-\tau)}
	\right]
	\varphi_L(\lambda) \lambda\ d\lambda d\tau.
	\end{align*}
The remaining integrals are bounded by the sum $\mcA(t,r) + \mcB(t,r)$ where $\mcA$ is defined by \eqref{A} and $\mcB$ is defined by \eqref{B}. 
Assembling the estimates for $I$, $II$, $III$, $IV_A$ and $IV_B$, we therefore find
\begin{equation}
\label{A_SS_rep}
|A_{SS}| \lesssim \frac{\epsilon_0}{r}  \left ( \Vert K \Vert + \Vert K \Vert_0^2 + \Vert K \Vert_1^2\right )\left ( \mcA(t,r) + \mcB(t,r) \right ).
\end{equation}

To conclude this section we focus on estimating the remaining double integral within $\mcB(t,r)$ given by \eqref{B}.
First, we use $\tau + \lambda \leq \tau + b = t + r$ to bound the logarithmic term so that
$$ \mcB(t,r)  \lesssim \ln(t+r+2L) \int_0^t\int_{a}^{b}
(\tau+\lambda+2L)^{-4}(\tau-\lambda+2L)^{-3}	
\varphi_L(\lambda) \lambda \ d\lambda d\tau.$$
Due to the bound on the support of $f$, we recall $\varphi_L(\lambda) = 0$ unless $\lambda \leq \gamma \tau + L$.
Hence, we find
$$\lambda \leq \gamma \tau + L \leq \tau + 2L$$
and
$$(\tau - \lambda + 2L)^{-3} \leq \left [ (1-\gamma) \tau + L \right ]^{-3} = (1-\gamma)^{-3} \left [\tau + \frac{1}{1-\gamma} L \right ]^{-3} \leq C(\tau + 2L)^{-3}.$$
Using these inequalities within the double integral gives
$$ \mcB(t,r)  \lesssim \ln(t+r+2L) \int_0^t\int_{a}^{b} (\tau+\lambda+2L)^{-4} (\tau + 2L)^{-2} \varphi_L(\lambda) \ d\lambda d\tau.$$
Next, we change variables via $\alpha = \tau + \lambda$ and $\beta = \tau$, using $\tau + b = t+r$ and $\tau + a  \geq |t-r|$ to change the limits of integration so that
\begin{align*}
\int_0^t\int_{a}^{b} (\tau+\lambda+2L)^{-4} (\tau + 2L)^{-2} \varphi_L(\lambda) \ d\lambda d\tau & \lesssim \int_{|t-r|}^{t+r} \frac{d\alpha}{(\alpha+2L)^4} \int_{0}^\infty \frac{d\beta}{(\beta+2L)^{2}}\\
& \lesssim (|t-r| + 2L)^{-1}\int_{|t-r|}^{t+r} \frac{d\alpha}{(\alpha+2L)^3}.
\end{align*}
Finally, we perform the integration and use $|t-r| \geq t-r$ to find
\begin{align*}
\int_0^t\int_{a}^{b} (\tau+\lambda+2L)^{-4} (\tau + 2L)^{-2} \varphi_L(\lambda) \ d\lambda d\tau  &\lesssim (|t-r| + 2L)^{-1}\left [ \left (|t-r| + 2L \right )^{-2} - \left (t+r+ 2L \right )^{-2} \right ] \\
&\lesssim r \left (t-r + 2L \right )^{-3} \left (t+r+ 2L \right )^{-2} (t+2L) \\
&\lesssim r \left (t-r + 2L \right )^{-3} \left (t+r+ 2L \right )^{-1}.
\end{align*}
Incorporating this estimate in the above inequality for $\mcB(t,r)$, we have shown
\begin{equation}
\label{B_est}
 \mcB(t,r)  \lesssim r \ln(t+r+2L)  \left (t-r + 2L \right )^{-3} \left (t+r+ 2L \right )^{-1}.
\end{equation}
Inserting \eqref{A_est} and \eqref{B_est} within \eqref{A_SS_rep}, we ultimately conclude
\begin{equation}
\label{A_SS}
|A_{SS}| \lesssim
\epsilon_0  \left ( \Vert K \Vert + \Vert K \Vert_0^2 + \Vert K \Vert_1^2\right ) \ln(t + r + 2L) (t-r + 2L)^{-3} (t + r + 2L)^{-1}.
\end{equation}

\textbf{5. Estimate of $A_{z}$ and $A'_{z}$.}

Finally, we estimate terms involving initial data, namely
        \begin{align*}
        A_z
        =&\
        (\partial_{x_k x_\ell} E^{*i})_0
        -
        \frac{1}{t^2}\iint_{|x-y|=t}d(\omega,v(p))\partial_{x_\ell}f_0(y,p)\ dp dS_y
        +
        \frac{1}{t}\iint_{|x-y|=t}e(\omega,v(p))(S\partial_{x_\ell} f)(0,y,p)\ dp dS_y
        \end{align*}
and
        \begin{align*}
        A_z'
        =
        C\|\hat{c}\|_\infty \frac1t\iint_{|x-y|=t}\frac{|\omega|}{|1+v\cdot\omega|}|K(0,y)||\partial_{x_\ell} f_0(y,p)|\ dpdS_y.
        \end{align*}

We start with the term $A_z$, which has one additional derivative than the comparable term appearing in \cite[eq. (6.33)]{Glassey}. The first term within $A_z$ is traced back to  \cite[eq. (6.27)-(6.28)]{Glassey} and is given by
        \begin{align*}
        (\partial_{x_k x_\ell}E^{*i})_0
        =&\
        \frac{1}{4\pi t^2}\int_{|x-y|=t}[\partial_{x_k x_\ell}E_0(y)+((y-x)\cdot\nabla)\partial_{x_k x_\ell}E_0(y)+t\nabla\times \partial_{x_k x_\ell}B_0(y)]\ dS_y\\
        &-
        \frac{1}{4\pi t}\iint_{|x-y|=t}v(p)\partial_{x_k x_\ell} f_0(y,p)\ dpdS_y\\
        &-
        \frac{1}{t}\iint_{|x-y|=t}\frac{\omega-(v\cdot\omega)v}{1+v\cdot\omega}\partial_{x_k x_\ell}f_0(y,p)\ dp dS_y
        \end{align*}

In evaluating the various terms we  use the fact that the initial data is supported in $\{|y|\leq L\}$ so that integrals over $|x-y|= t$ have support in $\{|t-r|\leq L\}=\{L\leq t-r+2L\leq 3L\}$. Points in this set satisfy, in particular, $3L(t-r+2L)^{-1}\geq 1$. Hence,  when bounding any of these integrals, we may use the fact that the support decays faster than $(t-r+2L)^{-N}$ for any power $N\in\N$. We shall take $N=3$ which is enough for our purposes. Together with the bound $t^{-1}\lesssim (t+r+2L)^{-1}$ from \eqref{eq:useful-t-bounds}, we conclude that  $A_z$ and $A_z'$ satisfy
        \begin{equation}
        \label{A_z}
        |A_z|+|A_z'|
        \lesssim
        \epsilon_0(t+r+2L)^{-1}(t-r+2L)^{-3},
        \end{equation}
where we use the fact that $S$ and $\partial_{x_\ell}$ commute and $Sf=-K\cdot\nabla_pf$ in the last term of $A_z$.

Combining all estimates, namely \eqref{A_omega}, \eqref{A_TT}, \eqref{A_STTS}, \eqref{A_SS}, and \eqref{A_z}, we ultimately find
$$| \partial_{x_k x_\ell} E^*(t,x)| \lesssim \epsilon_0\ln(t+r+2L) (t+r+2L)^{-1}(t-r+2L)^{-3} \left (1+ \Vert K \Vert + \Vert K \Vert_0^2 + \Vert K \Vert_1^2\right ).$$
For $t$ bounded the same inequality holds as
$$t + r + 2L \leq 2t + 3L \leq C,$$
which allows us to multiply by negative powers of $t-r+2L$ and $t + r + 2L$.
Similarly, we have
$$\ln(t + r + 2L ) \geq \ln(2L),$$
which allows us to multiply by the logarithmic term wherever necessary, as well.
Assembling large and small time estimates then gives
$$| \partial_{x_k x_\ell} E^*(t,x)| \leq C\epsilon_0\ln(t+r+2L) (t+r+2L)^{-1}(t-r+2L)^{-3} \left (1 + \Vert K \Vert + \Vert K \Vert_0^2 + \Vert K \Vert_1^2\right )$$
for all $t \geq 0$,
with an analogous estimate for second derivatives of $B$
so that
$$\Vert K^* \Vert_2 \leq C \epsilon_0 \left (1 + \Vert K \Vert + \Vert K \Vert_0^2 + \Vert K \Vert_1^2\right ),$$
which gives \eqref{K*2}.
With this estimate established, the standard iteration scheme, as well as the associated first and second derivatives, converges to a solution $(f,K)$ of \eqref{RVM} for $\epsilon, \epsilon_0$ sufficiently small. 
This completes the proof of Theorem \ref{T1} with the corresponding solution satisfying the estimates
$$\mcK_0(t) \lesssim t^{-2}, \qquad \mcK_1(t) \lesssim t^{-3}\ln(t) , \qquad \mathrm{and} \qquad \mcK_2(t) \lesssim t^{-4} \ln(t),$$ 
and hence all of the estimates in Sections $2-4$ are valid for the resulting solution.
\end{proof}

Finally, we prove the large time behavior theorems for arbitrary species by using the previous lemmas and applying them to each species as necessary.

\begin{proof}[Proof of Theorems \ref{T1} and \ref{T2}]
The first result follows by merely collecting the estimates of Sections $2-4$.
Indeed, as our assumptions on $(f^\alpha_0,E_0,B_0)$ satisfy the conditions of the global-in-time small data theorem (Theorem \ref{T0}), the initial data generate smooth solutions, and we conclude the field decay estimates
$$\mcK^\alpha_0(t) \lesssim t^{-2} \qquad \mathrm{and} \qquad \mcK^\alpha_1(t) \lesssim t^{-3}\ln(t)$$
for each $\alpha = 1, ...N$.
Then, we invoke Lemmas \ref{Xsupp} and \ref{Dvg} for every $\alpha = 1, ...N$ to produce the logarithmic growth of the spatial support of $g^\alpha$ and its momentum derivatives, namely
$$\mcR^\alpha(t) \lesssim \ln(t) \qquad \mathrm{and} \qquad \mcG^\alpha_1(t) \lesssim \ln^2(t).$$
Furthermore, Theorem \ref{T0} and Lemma \ref{D2g} guarantee that each $f^\alpha$ is $C^2$ and the resulting fields are $C^2$, as well, with the corresponding estimates
$$\mcK^\alpha_2(t) \lesssim t^{-4}\ln(t) \qquad \mathrm{and} \qquad \mcG^\alpha_2(t) \lesssim \ln^4(t).$$
With this, the spatial average $F^\alpha(t,p)$ of each species and its first derivatives converge with the estimates provided in Lemmas \ref{Funif} and \ref{DFunif}.
%
With the convergence of the spatial average, Lemmas \ref{LDensity} and \ref{LCurrent} provide the asymptotic behavior of the charge and current densities, while Lemmas \ref{LDensityderivative} and \ref{LCurrentderivative} yield the stated estimates on derivatives of the densities.
Lemma \ref{LField} then gives a refined estimate of the behavior of the fields and the Lorentz force of each species as $t \to \infty$.
The modified scattering result is provided by Lemma \ref{LModScattering} for each $\alpha = 1, ..., N$, which gives the stated convergence estimate and completes the proof of Theorem \ref{T1}.
Furthermore, if \alt{$\rho_\infty \not\equiv 0$ 
then} the provided estimates are sharp up to logarithmic corrections in the convergence rates.

Next, we consider \alt{$\rho_\infty \equiv 0$}. This immediately implies $j_\infty \equiv 0$, as well. Hence, the charge and current density estimates of Lemmas \ref{LDensity} and \ref{LCurrent}  yield
$$ \Vert \rho(t) \Vert_\infty + \Vert j(t) \Vert_\infty  \lesssim t^{-4} \ln^6(t).$$
Of course, as the limiting densities vanish, so too do their derivatives; thus, 
$\nabla_q\rho_\infty \equiv 0$ and $\nabla_q j_\infty \equiv 0$.
Hence, Lemmas \ref{LDensityderivative} and \ref{LCurrentderivative} give
$$ \Vert \nabla_x \rho(t) \Vert_\infty + \Vert \nabla_x j(t) \Vert_\infty + \Vert \partial_t j(t) \Vert_\infty \lesssim t^{-5}\ln^8(t).$$
Within Lemma \ref{LWave}, taking $\eta_\infty \equiv 0$ yields $\psi_\infty \equiv 0$.
Hence, we find $E_\infty \equiv B_\infty \equiv 0$.
Lemma \ref{LField} then gives
$$\sup_{|x| \lesssim \gamma t} \bigg (|E(t,x)| + |B(t,x)| \bigg ) \lesssim t^{-3} \ln^8(t),$$
and $\mcK^\alpha_0(t) \lesssim t^{-3}\ln^8(t)$ for all $\alpha = 1, ..., N$.
Lemma \ref{L6} yields the improved decay estimate
$$|\mcP^\alpha(t, \tau, x, p) - \mcP^\alpha_\infty(\tau, x, p) | \lesssim t^{-2} \ln^8(t)$$
for every $\alpha = 1, ..., N$,
which provides the stated decay rates.
We note that Lemma \ref{Xsupp} further yields
$$ \sup_{(x,p) \in \mcS_g^\alpha(0)} |\mcY^\alpha(t,0,x,p)| + \mcR^\alpha(t)  \lesssim 1.$$
Hence, the modification to the spatial trajectories is no longer needed in order for the particle distribution to scatter to its ``free-streaming'' profile.
However, to precisely establish such a claim, a refined estimate of field derivatives is needed (see Remark \ref{rem:1}).

%
%
%
%

\end{proof}

\section{\alt{The Non-Relativistic Vlasov-Maxwell System}}
\label{sec:other-models}
\alt{
As noted in \cite[Theorem 4]{GS}, small data results of this variety for the \emph{relativistic} Vlasov-Maxwell system are also applicable to the \emph{non-relativistic} Vlasov-Maxwell (VM) system where the relativistic velocity ${v}_\alpha(p)= \frac{p}{\sqrt{m_\alpha^2 + |p|^2/c^2}}$ is replaced by the classical velocity ${v}_\alpha(p)=p/m_\alpha$ in \eqref{Vlasov}, \eqref{Rel_vel} and \eqref{rhoj}. Essentially, one only needs to pay attention to the velocity support of the initial particle configuration to ensure that there are no particles traveling near the speed of light. The analogue of our main theorems to (VM) can then be stated in a similar manner. As a reminder to the reader, we have taken the speed of light $c=1$ within the statements of our main theorems and will do so below, as well. Hence, the constraint $L_p < 1$ and the bound $\beta < 1$, which appear in Theorem \ref{VM}, merely arise from the normalization of $c$.
\begin{theorem}
\label{VM}
For any $L_x > 0$ and $0<L_p<1$ there exist constants $\epsilon_0 > 0$ and $\beta \in (0,1)$ with the following property.
Let $f^\alpha_0 \in C^2$, $\alpha=1,\dots,N$, be non-negative functions supported on $\overline{\Gamma}_{L_x} \times \overline{\Gamma}_{L_p}$.
Let $E_0, B_0 \in C^3$ be supported on $\overline{\Gamma}_{L_x}$ satisfying the compatibility conditions \eqref{compat}.
If the initial data satisfy
$$\sum_{\alpha = 1}^N \Vert f^\alpha_0 \Vert_{C^2} + \Vert E_0 \Vert_{C^3} + \Vert B_0 \Vert_{C^3} \leq \epsilon_0,$$
then there exists a unique classical solution of (VM) for all $x,p\in\R^3$ and $t\geq0$ with the given initial data such that
$$f^\alpha(t,x,p) = 0 \qquad \mathrm{for} \qquad |p| \geq \beta$$
for all $\alpha = 1, ..., N$, $t\geq 0$, and $x \in \bfR^3$. This solution satisfies all the estimates appearing in Theorems \ref{T0}, \ref{T1} and \ref{T2} with the following modifications:
\begin{enumerate}
	\item
	 The domain of $v_\alpha^{-1}(q) = m_\alpha q$ is $q\in\R^3$, rather than $q \in \Gamma_1$. Subsequently, this is also the domain of $\rho_\infty$, $j_\infty$, $E_\infty$, and $B_\infty$.
	\item
	$\mathbb{A}_\alpha(p)=\nabla v_\alpha(p)$ is a constant matrix  defined by $\mathbb{A}_\alpha(p) = m_\alpha^{-1} \mathbb{I}$ so that $\mcD_\alpha(p)=m_\alpha^3$.
	\item
	$\mathbb{B}_\alpha(q)=\nabla v_\alpha^{-1}(q)$ is a constant matrix defined by $\mathbb{B}_\alpha(q) = m_\alpha \mathbb{I}$.
	\item
	$L$ is replaced by $L_x$ in the estimates appearing in Theorem \ref{T0}.
\end{enumerate}
\end{theorem}
}
\alt{
\begin{proof}[Remarks on the  Proof]
In \cite[Section 7]{GS}, the authors sketch the argument for the extension of Theorem \ref{GS} to the non-relativistic system. Those same ideas immediately apply to the proof of Theorem \ref{T0}; hence, they are not repeated here. Nearly all of the additional elements of the proofs remain unchanged, so we merely point out a few minor differences, upon dropping the multiple species notation as before. First, as $v(p)$ is linear, Lemma \ref{lem:v} becomes essentially trivial. This is also the case in all instances where $\nabla v$ and $\nabla v^{-1}$ appear, such as in the proof of Lemma \ref{LH}.
Additionally, within Lemma  \ref{LPrelim}, we can merely take $\zeta=\beta$.
\end{proof}
}



\begin{thebibliography}{}
%
%

\bibitem{BD} Bardos, C. and Degond, P., Global existence for the {V}lasov-{P}oisson equation in {$3$} space variables with small initial data. Ann. Inst. H. Poincar\'e Anal. Non Lin\'eaire {\bf 1985}, 2(2): 101-118.

\bibitem{BKR} Batt, J., Kunze, M., and Rein, G., On the asymptotic behavior of a one-dimensional, monocharged plasma and a rescaling method. Advances in Differential Equations {\bf 1998}, 3: 271-292.



\bibitem{BCP1} Ben-Artzi, J., Calogero, S., and Pankavich, S., Arbitrarily large solutions of the Vlasov-Poisson system. SIAM J. Math. Anal. {\bf 2018}, 50(4): 4311-4326.

\bibitem{BCP2} Ben-Artzi, J., Calogero, S., and Pankavich, S., Concentrating solutions of the relativistic Vlasov-Maxwell system. Commun. Math. Sci. {\bf 2019}, 17(2): 377-392.

\bibitem{BMP} Ben-Artzi, J., Morisse, B., and Pankavich, S., Asymptotic Growth and Decay of Two-Dimensional Symmetric Plasmas, Kinetic and Related Models {\bf 2024}, 17(1): 29-51.

\bibitem{Bigorgne} Bigorgne, L., Global existence and modified scattering for the small data solutions to the Vlasov-Maxwell system, https://arxiv.org/abs/2208.08360.

\bibitem{BigorgneRVM} Bigorgne, L., Sharp Asymptotic Behavior of Solutions of the 3d Vlasov-Maxwell System with Small Data, Comm. Math. Phys. {\bf 2020}, 376, 893-992.


\bibitem{Evans} Evans, L.C., Partial differential equations, Providence, R.I.: American Mathematical Society {\bf 2010}.

\bibitem{Pausader} Flynn, P., Ouyang, Z., Pausader, B. et al., Scattering Map for the Vlasov-Poisson System, Peking Math J {\bf 2021}.

\bibitem{GT} Gilbarg, D., and Trudinger, N. S., Elliptic partial differential equations of second order. Berlin: Springer-Verlag {\bf 1983}.

\bibitem{Glassey} Glassey, R., The Cauchy Problem in Kinetic Theory.  SIAM: {\bf 1996}.

\bibitem{GPS}  Glassey, R., Pankavich, S., and Schaeffer, J., Decay in Time for a One-Dimensional, Two Component Plasma. Math. Meth Appl. Sci. {\bf 2008}, 31: 2115-2132.

\bibitem{GPS2} Glassey, R., Pankavich, S., and Schaeffer, J., On long-time behavior of monocharged and neutral plasma in one and one-half dimensions. Kinetic and Related Models {\bf 2009}, 2: 465-488.

\bibitem{GPS4} Glassey, R., Pankavich, S., and Schaeffer, J., Time Decay for Solutions to the One-dimensional Equations of Plasma Dynamics, Quarterly of Applied Mathematics {\bf 2010}, 68: 135-141.

\bibitem{GPS5} Glassey, R., Pankavich, S., and Schaeffer, J., Large Time Behavior of the Relativistic Vlasov-Maxwell System in Low Space Dimension, Differential \& Integral Equations {\bf 2010}, 23: 61-77.

\bibitem{GSchNN} Glassey, R.T., Schaeffer, J., Global existence for the relativistic Vlasov-Maxwell system with nearly neutral initial data. Comm. Math. Phys. {\bf 1988}, 119: 353-384.

\bibitem{GS2} Glassey, R., Strauss, W., Singularity formation in a collisionless plasma could occur only at high velocities. Archive for rational mechanics and analysis {\bf 1986}, 92, 59-90.

\bibitem{GS} Glassey, R., Strauss, W., Absence of shocks in an initially dilute collisionless plasma. Comm. Math. Phys. {\bf 1987}, 113, 191-208.


\bibitem{Horst} Horst, E., Symmetric plasmas and their decay. Comm. Math. Phys. {\bf 1990}, 126: 613-633.

\bibitem{HRV} Hwang, H., Rendall, A., and Velazquez, J., Optimal gradient estimates and asymptotic behaviour for the Vlasov-Poisson system with small initial data. Archive for rational mechanics and analysis {\bf 2011}, 200: 313-360.


\bibitem{Ionescu} Ionescu, A., Pausader, B., Wang, X., Widmayer, K., On the asymptotic behavior of solutions to the Vlasov-Poisson system, International Mathematics Research Notices {\bf 2022}, 12: 8865-8889.

\bibitem{Kunze} Kunze, M., Yet another criterion for global existence in the 3D relativistic Vlasov-Maxwell system. J. Differential Equations {\bf 2015}, 259(9): 4413-4442.


\bibitem{LS} Luk, J. and Strain, R., A new continuation criterion for the relativistic Vlasov-Maxwell system. Comm. Math. Phys. {\bf 2014}, 331(3):1005-1027.

\bibitem{LS2} Luk, J. and Strain, R., Strichartz estimates and moment bounds for the relativistic Vlasov-Maxwell system. Arch. Ration. Mech. Anal. {\bf 2016}, 219(1):445-552.


\bibitem{Pankavich2020} Pankavich, S., Exact Large Time Behavior of Spherically-Symmetric Plasmas, SIAM J. Math. Anal., {\bf 2021}, 53(4): 4474-4512.

\bibitem{Pankavich2021} Pankavich, S., Asymptotic Dynamics of Dispersive Plasmas, Communications in Mathematical Physics {\bf 2022}, 391: 455-493.

\bibitem{Pankavich2022} Pankavich, S., Scattering and Asymptotic Behavior of Solutions to the Vlasov-Poisson System in High Dimension, SIAM J. Math. Anal., {\bf 2023}, 55(5): 4727-4750.

\bibitem{Patel} Patel N., Three new results on continuation criteria for the 3D relativistic Vlasov-Maxwell system. J. Differential Equations {\bf 2018}, 264(3):1841-1885.


\bibitem{Pfaf} Pfaffelmoser, K., Global classical solution of the Vlasov-Poisson system in three dimensions for general initial data. J. Diff. Eq.  {\bf 1992}, 95(2): 281-303.


\bibitem{Sch} Schaeffer, J., Large-time behavior of a one-dimensional monocharged plasma. Diff. and Int. Equations {\bf 2007}, 20(3): 277-292.

\bibitem{SchRVM} Schaeffer, J., A small data theorem for collisionless plasma that includes high velocity particles. Indiana Univ. Math. J. {\bf 2004}, 53(1): 1-34.

\bibitem{Smulevici} Smulevici, J., Small data solutions of the Vlasov-Poisson system and the vector field method.  Ann. PDE {\bf 2016} 2(2).

\bibitem{Wang} Wang, X., Propagation of regularity and long time behavior of the 3D massive relativistic transport equation II: Vlasov-Maxwell system. Comm. Math. Phys. {\bf 2022}, 389(2): 715-812.

\bibitem{WY} Wei, D. and Yang, S., On the 3D relativistic Vlasov-Maxwell system with large Maxwell field. Comm. Math. Phys., {\bf 2021} 383(3): 2275-2307.

\end{thebibliography}
\end{document}